\documentclass[11pt,a4paper,reqno]{amsart}
 
\usepackage[margin=1in]{geometry}

\usepackage[english]{babel}
\usepackage[utf8]{inputenc}

\usepackage{bbm,amsmath,amsthm,epsfig,latexsym,marvosym,mathrsfs}
\usepackage{amsfonts}
\usepackage{bigints}
\usepackage{amssymb}
\usepackage{color}
\usepackage{graphics}
\usepackage[position=b]{subcaption}

\usepackage{mathtools}
\usepackage{microtype}
\usepackage{newtxtext}
\usepackage{hyperref}
\graphicspath{{./}{figures/}}

\def\R{\mathbb{R}}
\def\N{\mathbb{N}}
\def\F{{F}}
\def\H{\mathcal{H}}
\def\L{\mathcal{L}}

\def\d{\mathrm{d}}

\def\ca{\mathbbmss{1}}

\def\e{\varepsilon}

\DeclareMathOperator*{\esssup}{ess\,\,sup}
\DeclareMathOperator*{\Glim}{\Gamma-\lim}
\DeclareMathOperator*{\Var}{Var}
\DeclareMathOperator*{\Lip}{Lip}
\DeclareMathOperator*{\grad}{grad}

\newcommand{\zz}{z}
\newcommand{\tmu}{\tilde \mu}
\newcommand{\M}{\mathcal{M}}

\newcommand{\te}{\textrm}

\definecolor{dred}{rgb}{0.75,0,0}
\definecolor{dgreen}{rgb}{0.01,0.5,0.10}
\definecolor{darkblue}{rgb}{0.1,0.1,0.6}

\definecolor{mygreen}{rgb}{0.1,0.75,0.2}

\newcommand{\nc}{\normalcolor}

\newlength\figureheight
\newlength\figurewidth

\newtheorem{theorem}{Theorem}[section]
\newtheorem{corollary}[theorem]{Corollary}
\newtheorem{proposition}[theorem]{Proposition}
\newtheorem{lemma}[theorem]{Lemma}

\theoremstyle{definition}
\newtheorem{remark}[theorem]{Remark}

\numberwithin{equation}{section}
\numberwithin{figure}{section}

\author{Marco Caroccia}
\address{Scuola Normale Superiore, Piazza dei Cavalieri, 7, 56126 Pisa PI, Italy  and Università degli studi di Firenze,  Dipartimento di Matematica e Informatica "Ulisse Dini", Viale Giovanni Battista Morgagni, 67/A, 50134 Firenze FI, Italy.}
\email{marco.caroccia@sns.it;
marco.caroccia@unifi.it}
\author{Antonin Chambolle}
\address{CMAP, CNRS and \'Ecole Polytechnique 91128 Palaiseau, France}
\email{chambolle.antonin@polytechnique.fr}

\author{Dejan Slep\v{c}ev}
\address{Department of Mathematical Sciences, Carnegie Mellon University, Pittsburgh, PA 15213, USA}
\email{slepcev@math.cmu.edu}

\makeindex
\title{Mumford -- Shah functionals on graphs and their asymptotics}

\begin{document}

\begin{abstract}
We consider adaptations of the Mumford-Shah functional to graphs. 
These are based on discretizations of nonlocal approximations to the Mumford-Shah functional. 
Motivated by applications in machine learning we study the random geometric graphs associated to random samples of a measure. We establish the conditions on the graph constructions under which the minimizers of graph Mumford-Shah functionals converge to a minimizer of a continuum Mumford-Shah functional. Furthermore we explicitly identify the limiting functional. Moreover we describe an efficient algorithm for computing the approximate minimizers of the graph Mumford-Shah functional. 
\end{abstract}

\maketitle
\tableofcontents


% 	49J55  	Problems involving randomness [See also 93E20]
% 	49J45  	Methods involving semicontinuity and convergence; relaxation
% 	68R10  	Discrete mathematics in relation to computer science: Graph theory
% 	62G20  	Nonparametric inference: Asymptotic properties
%   60D05   Geometric probability and stochastic geometry
%  	35J20   Variational methods for second-order elliptic equations
% 	65N12   Stability and convergence of numerical methods

\section{Introduction}
Our investigation of graph based Mumford-Shah functionals is motivated by problems arising in machine learning. Given a point cloud in euclidean space with (noisy) real-valued  labels, or a graph with labeled vertices, we investigate a model
to denoise the labels while allowing for jumps (discontinuities) in label values. As with the classical Mumford-Shah functional this allows one to   identify  the locations of sharp transitions of label values. %The approach is flexible and can be used to assign label values to unlabeled points, which makes it applicable to problems of classification and prediction.  
% Mathematically, one goal we consider is to segment the data into regions where labels change smoothly and transition layers where they change abruptly. A further goal is to denoise the label values in a way that preserves the sharp transitions.
Our primary focus is on graphs arising as neighborhood graphs of point clouds in a euclidean space, in dimension two or higher, where we can carry out rigorous analysis.  However some of the functionals we study can be formulated purely in the setting of weighted graphs and may be useful in applications. 

The model we study is based on ideas from image processing and go back to the celebrated  Mumford and Shah \cite{mumford1989optimal}  variational model for image segmentation.  To adapt the Mumford--Shah functional to point clouds and graphs we rely on the work of Gobbino  \cite{gobbino1998finite} and Gobbino and Mora \cite{gobbino2001finite}  who introduced a family of nonlocal models which approximate the Mumford--Shah functional. Ruf  \cite{Ruf} has recently adapted these nonlocal models to random discrete setting and studied them in the setting of stochastic lattices. Here we study such functionals in the setting of random geometric, and related, graphs relevant to machine learning.

 \medskip
 
 \emph{General graph setting.} Consider an undirected weighted graph with vertices  $V= \{1, \dots. n\}$ and edge  weights matrix $W = [w_{ij}]_{i,j=1, \dots,n}$. Edge weights are considered to be nonnegative and symmetric. 
 Let $f:V \to \R$ be the observed noisy labels. Let 
 $\zeta:[0,\infty) \rightarrow [0,\infty)$ be concave and such that $\zeta(0)=0$, $0< \zeta'(0)<\infty$.
For $p>1$ we define the \emph{Graph Mumford--Shah} functional acting on $u:V \to \R$ as 
\begin{equation}\label{eq:GMSf}
	\mathcal{GMS}_f(u):=  \frac{\lambda}{n} \sum_{i=1}^n |u_i - f_i|^2 + \frac{1}{\e n^2} \sum_{i,j=1}^n \zeta\left( \frac{1}{\e} |u_i - u_j|^2 \right) w_{ij}
\end{equation}

We note that when the differences $u_i-u_j$ are relatively small the functional is similar to the graph dirichlet energy, while for large values of  $u_i-u_j$  the functional saturates and in some ways considers $u$ to be discontinuous over the edge. It then just penalizes the size of the set of discontinuities.  
Minimizing the functional  allows one to find the sharp transitions in the data by detecting edges  where $u_i-u_j$  is large compared to $\e$. That is the parameter $\e>0$ sets the scale for what differences of the values are considered ``large''.
%about what smoothness of $u$ is to be expected. 
We note that the functional is nonconvex. 
\medskip
 
 \emph{Geometric graph setting.} We now consider the setting of point clouds and the random geometric graphs generated by them. The ability to measure the distance between vertices allow us to create a larger family of graph Mumford--Shah functionals. 
 Let $V_n = \{ x_1, \dots, x_n\}$ be a set of points in $\R^d$. The points $x_i$ are typically random samples of a measure describing the data distribution, but this interpretation is not essential in defining the functional. Given these points we define a graph by setting the edge weights to be $w_{i,i} = 0$ and for $i \neq j$
 \begin{equation} \label{eq:edgew}
w_{ij} = \eta_{\e}(|x_i - x_j|)
\end{equation}
 where $\eta$ is a nonnegative, nonincreasing function which decays to $0$ faster than a specified algebraic rate. 
  Let $f:V_n \to \R$ be the observed noisy labels and let $\zeta$ be as in the graph setting above. 
 For $p \in [1,d)$ 
 and $q \in [0,p-1]$ we define the \emph{Graph Mumford--Shah} functional acting on $u:V_n \to \R$ as 
\begin{equation}\label{eq:GMSfen}
	\mathcal{GMS}_{f,\e,n}(u):= \frac{\lambda}{n} \sum_{i=1}^n |u(x_i) - f_i|^2 + \frac{1}{\e} \frac{1}{n^2}\sum_{i,j=1}^n \zeta\left( \e^{1-p+q} \frac{|u(x_i) - u(x_j)|^p}{|x_i-x_j|^q}\right) \eta_{\e}(|x_i - x_j|).
	\end{equation}
We note that taking $q=0$ reduces this functional to one considered in the pure graph setting. 
\medskip

We rigorously study of the asymptotics of $\mathcal{GMS}_{\e,n}(u)$ as $n \to \infty$ and $\e \to 0$ and establish in Theorem \ref{thm main: G conv} that its minimizers converge to 
to minimizers of a  Mumford--Shah functional posed in continuum euclidean domain.  
We note that related results for a stochastic lattice model have been obtained recently by Ruf \cite{Ruf}, see Remark 
\ref{rem:ruf}. 
The conditions of the Theorem \ref{thm main: G conv}  are optimal in terms of scaling of $\e_n$ on $n$ for which the convergence holds for all dimensions $d \geq 2$. To show the result we follow the general strategy of \cite{GTS16} and use a number of results of calculus of variations, in particular the works of Gobbino and Mora \cite{gobbino1998finite,gobbino2001finite}. 
There are two notable advances:
\begin{itemize}
\item[I.] We introduce a strategy to overcome the issues that arise from the lack of control of the denominator in \eqref{eq:GMSfen}. Namely the discrepancy in the quotients inside of $\zeta$ may seem large if the standard tools to compare the discrete and continuum functionals using a transport map are used directly. In Remark \ref{rem:drill} we outline the steps we take to account for that.
\item[II.] Unlike in \cite{GTS16}, our results have optimal scaling in 2D. Using Lemma \ref{lem: improved convergence} and Lemma \ref{lem:improved convergence TLp} we develop an approach to $\Gamma$-convergence that uses a more relaxed way to compare the discrete and continuum measures. In particular the approach outlined at the beginning of Subsection \ref{sec:aux} would allow one to obtain optimal estimates for total variation, Laplacian, andf $p$-Laplacian functionals considered in \cite{GTS16}, \cite{GTSspectral} and \cite{SleTho17plap} respectively. We note that for the graph total variation optimal estimates in 2D were recently obtained by M\"uller and Penrose \cite{muller2018optimal}. The approach here is simpler, but does use the insight of M\"uller and Penrose that binning at an intermediate scale can be advantageous. 
\end{itemize}

\emph{Organization.} In Subsection \ref{sec:related} we review the works on related problems, primarily on the mathematical aspects of related data science questions. In Section \ref{sec:setting} we introduce the graph based and the continuum functionals and state the main results. 
in Section \ref{sec:prelim} we recall the mathematical notion of $\Gamma$-convergence and its main properties and we recall  the $TL^p$ space and its main properties. We introduce the relaxed way to compare measures with the empirical measures of their samples.  In Section \ref{sec:gamma} we prove the main results on $\Gamma$-convergence, while in Section \ref{sec: compactness} we prove the accompanying compactness result. In Section \ref{sec:numerics} we describe an algorithm for computing the approximate (local) minimizers of the graph Mumford--Shah functional and perform numerical experiments on synthetic data to showcase its properties and on real-estate sales data to highlight its applicability in prediction problems. In Appendices \ref{sec:app} and \ref{sec:app2} we prove two technical results needed in Section \ref{sec:gamma}.

\subsection{Related works} \label{sec:related}
%\subsubsection{Related models in data analysis.}
Here we review the related models in data analysis. The background about the Mumford-Shah functional and  has been provided in the introduction and the mathematical works which serve as the basis for our  proofs are recalled as we present our approach  in the introduction and Sections \ref{sec:prelim} and \ref{sec:gamma}.

Regularizing and denoising functions given on graphs has been studied in variety of contexts in machine learning. % ?? give some references. 
Here we focus on regularizations which still allow for the jumps in the regularized function. There are two lines of research which have led to such functionals. One, as is the case with our approach, is in using inspiration from image processing where variational approaches have been widely used for image denoising and segmentation. Particularly relevant in the context of imaging are the works of Chan and Vese  \cite{chan2001, vese2002}, who proposed a piecewise constant simplification of the Mumford-Shah functional and have shown its effectiveness in image segmentation, and Rudin, Osher, and Fatemi \cite{ROF} who proposed a TV (total variation) based regularization for the image denoising. 
%There are several papers which study related functionals on graphs. Namely functionals which include 
%both terms which reward regularity of the estimator, allow for jumps, penalize jumps by their (weighted) perimeter and reward fidelity with observed data. 
 In analogy with Chan and Vese,  \cite{hu2013} Hu, Sunu, and Bertozzi formulated the piecewise-constant Mumford functional on graphs. They also  developed an efficient numerical approach to compute the minimizers and used it to study a (multi-class) classification problem. 
 %The piecewise-constant Mumford-Shah functional plays an important role in image segmentation. This approach was developed by Chan and Vase \cite{chan2001, vese2002} and is sometimes referred to as the Chan-Vese model. In \cite{hu2013} this model was adapted to graphs.
A ROF functional on graphs, with $L^1$ fidelity term, was studied by Garc\'ia Trillos and Murray \cite{GTM17}. 

TV based regularizations have also been developed in statistics community. Mammen and van de Geer \cite{MamvdG97} have considered it is the setting of nonparametric regression and have shown that the TV regularization provides an estimator that achieves the optimal min-max recovery rate in one dimension over noisy samples of functions in unit  ball with respect to the BV norm. TV based regularizations in hifher dimensions have been considered by 
 by Tibshirani, Saunders, Rosset, Zhu, and Knight \cite{tibshirani_fused} who call the functional fused LASSO. 
H\"utter and Rigollet \cite{HutRig16}, show that, up to logarithms, in dimension $d \geq 2$, the TV regularization on grids achieves the optimal min-max rate over the unit  ball with respect to the BV norm.
Recently, 
Padilla, Sharpnack, Chen and Witten \cite{PSCW18} show for random. geometric graphs and for KNN graphs that  up to logarithms, in dimension $d \geq 2$, TV regularization again achieves the optimal min-max rate.

The paper \cite{hallac2015} by Hallac,  Leskovec,  Boyd extends fused LASSO to the graph setting and considers some further functionals which are closely related to the graph Mumford--Shah functional we consider here. In particular the initial models of the paper deal with convex functionals which include graph total-variation based terms, and are thus called ``Network LASSO''. The second part of the paper modifies the total-variation term, which leads to nonconvex functionals. 
Here we interpret some of these nonconvex functionals, in particular model (7) of \cite{hallac2015}), as the  graph-based Mumford--Shah functional, which, together with out asymptotic results, explains the behavior of these models.
Wang, Sharpnack, Smola, and  Tibshirani \cite{WSST_trend}  consider higher order total variation regularizers on graphs. further extensions. %In particular they consider functionals which include higher order difference operators on graphs and their $L^1$ norms as regularizers. In addition to introducing the model, highlighting its features and numerical experiments, the authors obtain error bounds under the assumption of normally distributed error in label values. 
We also note that the use of total variation penalization for signal denoising and filtering has also been considered in the signal processing community, see for example the work of Chen, Sandryhaila, Moura, and Kova\v{c}evi\'c \cite{chen_sand_moura_SRG}.

%
%\grn
%Min-max rates for regression: 
%In \cite{SaWaTi16} it states that the Min-Max error rate (measured in $L^2$ norm)  for 1D regression for functions in BV ball is
%\[ \mathcal R \sim n^{-2/3} \]
%using the wavelet estimator, and that the min-max rate over linear estimators is
%\[ \mathcal R_{lin} \sim n^{-1/2}. \]
%Both of the results were shown in  \cite{DonJohn98}. 
%Mammen and van de Geer  \cite{MamvdG97} showed that the ROF estimator in 1D achieves the optimal min-max rate. they call this locally adaptive regression splines. 
%
%H\"utter and Rigollet \cite{HutRig16}, page 11, show that, up to logarithms, in dimension $d \geq 2$, ROF on grid achieves the optimal min-max rate of 
%\[ \mathcal R \sim n^{-1/d}. \]
%Note that their $\lambda$ is one over our $\lambda$,. The rate in \cite{HutRig16} is for our $\lambda  \sim
%n / \ln n$ when $d \geq 3$. 
%
%
%Padilla, Sharpnack, Chen and Witten \cite{PSCW18} show for random. geometric graphs and for KNN graphs that  up to logarithms, in dimension $d \geq 2$, ROF on grid achieves the optimal min-max rate of 
%\[ \mathcal R \sim n^{-1/d}. \]
%under a number fo technical assumptions. 
%\nc

\section{Setting and Main Results} \label{sec:setting}

\subsection{Continuum Mumford--Shah functional and its nonlocal approximation}
In their celebrated paper \cite{mumford1989optimal}, Mumford and Shah proposed a variational approach for image segmentation. Given a domain $\Omega\subset \R^d$ and a potentially  noisy image with intensity $f$ they sought  to approximate it by a piecewise smooth function $u$, whose discontinuities delineate the segments of the image. 

We recall their functional using the formulation in the space of special functions of bounded variation.
For background on spaces of (special) functions of bounded variation we refer the reader to the book
\cite{ambrosio2000functions}. For $u\in SBV(\Omega)$
\begin{equation}\label{eqn:MSf}
	MS_f(u):= \lambda \int_\Omega |u-f|^2 dx + \int_{\Omega} |\nabla u|^2 \d x +  \H^{d-1}(S_u)
\end{equation}
where $f \in L^\infty(\Omega)$ is the noisy image, $\nabla u$ is the absolutely continuous (in the measure theoretic sense, and with respect to the Lebesgue measure)  part of the gradient $Du$ (which is a measure) of the function $u$, $\,S_u$ is the jump set of $u$,  and $\H^{d-1}$ is the $(d-1)-$dimensional Hausdorff measure. The first term of the functional ensures the closeness of the approximation $u$ to the original image $f$ while the next two terms reward the regularity of $u$. The idea is that natural images are piecewise smooth, but often do have jumps in intensity between different regions. Thus the terms of the functional reward the regularity of $u$, while still allowing jumps in the intensity. 

Thanks to the work of  Ambrosio in \cite{ambrosio1989compactness} and to the lower-semicontinuity of $MS_f$ with respect to the topology of the space $SBV(\Omega)$, the direct method of calculus of variation ensures us that a minimum $u_0\in SBV(\Omega)$ for the functional \eqref{eqn: MS functional} is always attained. 

For the considerations we have in mind the fidelity term $\lambda \int_\Omega |u-f|^2 dx$ is quite straightforward to treat. Hence, for readability, we introduce the functional without it and focus mainly on this functional: 
\begin{equation}\label{eqn: MS functional}
	MS(u):= \int
_{\Omega} |\nabla u|^2 \d x +  \H^{d-1}(S_u).
\end{equation}

As shown in \cite{braides1997non} any functional of the form of \eqref{eqn: MS functional} cannot be approximated in the sense of $\Gamma$-convergence by \textit{local} integral functional of the type
	\[
	\int_{\Omega} h_{\varepsilon}(\nabla u(x)) \d x
	\]
for $u\in W^{1,2}(\Omega)$.  De Giorgi conjectured that the Mumford-Shah functional can be approximated by nonlocal functionals. The conjecture was proved by  Gobbino in \cite{gobbino1998finite}, who showed that \eqref{eqn: MS functional} can be approximated by the functionals 
\begin{equation}\label{eqn: nonlocal MS}
{G}_{\varepsilon} (u) := \frac{1}{\varepsilon^{d+1}} \int_{\R^d \times \R^d} \arctan\left(\frac{|u(y) - u(x)|^2}{|y-x|}\right) e^{-\frac{|y-x|^2}{\e^2}} \d x \d y
\end{equation}
defined for $u\in L^1_{loc}(\Omega)$. He shows that for appropriate dimensional constants $\theta, \sigma$
	\[
	\Glim_{\e\rightarrow 0} {G}_{\e}= \theta\int_{\Omega} |\nabla u|^2 \d x + \sigma \H^{d-1}(S_u)
	\]
where the $\Gamma$-limit is considered with respect to $L^1$ topology.
The work in \cite{gobbino1998finite} has been then generalized in \cite{gobbino2001finite} 
to functionals defined on $SBV(\Omega)$ of the form 
	\begin{equation}\label{eqn: general MS}
	\F(u):=\int_{\Omega} \varphi \left(|\nabla u(x)|\right)  \d x + \int_{S_u} \psi(|u^+(x) - u^-(x)|) \d \H^{n-1}(x)
	\end{equation}
where $u^+(x)$ and $u^-(x)$ denote the so-called  approximate $\liminf$ and $\limsup$ of $u$ at the point $x$:
\begin{equation} \label{eq:uplus}
u^+(x) = \sup \left\{ t \in \R \::\; \lim_{r \to 0+} \frac{1}{r^n} |\{ y \in B(x,r) \::\: u(y) > t \}|  > 0 \, \right\}. 
\end{equation}
They show that   for suitable $\varphi, \psi$ the functional can be approximated in the $\Gamma$-convergence sense with the family of non-local functionals of the form 
	\begin{equation}\label{eqn: general nonlocal MS}
	\F_{\e}(u):=\int_{\R^d\times \R^d} \varphi_{|x-y|} \left( \frac{|u(x) - u(y)|}{|x-y|} \right) \eta_{\e}\left(x-y\right) \d x \d y
	\end{equation}
where $\{\varphi_{\e}\}_{\e}$ is a family of functions related to $\varphi, \psi$ and $\{\eta_{\e}\}_{\e>0} \subset  L^1(\Omega)$ is a kernel.  

%\dred Our purpose here, following a similar work on the non-local total variation \cite{GTS16}, is to consider a family of non-local Mumford-Shah $\mathcal{GMS}_{\e}$ defined on a clouds of $n$ i.i.d points of $\Omega$ and to investigate their behavior as $n$ goes to $\infty$ and $\e$ approaches $0^+$. 
%\nc

\subsection{Point cloud Mumford--Shah functional} \label{sec:assump}
The above nonlocal approximation to the Mumford--Shah functional can be adapted to the graph setting. We consider the setting of random geometric graphs formulated on random samples of a measure $\mu$ with density $\rho$, which describes the underlying data distribution.
Consider an open, bounded set with Lipschitz boundary $\Omega$.  The density $\rho$ is assumed to satisfy: 
$\rho \in C^1(\Omega) \cap C^0(\overline{\Omega})$ and
	\begin{equation}\label{eqn: lwr bound on rho}
0< c \leq\min_{x\in \overline{\Omega}}  \, \rho(x) \leq \max_{x\in \in \overline{\Omega}} \, \rho(x) \leq C<\infty.	
	\end{equation}
%
%For $L,\tau>0, $ $L,\tau \in \R_+$ we define the class of admissible densities $\rho$ to be
%	\begin{equation}\label{eqn: lwr bound on rho}
%	R(\Omega,L,\tau):=\left\{\rho\in C^1(\Omega) \cap C^0(\overline{\Omega}) \ \left|\  \text{Lip}(\rho)\leq L, \ \frac{1}{\tau} \leq \rho(x) \leq \tau\ \ \text{for all $x\in \Omega$} \right. \right\}.
%	\end{equation}
We consider $\zeta:[0,\infty) \rightarrow [0,\infty)$ such that
	\begin{itemize}
	\item[(A1)] $\zeta$ is concave and differentiable at $0$;
	\item[(A2)] $\zeta$ is non decreasing;
	\item[(A3)] $\zeta'(0)<\infty$ and
\begin{equation} \label{eq:theta}
 \Theta:=\lim_{x\rightarrow \infty} \zeta(x).
\end{equation}
\end{itemize}
We fix $p\geq 1$, $q\in [0,p)$ and we assume that the kernel $\eta:[0,\infty) \rightarrow [0,\infty)$ satisfies
	\begin{itemize}
	\item[(B1)] $\eta$ is a nonincreasing  $L^1$ function, non identically $0$;
	\item[(B2)] $0< \int_0^{\infty} (t^{d} +  t^{p-q+d-1} ) \eta(t) \d t <\infty$.
	\end{itemize}

In the sequel, we always assume the functions $\eta,\zeta$ and $\rho$ to satisfy the above assumption.
\medskip 

Let $x_1,\ldots, x_n \in \Omega$ a set of $n$ i.i.d random points on $\Omega$ chosen according to $\mu=\rho\d x$. The empirical measure of the sample is defined by 
	\[ 
	\mu_n:=\frac{1}{n}\sum_{i=1}^n \delta_{x_i} %\rightharpoonup^* \mu
	\] 
Given a Borel measure $\sigma$ on $\Omega$, the space  $L^p(\Omega, \sigma)$ is the space of equivalence classes of measurable functions $u : \Omega \to \R$ with 
$\int_{\Omega} |u|^p \d \sigma $ finite.   Notice that, under this assumption on $\rho$, we have that $L^1(\Omega;\rho)=L^1(\Omega)$. For that reason we often write $u\in L^1(\Omega)$ in place of $u\in L^1(\Omega;\rho)$.
\medskip

The  graph Mumford--Shah functional we devote the most attention to is the functional \eqref{eq:GMSfen} without the fidelity term. Namely  for a function $u\in L^1(\Omega;\mu_n)$ let 
\begin{equation}\label{eqn: graph non-local MS with the p}
	\mathcal{GMS}_{\e,n}(u):=\frac{1}{\e} \frac{1}{n^2}\sum_{i,j=1}^n \zeta\left( \e^{1-p+q} \frac{|u(x_i) - u(x_j)|^p}{|x_i-x_j|^q}\right) \eta_{\e}(|x_i - x_j|)
	\end{equation}
Here $\eta_{\e}(s):=\e^{-d}\eta(s/\e)$.\\

\subsection{Main results} We prove a $\Gamma$-convergence result (see Theorem \ref{thm main: G conv})  stating that the Graph Mumford--Shah  functional \eqref{eqn: graph non-local MS with the p} $\Gamma$-converges (in the $TL^1$ sense, recalled in  subsection \ref{sbsec: TLP} below), along suitable sequence $\{\e_n\}_{n\in \N}$ (see Remark \ref{rmk: on the eps condition}) to 
\begin{equation}\label{eqn: weighted MS with density}
MS_{\eta,\zeta}(u;\rho):= \vartheta_{\eta}(p,q)\zeta'(0) \int_{\Omega} |\nabla u(x)|^p \rho(x)^2 \d x 
 +\sigma_{\eta}  \Theta \int_{S_u} \rho(y)^2\d \H^{d-1}(y)
		\end{equation}
defined for all $u\in SBV^p(\Omega)$  and where $\Theta$ is defined by \eqref{eq:theta} and 
	\begin{equation} \label{eqn: costants definition}
	\left\{
	\begin{array}{rl}
	\vartheta_{\eta}(p,q)&:=\displaystyle2\omega_{d-1} \frac{\Gamma(p/2+1/2)\Gamma(d/2+1/2)}{\Gamma(p/2+d/2)} \int_0^{\infty} t^{p-q+d-1}\eta(t)\d t
	\text{}\\
\sigma_{\eta}&:=\displaystyle2\omega_{d-1}\int_{0}^{\infty}t^{d} \eta(t) \d t.\\
	\end{array}
	\right.
\end{equation}
We point out that assumption (B2) on $\eta$ is the one that guarantees the finiteness of $\sigma_{\eta},\vartheta_{\eta}(p,q)$.  With all these in mind  we are able to  show the validity of the following statement.
\begin{theorem}[$\Gamma$-convergence]\label{thm main: G conv}
Let $\Omega$ be an open set and $\rho$ be a probability density satisfying \eqref{eqn: lwr bound on rho}. Consider $\zeta, \eta$ satisfying the assumptions (A1)-(A3)  and (B1)-(B2). 

Let $\{x_i\}_{i\in \N}$ be a sequence of i.i.d. random points chosen accordingly to the density $\rho$ and $\{\e_n\}_{n\in \N}$ be a sequence of positive number converging to $0$ such that
%\begin{equation}\label{eqn: condition on eps}
%\left\{
%\begin{array}{ll}
%\displaystyle\lim_{n\rightarrow \infty} \frac{(\log(n) )^{3/4}}{\e_n n^{1/2}}=0  \ \ \ & \ \ \text{if $d=2$,}\\
%\text{}\\
%\displaystyle\lim_{n\rightarrow \infty} \frac{(\log(n) )^{1/d}}{\e_n n^{1/d}}=0  \ \ \ & \ \ \text{if $d\geq 3$}.
%\end{array}
%\right.
%\end{equation}
\begin{equation}\label{eqn: condition on eps}
\displaystyle\lim_{n\rightarrow \infty} \frac{(\log(n) )^{1/d}}{\e_n n^{1/d}}=0  \ \ \ \ \ \text{for $d\geq 2$}.
\end{equation}
%Assume that the function $\zeta$ and the kernel $\eta$ satisfy hypothesis (A1)-(A3), (B1)-(B2). 
Then $\mathcal{GMS}_{\e_n,n}$, defined in \eqref{eqn: graph non-local MS with the p},  $\,\Gamma-$converges to $MS_{\eta,\zeta}(\cdot; \rho)$, defined in \eqref{eqn: MS functional}, in the $TL^1$ sense.
\end{theorem}
We refer to \cite{GTS16} for detailed introduction of the $TL^p$ topology. For the reader's convenience we retrieve the main concepts in Subsection \ref{sbsec: TLP} below.
\begin{remark}\label{rmk: on the eps condition}
The condition \eqref{eqn: condition on eps} of Theorem \ref{thm main: G conv} comes from the following fact. Given random samples $\{x_1, \dots, x_n\}$ as above, we show in Lemma \ref{lem: improved convergence} that there exists  a sequence of probability measures $\tilde{\mu}_n$, absolutely continuous with respect to Lebesgue such that $\frac{\d \tilde{\mu}_n}{\d \mu} \rightrightarrows 1$ and whose $\infty$-Wasserstein distance from the empirical measure of the sample $\mu_n$ is decaying faster than $\e_n$. More precisely, there exist $T_n:\Omega\rightarrow \{x_1, \dots, x_n\}$  transport maps between $\tilde{\mu}_n=\rho_n \L^d$ and $\mu_n=\frac{1}{n}\sum_{i=1}^n \delta_{x_i}$, such that
\begin{equation} \label{eqn: decay on epsilon}
\displaystyle\lim_{n\rightarrow \infty} \frac{ \|T_n-\text{Id}\|_{\infty}}{\e_n } =  0.
\end{equation}
%let $\ell_n:=\|T_n-\text{Id}\|_{\infty}$
\end{remark}
In Section \ref{sec: compactness} we discuss compactness of the functionals. 
In particular we establish the following theorem:
\begin{theorem}\label{thm : compact}
Let $\Omega$, $\rho$, $\zeta$, $\eta$, and $x_i,\; \, i=1, \dots, n$ satisfy the assumptions of Theorem \ref{thm main: G conv}. Consider a sequence of $\{\e_n\}_{n\in \N}$ satisfying \eqref{eqn: condition on eps}. If $u_n \in L^{\infty}(\Omega;\mu_n)$ satisfy
	\[
	\sup_{n\in \N}\left\{\|u_n\|_{\infty}+\mathcal{GMS}_{\e_n,n}(u_n)\right\}<\infty,
	\]
then the sequence $\{(\mu_n,u_n)\}_{n \in N}$ it is $TL^1$-relatively compact.
\end{theorem}
\begin{remark}
%In order to prove Theorem \ref{thm : compact} we are exploiting the compactness result for the graph total variation (\cite[Theorem 1.2]{GTS16}) that requires the additional assumption $\eta(0)>0$ together with its continuity at $0$. 
Note that we ask for an $L^{\infty}$ bound on the sequence, instead of a weaker $L^1$ bound (as it is done in \cite[Theorem 1.2]{GTS16}). Here $L^1$ bound would not be sufficient as we show in Section \ref{sec: compactness} and Remark \ref{rmk: counterexample}. On the other hand since the signal $f$ in \eqref{eqn:MSf} is bounded in applications, the minimizers are also bounded by the same bound.  
\end{remark}

\begin{remark}  \label{rem:ruf}
Recently Ruf \cite{Ruf} has studied the convergence of graph Mumford--Shah functionals on random lattices to the continuum Mumford--Shah functional. These  interesting results are closely related, but also substantially different both in terms of their nature and the techniques used. One difference is the nature of randomness of the structure considered. Here we consider random samples, or in fact any discrete sets of points whose empirical measures weakly approximate the continuum measure. Ruf considers random lattices in 2D, which have precise requirements of the distribution of points. The disordered structure of the points we allow forces us to require that the typical degree of a vertex converges to infinity faster than $\log n$, while Ruf is able to work with graphs of bounded degree. On the flip side we identify the $\Gamma$ limit explicitly, while, due to the graph construction, Ruf only identifies the functional up to a constant. In a sense he is able to work under homogenization type graph-behavior where the compactness arguments show that $\Gamma$ limit exists without fully identifying it. 
\end{remark}

\medskip

\subsubsection{Convergence of functionals with the fidelity term.} 
The Theorems \ref{thm main: G conv} and \ref{thm : compact} enable us to show the convergence of the Mumford--Shah functional with the fidelity term as well. We establish two results. The first one is in the setting without noise. In order to be able to evaluate the signal at sample points we assume that $f$ is a bounded piecewise continuous function, that is that the set of discontinuities $J_f$ is of finite ${d-1}$ dimensional Hausdorff measure, $\H^{d-1}(J_f) < \infty$.

\begin{corollary} \label{cor:nonoise}
Let $\Omega$, $\rho$, $\zeta$, $\eta$, and $x_i,\; \, i=1, \dots, n$ satisfy the assumptions of Theorem \ref{thm main: G conv} and assume that $p\geq q+1$.
Assume $f : \Omega \to \R$ is a bounded, piecewise continuous function. 
Assume also that $p\geq q+1$.
Consider a sequence of $\{\e_n\}_{n\in \N}$ satisfying \eqref{eqn: condition on eps}. 
Then the functional $\mathcal{GMS}_{f,\e,n}$ defined in \eqref{eq:GMSfen},  considered with $f_i = f(x_i)$ for $i=1, \dots, n$, $\Gamma$-converges in $TL^2$ topology to  $MS_{\eta, \zeta}(u; \rho)+ \int |u-f|^2 \rho(x) dx $,  where $MS_{\eta, \zeta}(u; \rho)$ is defined in  \eqref{eqn: weighted MS with density}. Furthermore any sequence of minimizers  $u_n$ of $\mathcal{GMS}_{f,\e,n}$,
converge along a subsequence to a minimizer of $MS_{\eta, \zeta}(u; \rho)+ \int |u-f|^2 \rho(x) dx $.
\end{corollary}
We note that due to the fidelity term the topology of $\Gamma$ convergence in the corollary is $TL^2$ instead of $TL^1$. We remark that the change of the topology when considering the fidelity term was not needed in \cite{gobbino1998finite} since Gobbino could rely on the Fubini's theorem. However due the fact that we also deal with the discrete-to-continuum passage, % namely that the function $f$ itself needs to be compared  between the discrete and the continuum setting 
the stronger topology is needed. 
\medskip

More importantly and more interestingly we are able to establish the convergence of minimizers of the graph Mumford--Shah functional when the labels are noisy. We note that the limit is a minimizer of a deterministic variational problem, even though the amount of noise does not vanish as $n \to \infty$. 

\begin{corollary} \label{cor:noise}
Let $\Omega$, $\rho$, $\zeta$, $\eta$ satisfy the assumptions of Theorem \ref{thm main: G conv} and assume that $p\geq q+1$.
Assume $f : \Omega \to \R$ is a bounded, piecewise continuous function. 
Let $\beta$ be a measure on $\R$ modeling the noise. We assume $\beta$ has compact support and mean zero. Let $(x_i, y_i)$ for $i=1, \dots, n$ be i.i.d. samples of the measure $\mu \times \beta$. 
Consider a sequence of $\{\e_n\}_{n\in \N}$ satisfying \eqref{eqn: condition on eps}. 

Then the functional $\mathcal{GMS}_{f,\e,n}$ defined in \eqref{eq:GMSfen},  considered with 
\[ f_i = f(x_i) + y_i  \qquad \te{for } i=1, \dots, n, \]
 $\Gamma$-converges in $TL^2$ topology to  $MS_{\eta, \zeta}(u; \rho)+ \int |u-f|^2 \rho(x) dx + \Var(\beta) $,  where $MS_{\eta, \zeta}(u; \rho) $ is defined in  \eqref{eqn: weighted MS with density}. Furthermore any sequence of minimizers  $u_n$ of $\mathcal{GMS}_{f,\e,n}$,
converge along a subsequence to a minimizer of $MS_{\eta, \zeta}(u; \rho)+ \int |u-f|^2 \rho(x) dx $.
\end{corollary}

We make several observations. Note that while the minimizers of the random functional converge to the minimizers of a deterministic functional, and that the limit of the minimizers does not depend on the amount of noise. In a sense noise does not create a bias.
 The randomness affects the limit in that the actual values of the 
random discrete functionals are higher when there is more noise which leads to the presence of $\Var(\beta)$ in the deterministic limit. We note that while we do not allow for Gaussian noise this is purely for technical reasons, to make the proof of compactness easier. On the other hand we do not require the noise to have continuous density with respect to Lebesgue measure.

\begin{remark}
Let us contrast the result of Corollary \ref{cor:noise} to results on min-max recovery rates in nonparametric regression that we mentioned in the introduction (see \cite{HutRig16,MamvdG97, PSCW18, WSST_trend}). In the setting of regression one is concerned with recovering a function $f^\dagger$ in some functional class (e.g., BV unit  ball) whose noisy samples are available. Thus the fidelity term is made stronger as $n \to \infty$. Namely $\lambda$ in \eqref{eq:GMSfen} is taken to infinity at appropriate rate as $n$ increases. The works  obtain rates at which minimizers of functionals like \eqref{eq:GMSfen}, with TV regularization instead of the Mumford--Shah term, converge to $f^\dagger$. We conjecture that for the functions $f^\dagger$ in BV ball the functional \eqref{eq:GMSfen} also achieves the optimal min-max rate. One difference between the Mumford--Shah and the TV regularization is that the Mumford--Shah one does not decrease the contrast over sharp edges as TV regularization does. See Example \ref{ex:synth}. 
In the context of regression our contribution is that by taking the limit $n \to \infty$ as $\lambda$ is fixed we shed the light on what is the precise amount of regularization introduced by the  Mumford--Shah term at finite $\lambda$. 
\end{remark}

\subsubsection{Extension to data on a manifold.}
In machine learning it is often relevant to consider data that lie in a potentially high dimensional space, but have an intrinsic low dimensional structure. Here we remark that it is straightforward to extend our results to the setting where data are sampled from a measure $\mu$ whose support is a $d$-dimensional $C^2$ manifold, without boundary, $\M$, embedded in $\R^D$ for some $D>d$. We require that the measure $\mu$ has a continuous density $\rho$ with respect to the volume form of the manifold $\M$: $d\mu = \rho \te{dVol}_\M$. 
The form of the graph Mumford--Shah functional remains the same, while the only change in the limiting functional \eqref{eqn: weighted MS with density} is that $\nabla u$ is replaced by the manifold gradient $\grad_{\M} u$. We note that the scaling of $\eta_\e$ and the definition of $\sigma_\eta$ depend on the intrinsic dimension $d$, but not the ambient dimension $D$. Full details of how related statements are extended to the manifold setting  can be found in \cite{GGHS}. 

\subsubsection{Extension to vector valued functions.} We note that in machine learning it is also natural to consider functions on graphs which are vector valued. The graph  Mumford-Shah functional \eqref{eq:GMSfen} can be considered for functions $u$ with values in $\R^m$. In fact such functionals have been used in the work of Hallac,  Leskovec,  and  Boyd \cite{hallac2015}. We do not rigorously treat the limits of vector-valued functionals in this paper. Nevertheless we remark that we expect that the $\Gamma$-limit of the vector valued graph Mumford-Shah functional for $p=2$ is the following Mumford-Shah type functional:  for $u\in SBV(\Omega)^m$
	\begin{equation*}
MS^m_{\eta,\zeta}(u;\rho):= \vartheta_{\eta}(2,q)\zeta'(0) \int_{\Omega} |\nabla u(x)|^2 \rho(x)^2 \d x 
 +\sigma_{\eta}  \Theta \int_{S_u} \rho(y)^2\d \H^{d-1}(y)
		\end{equation*}
where $S_u$ is the union of jump sets for each of coordinate functions, $u=(u_1, \dots, u_m)$ and $S_u = \bigcup_{j=1}^m S_{u_i}$, and where $\Theta$ is defined by \eqref{eq:theta} and $\vartheta_\eta$ and $\sigma_\eta$ by \eqref{eqn: costants definition}. We furthermore expect that one can prove this using similar techniques that we use. We note that this would require a careful verification of the slicing argument in multiple dimensions.

%%%%%%%%%%%%%%%%%%%%%%%%%%%%%%%%%%%%%%%%%%%%%%%%%
\section{$\Gamma$-convergence and topology in the space of configuration} \label{sec:prelim}
Given a sequence of functionals $F_{n}:X\rightarrow \R$ and a metrizable (and separable) topology $\mathcal{T}$ on $X$ we say that $F_{n}$ $\Gamma$-converges at $F:X\rightarrow \R$ with respect to the topology $\mathcal{T}$  if the following two conditions are satisfied:
	\begin{itemize}
	\item[(i)] For every sequence $\{x_{n}\}_{n \in \N} \subset X$ such that $\displaystyle x_{n} \stackrel{\mathcal{T}}{\longrightarrow} x$ it holds that
		\[
		\displaystyle \liminf_{n\rightarrow \infty} F_{n} (x_{n}) \geq F(x);
		\]
	\item[(ii)] For all $x\in X$ there exists a sequence  $\{x_{n}\}_{n \in \N} \subset X$ such that  $\displaystyle x_{n} \stackrel{\mathcal{T}}{\longrightarrow} x$ and for which 	
	\[
	\displaystyle  \limsup_{n\rightarrow \infty} F_{n}(x_{n} )\leq F(x).
	\]
	\end{itemize}
In this case we write
	\[
	\Glim_{n \rightarrow 0} F_{n}=F.
	\]
Notice that, if $	\displaystyle \Glim_{n \rightarrow \infty} F_{n}=F$ then the following assertions also hold:
	\begin{itemize}
	\item[(i)] $F$ is lower semi-continuous and
		\begin{align*}
		F(x)&=\inf\left\{\liminf_{n \rightarrow \infty} F_{n}(x_{n}) \ \Big{|} \ \ \{x_{n}\}_{n \in \N} \subset X,\   x_{n} \stackrel{\mathcal{T}}{\longrightarrow} x \right\}\\
			&=\inf\left\{\limsup_{n \rightarrow \infty} F_{n}(x_{n}) \ \Big{|} \ \ \{x_{n}\}_{n \in \N} \subset X,\   x_{n} \stackrel{\mathcal{T}}{\longrightarrow} x \right\};
		\end{align*}
		\item[(ii)] If $\{x_{n}\}_{n \in \N}$  is a sequence of minima of $F_{n}$ on $X$, namely
			\[
			F_{n}(x_{n}) = \min_{y\in X} \{ F_{n}(y) \},
			\]
		converging to some $x\in X$ in the topology $\mathcal{T}$ then $x$ is a minimum of $F$ on $X$:
			\[
			F(x)=\min_{y\in X} \{ F(y) \}.
			\]
	\end{itemize}
	
\subsection{The $TL^p$ topology: brief overview.}\label{sbsec: TLP}
The $TL^p$ space has been introduced in \cite{GTS16}
% to handle a problem similar to the one treated here, involving total variation on graph. 
Given a bounded open set $\Omega$ let $\mathcal P(\Omega)$ be the set of Borel probability measures on $\overline \Omega$.
% and let $\mathcal P_p(\Omega)$ be the subset of measures in $\mathcal P(\Omega)$ which have finite $p$-th moment $\int |x|^p d \mu(x) < \infty$. 
The space $TL^p(\Omega)$ is defined by 
	\begin{equation}\label{eqn: TLp space}
	TL^p(\Omega):= \{ (\mu,f) \: : \: \mu \in \mathcal P(\Omega)  \ f \in L^p(\Omega;\mu) \}.
	\end{equation}
Given $(\mu,f), (\nu,g)\in TL^p(\Omega)$ their $TL^p$ distance is defined as follows
	\[
	d_{TL^p}((\mu,f), (\nu,g)):= 
\begin{cases}
\inf_{\pi\in \Gamma(\mu,\nu)}  \left(\int_{\Omega\times \Omega} |x-y|^p + |f(x)-g(y)|^p  \d \pi(x) \right)^{\frac{1}{p}} \; & \te{if } p \in [1, \infty), \\
\inf_{\pi\in \Gamma(\mu,\nu)} \esssup_{(x,y) \in \te{supp}\; \pi}  |x-y| + |f(x) -g(y)| & \te{if } p = \infty
\end{cases}
	\]
where the infimum is taken among all transport plans $\pi$  between $\mu$ and $\nu$:
		\[
		\Gamma(\mu,\nu):= \{ \gamma \in \mathcal P( \Omega \times \Omega) \::\:  
		(\forall A \text{ Borel}) \; \gamma(A \times \Omega)   = \mu(A), \, \gamma( \Omega \times A)   = \nu(A) \}.
		\]
%Here $P_x, P_y:\Omega\times \Omega \rightarrow \Omega$ are the projections with respect to the first and the second variable, $P_x((\xi,\eta))=\xi$, $P_y(\xi,\eta)=\eta$, and 
%\begin{equation}
%\begin{array}{lr}
%\left\{
%	\begin{array}{ll}
%	(P_x)_{\#}\gamma (A):=& \gamma(P_x^{-1}(A) ) \\
%	(P_y)_{\#}\gamma  (A):= &\gamma(P_y^{-1}(A) )
%	\end{array}
%	\right.
%&
% 	\text{for all $A\subset \Omega$ measurable.}
% 	\end{array}	
% \end{equation}
% \begin{remark}
%\medskip

Given $\mu \in \mathcal{P}(\Omega)$ and a measurable mapping $T: \Omega \rightarrow \Omega$, we recall that $\nu = T_{\#} \mu$ is the \emph{push-forward} of $\mu$ by $T$, namely the measure on $\Omega$ such that for any $A$ Borel $\nu(A) = \mu(T^{-1}(A))$. A consequence of the definition is 
the following change of variables identity
	\begin{equation}\label{eqn: change of variable formula}
	\int_\Omega f(T(x)) \d \mu(x)= \int_{\Omega} f(y) \d \nu (y).
	\end{equation}
%}
% \end{remark}
Well-known results of the theory of optimal transportation, \cite{brenier87} for $p =2$, \cite{gangbomccann} for $p< \infty$ and \cite{champion}, for $p=\infty$, provide that when $\mu$ is absolutely continuous with respect to $\mathcal{L}^d$ then there exists an optimal  transport map between $\mu$ and $\nu$, namely 
$T: \Omega\rightarrow \Omega$ such that  $T_{\#} \mu=\nu$ and 
\begin{align}
\begin{split}
d_p^p(\mu, \nu) & :=	\inf_{\gamma \in \Gamma(\mu,\nu)} \int_{\Omega\times\Omega} |x-y|^p \d \gamma(x,y)  = \int_{\Omega} |T(x)-x|^p \d \mu(x)  \quad  \te{when } p < \infty, \\
d_\infty(\mu, \nu) & :=\inf_{\gamma \in \Gamma(\mu,\nu)} \esssup_{(x,y) \in \te{supp}(\gamma)}  |x-y|  =  \esssup_{x \in \te{supp}(\mu)}  |T(x)-x|  \quad \te{ when } p = \infty.
\end{split} \label{def:dp}
\end{align}
In particular the transport plan induced by $T$, namely $\pi:=(\text{Id}, T)_{\#} \mu $, is optimal. 
The distance $d_p(\mu, \nu)$ we define in the expressions above is called the $p$-transportation metric (sometimes referred to as the $p$-Wasserstein distance).

When considering convergence of a sequence  $(\mu_n,f_n)$ toward $(\mu,f)$, the following sufficient criterion will be useful. We say that a sequence of transportation maps is \emph{stagnating} if
\begin{equation} \label{eq:stagnating}
 T_{n\#} \mu = \mu_n \quad \quad \text{and} \quad \quad \| I - T_n \|_{L^p(\mu)}^p = \int_\Omega |x-T_n(x)|^p  \, \d \mu(x) \to 0 
\end{equation}
as $n\to \infty$.  To show $TL^p$ convergence it thus suffices to find a stagnating sequence of transportation maps such that $\int |f(x) - f_n(T_n(x))|^p d \mu(x)$ converges to zero as $n \to \infty$. 
\medskip

We now introduce the new results that allow us to  obtain the optimal scaling of $\e_n$ for $\Gamma$-convergence in 2D. Namely while $d_\infty(\mu, \mu_n) \sim \frac{(\ln n)^{3/4}}{\sqrt n}$ when $d=2$ we introduce an auxiliary measure $\tmu_n$ which is absolutely continuous with respect to Lebesgue measure and satisfies both that $d_\infty(\tilde \mu_n, \mu_n) \sim \sqrt{ \frac{(\ln n)}{ n}}$ when $d=2$ and that its density with respect to measure $\mu$ uniformly converges to $1$. These two facts are enough to pass to the limit in the functionals we consider, and many others (like total variation or dirichlet energy). Lemma \ref{lem:improved convergence TLp} is a technical result needed to transfer the $TL^p$ convergence to the desired measures. 

\begin{lemma}\label{lem: improved convergence}
Let $\mu$ be a probability measure with continuous  density $\rho$, supported on $\overline \Omega$, where $\Omega$ is a bounded  open set with Lipschitz boundary in $\R^d$, $d \geq 2$ and which satisfies the assumption \eqref{eqn: lwr bound on rho}. Let $\e_n$ be a sequence of positive numbers converging to zero and satisfying \eqref{eqn: condition on eps}.
Let $\{x_i\}_{i\in \N}$ be a sequence of i.i.d. random points chosen according to the density $\rho$, and let $\mu_n = \frac{1}{n} \sum_{i=1}^n \delta_{x_i}$. Then there exists a sequence of probability  measures $\tmu_n$  which are absolutely continuous with respect to the measure $\mu$ and satisfy
\begin{itemize}
\addtolength{\itemsep}{3pt}
\item[(i)]  Almost surely $\displaystyle{ \ell_n := d_\infty(\mu_n, \tmu_n) \ll \e_n}$ 
\item[(ii)] As $n \to \infty,$ $\, \displaystyle{\frac{d\tmu_n}{d\mu} }$ almost surely as  converges to $1$ uniformly on $\overline \Omega$. 
\end{itemize}
\end{lemma}
\begin{proof}
Let us first consider the case that $\Omega = [0,1]^d$. Given a sequence $\e_n$ satisfying the assumptions of the lemma let $\{b_n\}_{n\in \N}$ and $\{c_n\}_{n \in \N}$ be increasing sequences of positive numbers such that $\left(b_n \, \frac{\ln n}{n} \right)^{-1/d}$ is integer and
\begin{equation}
b_n \to \infty, \quad
\e_n \gg \left(b_n \, \frac{\ln n}{n} \right)^{1/d}  \te{ and } 
\; c_n^2 \gg  b_n  \gg c_n \; \quad \te{ as } n \to \infty.
\end{equation}
Let, for $n \geq 2$, 
\[ \delta^d= b_n \, \frac{\ln n}{n} \; \te{ and }\;  t = c_n \, \frac{\ln n}{n} \]
We divide $[0,1]^d$ into $m =\left(b_n \, \frac{\ln n}{n} \right)^{-1}$ disjoint cubes $K_j$, $j=1, \dots, m$ with side length $\delta$. Note that the probability $p_j$ that a point $x_i$ is in the box $K_j$ is equal to $\mu(K_j)$ and that 
\[ c \delta^d \leq p_j \leq C \delta^d. \]
Bernstein's inequality \cite{bern24} gives
\begin{align} \label{ineq:Ber}
\begin{split}
P\left( |\mu_n(K_j) - p_j|\geq t  \right) & < 2 \exp \left(- \frac{\frac12 n t^2}{p_j (1-p_j) + \frac13 t} \right) < 2 \exp\left(- \frac{nt^2}{3 p_j}  \right)  \\
& \leq 2 \exp\left(- \frac{c_n^2 \ln n}{3C b_n}  \right)  = 2 n^{- \frac{c_n^2}{3C b_n}}. 
\end{split}
\end{align}
It follows, by union bound, that the probability that in all boxes $ |\mu_n(K_j) - p_j| <  t$ satisfies 
\begin{equation} \label{eq:union}
 P(\{ (\forall j =1, \dots, m) \;\,  |\mu_n(K_j) - p_j|< t) \geq 1 - m    2 n^{- \frac{c_n^2}{3C b_n}} \geq 1- n^{-2},
\end{equation}
for all $n$ large enough. By Borel--Cantelli Lemma we conclude that almost surely for $n$ sufficiently large for all boxes it holds that  $ |\mu_n(K_j) - p_j|< t$. 

Define the measure $\tilde \mu_n$ as follows:
\[ d \tilde \mu_n = \sum_{j=1}^m \ca_{K_j} \frac{\mu_n(K_j) }{\delta^d} \, dx \]
Since $\tilde \mu_n(K_j) = \mu_n(K_j)$, the distance $d_\infty(\tilde \mu_n , \mu_n)$ is at most the diameter of the boxes, namely
\[ d_\infty(\tilde \mu_n , \mu_n) \leq \sqrt{d}\,  \delta. \]

For large $n$ and arbitrary $x \in \Omega$ let $K_j$ be such that $x \in K_j$. Using \eqref{ineq:Ber} we obtain
\begin{align*}
\left| \frac{d \tilde \mu_n}{d\mu}(x) -1 \right|& \leq \left|  \frac{\frac{\mu_n(K_j)}{\delta^d} - \rho(x)}{\rho(x) }\right| \\
& \leq   \frac{|\mu_n(K_j) - \mu(K_j)|}{\delta^d \, \rho(x)} + \frac{ \int_{K_j} |\rho(z) - \rho(x)| dz }{\delta^d \, \rho(x)} \\
& \leq  \frac{t}{c \delta_d} + \frac{1}{c}  \omega(\sqrt{d}\, \delta) \leq \frac{c_n}{c b_n} + \frac{1}{c}  \omega(\sqrt{d}\, \delta),
\end{align*}
where $\omega$ is the modulus of continuity of $\rho$. The uniform convergence follows since the terms on the right-hand side converge to zero. 
\medskip

Extending the argument to general $\Omega$ with smooth boundary is straightforward using  the partition procedure detailed in Section 3 of \cite{GTS15a}. To general case of  $\Omega$ with Lipschitz boundary can be reduced to are domains with smooth boundary using the result of Ball and Zarnescu \cite{ball2017partial}, as was done is the Step 2 of the proof of Theorem 1.2 in \cite{GTS15a}. 
\end{proof}

\begin{lemma} \label{lem:improved convergence TLp}
Let $\mu$ be a probability measure with density $\rho$ supported on a compact set in $\R^d$. Let $\tmu_n$ be a sequence of probability measures which are absolutely continuous with respect to $\mu$  such that 
\[ \frac{d \tmu_n}{d \mu} \to 1 \quad \te{ uniformly on the support of } \mu.  \]
Assume $f_n \to f$ in $L^p(\mu)$ as $n \to \infty$. Then
\begin{equation} \label{eq:mod_rho_conv}
 (\tmu_n , f_n) \to (\mu,f) \quad \te{ in } TL^p \te{ as } n \to \infty.
\end{equation}
\end{lemma}
\begin{proof}
From the assumption that $ \frac{d \tmu_n}{d \mu} $ uniformly converges to $1$ follows that the L\'evy--Prokhorov metric between $\tmu_n$ and $\mu$ converges to zero. Since the Levy-Prokhorov and the $p$-transportation metric, $d_p$, defined by \eqref{def:dp}, both metrize the weak convergence of measures on compact sets we conclude that  $d_p(\mu, \tmu_n) \to 0$ as $n \to \infty$. 
Thus there exists a sequence of transportation plans $\pi_n \in \Pi(\mu, \tmu_n)$ such that 
\[ \int |x-y|^p d\pi_n(x,y) \to 0 \qquad \te{as } n \to \infty. \]
Since Lipschitz continuous functions are dense in $L^p(\mu)$ there exists a sequence of Lipschitz continuous functions $g_m$ which converges to $f$ in $L^p(\mu)$. Let $\rho_n$ be the Lebesgue density of $\tmu_n$. Since $ \frac{\rho_n}{\rho}$ uniformly converges to $1$, there exists $n_1$ such that for all $n \geq n_1$, $\, \frac12 \leq  \frac{\rho_n}{\rho} \leq 2$. For $n \geq n_1$
\[  \int |f_n(y) - f(x)|^p d \pi_n(x,y)  \leq 2^p \left( \int |f_n(y) - f(y)|^p d \tmu_n(y) + \int |f(y) - f(x)|^p d \pi_n(x,y) \right).  \]
We estimate the terms separately:
\[ \int |f_n(y) - f(y)|^p d \tmu_n(y) \leq 2 \int |f_n(y) - f(y)|^p d \mu(y), \]
which converges to zero as $n \to \infty$. 
\begin{align*}
\! \int \! |f(y) - f(x)|^p d \pi_n(x,y) & \lesssim \!\! \int \! |f(y) - g_m(y) |^p + |g_m(y) - g_m(x)|^p + |g_m(x) - f(x)|^p d \pi_n(x,y) \\
 & \lesssim \| f - g_m\|_{L^p(\mu)}^p + \int  \Lip(g_m) |x-y|^p  \pi_n(x,y).
\end{align*}
We observe that the right hand can be made arbitrarily small by taking $m$ large enough and then taking $n$ sufficiently large. Consequently $ \int |f_n(y) - f(x)|^p d \pi_n(x,y)$ converges to zero as $n \to \infty$, which implies that \eqref{eq:mod_rho_conv} holds. 
\end{proof}

%%%%%%%%%%%%%%%%%%%%%%%%%%%%%%%%%%%%%%%%%
\section{Proof of the $\Gamma$-convergence (Theorem \ref{thm main: G conv}) } \label{sec:gamma}

We prove Theorem \ref{thm main: G conv} by separately proving the $\Gamma$-liminf bound and 
by  building a recovery sequence.  The proof of $\Gamma$-liminf bound relies on slicing (i.e. one dimensional decomposition) as used by Gobbino \cite{gobbino1998finite} and the techniques of \cite{GTS16} to deal with randomness of the sample. However since the spatial coordinates appear also in a denominator within the functional \eqref{eqn: graph non-local MS with the p}, the way 
\cite{GTS16} deals with space is not precise enough and new ideas are needed to overcome this challenge. These are discussed in Subsection \ref{sec:liminf}. In the subsection below we introduce notation and present the one-dimensional slicing of the continuum Mumford--Shah functional. 

% inserted part
\subsection{One dimensional slicing} \label{sec:slice}
 We will argue on the one dimensional slices in order to achieve the desired bounds. We repeatedly use the following computation, which we sketch for a generic function $f$.

Let $f:\Omega \times \Omega \rightarrow \R$ be a integrable function. Then, for $\e>0$ it holds
\begin{equation}\label{eqn: cov}
\begin{split}
\int_{\Omega \times \Omega} f(x,y) \d x \d y&=	\e^d\int_{\Omega} \d x\int_{\frac{\Omega - x}{\e}} f(x,x+\e \xi)  \d \xi\\
&=\e^d\int_{\R^d} \d \xi \int_{\Omega\cap (\Omega -\e \xi) } f(x,x+\e \xi) \d x
\end{split}
\end{equation}
where we exploited the identity
	\[
	\ca_{\Omega}(x)\ca_{(\Omega - x)/\e}(\xi)=\ca_{\R^d}(\xi) \ca_{\Omega -\e \xi}(x)
	\]
and then Fubini's Theorem (here $\ca_E(x)$ stands for the characteristic function of the set $E$ and takes value $1$ for $x\in E$ and $0$ otherwise).
%------------------------------------------------------------
% Justification of the above identity.
%	\begin{align*}
%	(x,\xi)\in\Omega \times (\Omega - x)/\e \ & \Rightarrow \ \xi=(y-x)/\e  \  \ (\exists \ y\in \Omega), \ x=y-\e\xi \\ 
%	&\Rightarrow \ x\in  \Omega \cap (\Omega-\e \xi), \\
%	&\Rightarrow \ \Omega \times (\Omega - x)/\e\subseteq \Omega\cap(\Omega - x)/\e \times \R^d. \\
%	&\text{}\\
%		(x,\xi)\in \Omega \cap (\Omega-\e \xi)\times \R^d \ & \Rightarrow \ x=y-\e\xi\ \  \ (\exists \ y\in \Omega), \ \xi=(y-x)/\e\\ 
%	&\Rightarrow \ \xi \in (\Omega-x)/\e , \\
%	&\Rightarrow  \ (\Omega)\cap(\Omega - x)/\e \times \R^d\subseteq  \Omega \times (\Omega - x)/\e. \\
%	\end{align*}
%------------------------------------------------------------
Given $A\subseteq \R^d$ and  $\xi \in \R^d$  we define, for and for $z\in \xi^{\perp}$, the one dimensional slice 
\begin{equation} \label{eq:slice}
	[A]_z:=\{t\in \R \ : \ z+t\xi/|\xi|\in A\}.
\end{equation}
Above we are omitting the dependence on $\xi$, since it will be  clear from the context. 
We proceed to consider one dimensional slices as follows
\begin{equation}\label{eqn: slicing process}
\begin{split}
\int_{\Omega \times \Omega} f(x,y) \d x \d y&=\e^d\int_{\R^d} \d \xi \int_{\Omega\cap (\Omega -\e \xi) } f(x,x+\e \xi) \d x\\
&=\e^d\int_{\R^d} \d \xi \int_{\xi^{\perp}} \d z \int_{[\Omega\cap (\Omega -\e \xi)]_z} f\left(z+t\frac{\xi}{|\xi|},z+(t+\e|\xi|)\frac{\xi}{|\xi|} \right) \d t.
\end{split}
\end{equation}

When using  one dimensional decompositions we will make several uses of the following notation. Given $g:\R^d\rightarrow \R$, $\xi\in \R^d$ and $z\in \xi^{\perp}$ we define
\begin{equation} \label{eq:subxi}
	g_{\xi}(t;z):=g\left(z+t\frac{\xi}{|\xi|}\right).
\end{equation}

We now state and prove two technical lemmas that we use in the sequel.
\begin{lemma}\label{lem: reconstruction of u}
Let $u\in SBV(\Omega)^p$. Then
	\begin{align}
	MS_{\eta,\zeta}(u;\rho)=\zeta'(0) \int_{\R^d} |\xi|^{p-q}&\eta(|\xi|)\d \xi\int_{\xi^{\perp}} \d z \int_{[\Omega]_z} \left| u_{\xi}'(t;z)\right|^p \rho_{\xi}(t;z)^2\d t \nonumber\\
	&+ \Theta\int_{\R^d} |\xi|\eta(|\xi|)\d \xi\int_{\xi^{\perp}} \d z \int_{S_{u_{\xi}(\cdot;z)}}  \rho_{\xi}(t;z)^2\d \H^0(t)\label{eqn: MS sliced}\\
	= \zeta'(0) \int_{\R^d} |\xi|^{p-q}&\eta(|\xi|)\d \xi\int_{\Omega} \left| \nabla u(x)\cdot \frac{\xi}{|\xi|}\right|^p \rho(x)^2\d x \nonumber\\
	&+ \Theta\int_{\R^d} |\xi|\eta(|\xi|)\d \xi \int_{S_{u}} \left|N_u(y) \cdot \frac{\xi}{|\xi|} \right|\rho\left(y\right)^2\d \H^{d-1}(y).\label{eqn: MS almost sliced}
	\end{align}
where $N_u(y)$ is any vector field normal to $S_u$.
\end{lemma}
\begin{proof}
We can rewrite the right-hand side of \eqref{eqn: MS sliced} as follows
	\begin{align*}
	 \zeta'(0) \int_{\R^d} |\xi|^{p-q}&\eta(|\xi|)\d \xi\int_{\xi^{\perp}} \d z \int_{[\Omega]_z} \left| u_{\xi}'(t;z)\right|^p \rho_{\xi}(t;z)^2\d t \\
	&+ \Theta\int_{\R^d} |\xi|\eta(|\xi|)\d \xi\int_{\xi^{\perp}} \d z \int_{S_{u_{\xi}(\cdot;z)}}  \rho_{\xi}(t;z)^2\d \H^0(t)\\
	= \zeta'(0) \int_{\R^d} |\xi|^{p-q}&\eta(|\xi|)\d \xi\int_{\xi^{\perp}} \d z \int_{[\Omega]_z} \left| \nabla u\left(z+t\frac{\xi}{|\xi|}\right)\cdot \frac{\xi}{|\xi|}\right|^p\rho\left(z+t\frac{\xi}{|\xi|}\right)^2\d t \\
	&+ \Theta\int_{\R^d} |\xi|\eta(|\xi|)\d \xi\int_{\xi^{\perp}} \d z \int_{S_{u_{\xi}(\cdot;z)}} \rho\left(z+t\frac{\xi}{|\xi|}\right)^2\d \H^0(t)	\\
	= \zeta'(0) \int_{\R^d} |\xi|^{p-q}&\eta(|\xi|)\d \xi\int_{\Omega} \left| \nabla u(x)\cdot \frac{\xi}{|\xi|}\right|^p \rho(x)^2\d x \\
	&+ \Theta\int_{\R^d} |\xi|\eta(|\xi|)\d \xi \int_{S_{u}} \left|N_u(y) \cdot \frac{\xi}{|\xi|} \right|\rho\left(y\right)^2\d \H^{d-1}(y).
	\end{align*}
The last equality above  follows from the Coarea formula, using a well known relation between the one-dimensional slices of an SBV function and the total length of its jump set (see for instance \cite{ambrosio1989compactness}). Notice that
\begin{align*}
\int_{\R^d} |\xi|  \eta(|\xi|) \left|N_u(y)\cdot \frac{\xi}{|\xi|} \right| \d \xi&=\int_{0}^{\infty} t \d t \int_{\partial B_t} \eta(t) \left|N_u(y) \cdot \frac{\xi}{t} \right| \d \H^{d-1}(\xi)\\
&=\int_{0}^{\infty}t^{d}  \eta(t) \d t \int_{\partial B_1}  \left|N_u(y) \cdot v  \right| \d \H^{d-1}(v)\\
&=2 \omega_{d-1}\int_{0}^{\infty}t^{d} \eta(t) \d t =\sigma_{\eta}
\end{align*}
%------------------------------------
%Justification of the above formula
%\begin{align*}
%\int_{B_r}  \left|\nu \cdot v  \right| \d v&=  r^{d+1} \int_{B_1}  \left|\nu \cdot v  \right| \d v\\
%&=  2r^{d+1} \int_{0}^{1}  s \H^{d-1} (B_1\cap \{\nu \cdot v = s\}) \d s.
%\end{align*}
%\begin{align*}
%\H^{d-1} (B_1\cap \{\nu \cdot v = s\})&=\omega_{d-1}(1-s^2)^{(d-1)/2}.
%\end{align*}
%
%\begin{align*}
%\int_{B_r}  \left|\nu \cdot v  \right| \d v&=  r^{d+1} \int_{B_1}  \left|\nu \cdot v  \right| \d v\\
%&=  2r^{d+1} \omega_{d-1} \int_{0}^{1}  s (1-s^2)^{(d-1)/2}  \d s\\
%&=  r^{d+1} \omega_{d-1} \int_{0}^{1}  (1-x)^{(d-1)/2}  \d x=\omega_{d-1}r^{d+1} \frac{\Gamma(d/2+1/2)}{\Gamma(d/2+3/2)}\\
%&=\frac{2r^{d+1}}{d+1} \omega_{d-1}
%\end{align*}
%------------------------------------
where we have exploited the relation
\begin{align*}
\int_{\partial B_1}  \left|N_u(y) \cdot v  \right| \d \H^{d-1}(v) =2 \omega_{d-1} .
\end{align*}
Moreover
\begin{align*}
\int_{\R^d} |\xi|^{p-q} & \eta(|\xi|) \left|\nabla u(x) \cdot \frac{\xi}{|\xi|}\right|^p\d \xi\\
&= \int_0^{\infty} t^{p-q+d-1}\eta(t)\d t \int_{\partial B_1} |\nabla u(x)\cdot v|^p\d \H^{d-1}(v)\\
&=2|\nabla u(x)|^p \omega_{d-1}  \frac{\Gamma(p/2+1/2)\Gamma(d/2+1/2)}{\Gamma(p/2+d/2)} \int_0^{\infty} t^{p-q+d-1}\eta(t)\d t\\
&= |\nabla u(x)|^p \vartheta_{\eta}(p,q)
\end{align*}
%------------------------------------
%Justification of the above formula
%\begin{align*}
%\int_{ B_r} |w\cdot v|^p\d v&=2 r^{d+p}|w|^p \int_{0}^{1} t^p\d t  \int_{B_1\cap \{ v \cdot w/|w|=t\}}\d \H^{d-1}(v)\\
%&=2 r^{d+p}|w|^p \omega_{d-1} \int_{0}^{1} t^p (1-t^2)^{(d-1)/2} \d t \\
%&=r^{d+p}|w|^p \omega_{d-1} \int_{0}^{1} s^{(p-1)/2} (1-s)^{(d-1)/2} \d s \\
%&= r^{d+p}|w|^p \omega_{d-1}  \frac{\Gamma(p/2+1/2)\Gamma(d/2+1/2)}{\Gamma(p/2+d/2+1)} 
%\end{align*}
%------------------------------------
where we made use of
\begin{align*}
\int_{\partial B_1}|w\cdot v|^p\d \H^{d-1}(v)&=2|w|^p \omega_{d-1}  \frac{\Gamma(p/2+1/2)\Gamma(d/2+1/2)}{\Gamma(p/2+d/2)} .
\end{align*}
In particular
\begin{align*}
	 \zeta'(0) \int_{\R^d} |\xi|^{p-q}&\eta(|\xi|)\d \xi\int_{\xi^{\perp}} \d z \int_{[\Omega]_z} \left| u_{\xi}'(t;z)\right|^p \rho_{\xi}(t;z)^2\d x \\
	&+ \Theta\int_{\R^d} |\xi|\eta(|\xi|)\d \xi\int_{\xi^{\perp}} \d z \int_{S_{u_{\xi}(\cdot;z)}}  \rho_{\xi}(t;z)^2\d \H^0(t)\\
	= \zeta'(0) \int_{\R^d} |\xi|^{p-q}&\eta(|\xi|)\d \xi\int_{\Omega} \left| \nabla u(x)\cdot \frac{\xi}{|\xi|}\right|^p \rho(x)^2\d x\\
	&+ \Theta\int_{\R^d} |\xi|\eta(|\xi|)\d \xi \int_{S_{u}} \left|N_u(y) \cdot \frac{\xi}{|\xi|} \right|\rho\left(y\right)^2\d \H^{d-1}(y)\\
		=  \vartheta_{\eta}(p,q)\zeta'(0)& \int_{\Omega} | \nabla u(x)|^p \rho(x)^2\d x\\
	&+ \sigma_{\eta}\Theta \int_{S_{u}} \rho\left(y\right)^2\d \H^{d-1}(y)=MS_{\eta,\zeta}(u;\rho).
	\end{align*}
\end{proof}

\subsection{Auxiliary functionals} \label{sec:aux}
We introduce two auxiliary functionals: $\overline{\mathcal{GMS}}$ and $\mathcal{GAMS}$. The first one is motivated by the calculation below and allows us to switch the auxiliary measure $\tmu_n$, constructed in Lemma \ref{lem: improved convergence} by the measure $\mu$. The functionals $\mathcal{GAMS}$ moves a step further towards to local, limiting,  functional by replacing the integral over the product measure by one with weight $\rho(x)^2$. 
We will first establish the $\liminf$ and $\limsup$ bounds on an auxiliary energy $\mathcal{GAMS}$, and then, by exploiting Lemma \ref{lem I can kill the remainder} apply these bounds  to  $\overline{\mathcal{GMS}}$. Thanks to \eqref{eqn:key equality} we can then prove the statement of Theorem \ref{thm main: G conv} for $\mathcal{GMS}$.
%%%%% above was inserted

Consider the setting of Theorem \ref{thm main: G conv} and let $\tmu_n$ be the measures constructed in Lemma \ref{lem: improved convergence}. 
%\begin{remark} \label{rmk: convenient way of writing GF}
Let $T_n: \Omega \rightarrow \Omega$ be the $d_\infty$ optimal transport  from $\tmu_n$ to $\mu_n$. Let $\ell_n = d_\infty(\tmu_n, \mu_n) = \|T - Id\|_{L^\infty}$. By exploiting the change of variable \eqref{eqn: change of variable formula}, we can rewrite the Mumford--Shah functional on the point clouds in the following integral form:
	\begin{align*}
	\mathcal{GMS}_{\e,n}&(u):=\frac{1}{\e} \frac{1}{n^2}\sum_{i,j=1}^n \zeta\left( \e^{1-p+q} \frac{|u(x_i) - u(x_j)|^p}{|x_i-x_j|^q}\right) \eta_{\e}(|x_i - x_j|)\\
	&=\frac{1}{\e} \int_{\Omega \times \Omega} \zeta\left( \e^{1-p+q} \frac{|u(x) - u(y)|^p}{|x-y|^q}\right) \eta_{\e}(|x- y|) \d \mu_n(x) \d \mu_n(y)\\
	&=\frac{1}{\e} \int_{\Omega \times \Omega} \zeta\left( \e^{1-p+q}\frac{|u(x) - u(y)|^p}{|x-y|^q}\right) \eta_{\e}(|x- y|) \d (T_{n\#} \tmu_n )(x) \d (T_{n\#} \tmu_n ) (y)\\
	&=\frac{1}{\e} \int_{\Omega} \int_{\Omega} \zeta\left( \e^{1-p+q} \frac{|u(T_n(x) ) - u(T_n(y) )|^p}{|T_n(x)-T_n(y)|^q}\right) \eta_{\e}(|T_n(x)- T_n(y)|) \d \tmu_n (x) \d\tmu_n (y) \\
& \leq \sup_{x \in \Omega} \left| \frac{\rho_n}{\rho} \right|^2 
\frac{1}{\e} \int_{\Omega} \int_{\Omega} \zeta\left( \e^{1-p+q} \frac{|u(T_n(x) ) - u(T_n(y) )|^p}{|T_n(x)-T_n(y)|^q}\right) \eta_{\e}(|T_n(x)- T_n(y)|) \d \rho (x) \d \rho (y).
\nc
	\end{align*}
Analogously
\begin{multline*}
 \mathcal{GMS}_{\e,n}(u) 
 \\ \geq 
\inf_{x \in \Omega} \left| \frac{\rho_n}{\rho} \right|^2 
\frac{1}{\e} \int_{\Omega} \int_{\Omega} \zeta\left( \e^{1-p+q} \frac{|u(T_n(x) ) - u(T_n(y) )|^p}{|T_n(x)-T_n(y)|^q}\right) \eta_{\e}(|T_n(x)- T_n(y)|) \d \rho (x) \d \rho (y). 
\end{multline*}
%Now you can just use the old argument as the prefactor converges to one as $n \to \infty$

In the light of the above computation, we consider three different functionals, which, as we show,  share the same $\Gamma$-limit. We consider
	\begin{align*}
	\mathcal{GMS}_{\e,n}(u)&:=\frac{1}{\e} \frac{1}{n^2}\sum_{i,j=1}^n \zeta\left( \e^{1-p+q} \frac{|u(x_i) - u(x_j)|^p}{|x_i-x_j|^q}\right) \eta_{\e}(|x_i - x_j|)\\
		\overline{\mathcal{GMS}}_{\e,n}(u)&:=  \frac{1}{\e} \int_{\Omega} \int_{\Omega} \zeta\left(  \frac{|u(T_n(x) ) - u(T_n(y) )|^p}{\e^{p-q-1} |T_n(x)-T_n(y)|^q}\right) \eta_{\e}(|T_n(x)- T_n(y)|) \d \rho (x) \d \rho (y)\\
	\mathcal{GAMS}_{\e,n}(u)&:= \frac{1}{\e} \int_{\Omega} \int_{\Omega}  \zeta\left(\frac{|u(T_n(x))-u(T_n(y))|^p}{\e^{p-q-1}|T_n(x)-T_n(y)|^q}\right) \eta_{\e}(|T_n(x)-T_n(y)|)\rho(x)^2\d x\d y.
	\end{align*}
 Notice that since $\frac{\rho_n}{\rho}$ uniformly converges to $1$ (by Lemma \ref{lem: improved convergence})
	\begin{equation}\label{eqn:key equality}
	\Glim_{n\rightarrow {+\infty}} \mathcal{GMS}_{\e_n,n}=\Glim_{n\rightarrow {+\infty}} \overline{\mathcal{GMS}}_{\e_n,n}.
	\end{equation}
%\end{remark}
\nc

The next Lemma shows that, in case of a compactly supported kernel $\eta$, the auxiliar energy $\mathcal{GAMS}$ is asymptotically equivalent to $\overline{\mathcal{GMS}}$. 
\begin{lemma}\label{lem I can kill the remainder}
Let $\{\e_n\}_{n\in \N}$ be any sequence satisfying \eqref{eqn: condition on eps}. Le $\{u_n \}_{n\in \N}$ be a sequence in $L^1(\Omega;\mu_n)$ and assume that the kernel $\eta$ is compactly supported.
%Suppose that
%	\begin{equation}\label{eqn: bound energy lem remainder}
%	\sup_{n\in \N}\{\mathcal{GMS}_{\e_n,n}(u_n)\}\leq M<\infty.
%	\end{equation}
Then
\begin{align*}
	 \liminf_{n\rightarrow +\infty} \overline{\mathcal{GMS}}_{\e_n,n}(u_n)&=\liminf_{n\rightarrow +\infty}\mathcal{GAMS}_{\e_n,n}(u_n)\\
	 \limsup_{n\rightarrow +\infty} \overline{\mathcal{GMS}}_{\e_n,n}(u_n)&=\limsup_{n\rightarrow +\infty}\mathcal{GAMS}_{\e_n,n}(u_n)
	\end{align*}
\end{lemma}
\begin{proof}
We introduce the following definitions, subordinated to the proof in order to compress the notations
	\begin{align*}
	\mathcal{R}_n(u_n)&:=\\
	\frac{1}{\e_n}&\int_{\Omega\times \Omega} \zeta\left(\frac{|u_n(T_n(x))-u_n(T_n(y))|^p}{\e_n^{p-q-1}|T_n(x)-T_n(y)|^q}\right) \eta_{\e_n}(|T_n(x)-T_n(y)|)[\rho(y)-\rho(x)]\rho(x)\d x\d y.
	\end{align*}
Notice that
	\[
	\overline{\mathcal{GMS}}_{\e_n,n}(u_n)=\mathcal{GAMS}_{\e_n,n}(u_n)+\mathcal{R}_{n}(u_n).
	\]
Moreover, due to the properties of $\rho$, we can find a constant $L>1$ such that
	\[
		\frac{1}{L}\overline{\mathcal{GMS}}_{\e_n,n}(u) \leq 	\mathcal{GAMS}_{\e_n,n}(u) \leq 	L \overline{\mathcal{GMS}}_{\e_n,n}(u)
	\]
for all $u\in L^1(\Omega;\mu_n)$ and for all $n\in \N$. In particular
	\begin{align*}
	\overline{\mathcal{GMS}}_{\e_n,n}(u_n) = 0 \ \ &\Leftrightarrow \ \ \mathcal{GAMS}_{\e_n,n}(u_n)=0\\
	\overline{\mathcal{GMS}}_{\e_n,n}(u_n) = +\infty \ \ &\Leftrightarrow \ \ \mathcal{GAMS}_{\e_n,n}(u_n)=+\infty.
	\end{align*}
For the purpose of our proof we can hence restrict to 
	\[
	0< \mathcal{GAMS}_{\e_n,n}(u_n)<+\infty \ \ \ \ \text{for all $n\in \N$}.
	\]
since otherwise the seek equality trivially hold. With this assumption in mind we consider the finite ratio $\mathcal{D}_n:=\mathcal{R}_n(u_n)/ \mathcal{GAMS}_{\e_n,n}(u_n)$ so that
	\begin{equation}\label{eqn: first one}
	\overline{\mathcal{GMS}}_{\e_n,n}(u_n)=\mathcal{GAMS}_{\e_n,n}(u_n)(1+\mathcal{D}_n).
		\end{equation}
Since $\eta$ is compactly supported, and since
	\[
	|T_n(x)-T_n(x)|\geq |x-y|+2\ell_n
	\]
we have that $\eta_{\e_n}(|T_n(x)-T_n(x)|)=0$ for all $x,y$ such that $|x-y|\geq M \e_n $ for some $M\in \R$.  Moreover if $|x-y|<M\e_n$ we have $|\rho(x)-\rho(y)|\leq \text{Lip}(\rho) M\e_n$. Thus, since $\rho$ is bounded from below, we get
\[
	\mathcal{R}_{n}(u_n)\leq \frac{\text{Lip}(\rho)M\e_n}{c}\mathcal{GAMS}_{\e_n,n}(u_n).
\]
Hence
	\begin{align*}
	\lim_{n\rightarrow +\infty} \mathcal{D}_n=0
	\end{align*}
that, combined with \eqref{eqn: first one} completes the proof.
\end{proof}
\nc

\subsection{The liminf inequality} \label{sec:liminf}
This subsection is devoted to the proof of the liminf inequality claim of the $\Gamma$-convergence of Theorem 
\ref{thm main: G conv}:

\begin{proposition}\label{prop: liminf inequality}
Let $\{\e_n\}_{n\in \N}$ be any sequence satisfying \eqref{eqn: condition on eps}. Let $u_n\in L^1(\Omega; \mu_n)$, $u\in  L^1(\Omega)$ such that $(\mu_n,u_n)\rightarrow (\mu,u)$ in $TL^1$. Then 
	\[
	\liminf_{n\rightarrow \infty} \mathcal{GMS}_{\e_n,n}(u_n) \geq MS_{\eta,\zeta}(u;\rho).
	\]
\end{proposition}

The proof relies on the following result, which can be obtained following ideas of \cite[Corollary 3.3]{gobbino2001finite}. Its proof is provided in  the Appendix \ref{sec:app}. 
\begin{lemma}\label{lem: hope is true}
Let $A\subset \R$ be a finite union of intervals and $f\in C_0^1(A)\cap C_0(\overline{A})$ be such that 
	\[
	0< c \leq\min_{x\in \overline{A}} \{f(x)\}\leq \max_{x\in   \overline{A}} \{f(x)\} \leq C<\infty.
	\]  
Define for every $\xi \in \R^d$ the one-dimensional functional
	\begin{equation}\label{eqn: useful 1}
	 E_{\delta}(u;\xi,A):=\frac{1}{\delta} \int_{A } \zeta\left(\frac{|u(x+\delta|\xi|)-u(x)|^p}{\delta^{p-1}|\xi|^q}\right) f(x) \d x 
	\end{equation}
and
	\begin{equation}\label{eqn: useful 2}
	\begin{split}
	E(u;\xi,A):=\zeta'(0)|\xi|^{p-q} \int_{A} &|u'(x)|^{p} f \d x+ \Theta|\xi| \int_{S_u\cap A} f(y) \d \H^{0}(y)
	\end{split}x
	\end{equation}
for $u\in SBV^p(\R)$. Then 
%	\begin{itemize}
%	\item[i)] $E_{\delta}(u;\xi,A)\leq  E(u;\xi,A)$ for all $u\in SBV(A)$;
\begin{equation}\label{eqn: liminf one d}
 \liminf_{\delta\rightarrow 0} E_{\delta}(u_{\delta};\xi,A) \geq E(u;\xi,A)
 \end{equation}
for all $\{u_{\delta}\}_{\delta>0}\subset L^1(A), u \in SBV^p(A)$ such that $ u_{\delta} \rightarrow u $ in $L^1$.
%In particular, the functional $E_{\delta}$ Gamma converges to $E$ in the $L^1$ topology.
\end{lemma}
We also need the following lemma.
\begin{lemma}\label{Lemma: tech sectional}
Let $u_n:\{x_1,\ldots, x_n\} \rightarrow \R$ be a sequence of function $L^p( \mu_n)$ such that $(\mu_n,u_n) \rightarrow (\mu,u)$ in $TL^1$. 
% Dejan ??:did we say that we omit measures when clear from context?
Let $\tilde{\mu}_n$ be the sequence of measures yield by Lemma \ref{lem: improved convergence} and consider $T_n:\R^d\rightarrow \R^d$ to be the associated transport maps between $\tmu_n$ and $\mu_n$. Fix $\xi \in \R^d$, $z\in \xi^{\perp}$ and define the sectional functions
	\[
	t \mapsto \tilde{u}_{n,\xi}(t;z):=u_n\circ T_n \left(z+t\frac{\xi}{|\xi|}\right).
	\]
Then, for $\H^{d-1}-$a.e. $z\in \xi^{\perp}$, the function $\tilde{u}_{\xi,n}(t;z)$ converge to $u_{\xi}(t;z)$ in $L^1(\R; \rho_{\xi}(\cdot;z) \L^1)$.
\end{lemma}
\begin{proof}
Assume by contradiction that there is a vector $\xi$ and a set $E\subset \xi^{\perp}$ with $\H^{d-1}(E)>0$ such that, for all $z\in E$ it holds
	\begin{align*}
	\liminf_{\e\rightarrow 0^+} \int_{\R} |\tilde{u}_{\xi,n}(t;z)- u_{\xi}(t;z)| \rho_{\xi}(t;z) \d t > 0 .
	\end{align*}
Thanks to the $TL^1$ convergence of the sequence $(\mu_n,u_n)$ and to the properties of our measures $\tilde{\mu}_n$ given by Lemma \ref{lem: improved convergence} we can infer $(u_n\circ T_n,\tilde{\mu}_n)\rightarrow (\mu,u)$ in $TL^1$. In particular thanks also to \cite[Assertion 5, Proposition 3.12]{GTS16}) we know that
	\begin{equation}\label{eqn: tec1}
	\lim_{n\rightarrow \infty} \int_{\R^d} |u_n \circ T_n(x) - u  (x)| \rho (x) \d x =0.
	\end{equation}
%where $\rho_n:=\frac{\d \tilde{\mu}_n}{\d \L^n}$. In particular, by virtue of Lemma \ref{lem: improved convergence} we have also
%\begin{equation}
%	\lim_{n\rightarrow \infty} \int_{\R^d} |u_n \circ T_n(x) - u  (x)| \rho(x) \d x =0.
%	\end{equation}
We can rearrange and estimate the above integral as
	\begin{align*}
	\int_{\R^d} |u_n \circ & T_n(x) - u (x)| \rho (x) \d x=\int_{0}^{\infty} \d t \int_{\left\{\frac{x\cdot \xi}{|\xi|}=t \right\}} |u_n \circ T_n(x) - u (x)| \rho (x) \d \H^{n-1}(x)\\
	&=\int_{0}^{\infty} \d t \int_{\xi^{\perp} } |\tilde{u}_{\xi,n}(t;z) - u_{\xi}(t;z)| \rho_{\xi}(t;z) \d \H^{n-1}(z)\\
	&\geq \int_{E} \d \H^{n-1}(z) \int_{0}^{\infty}  |\tilde{u}_{\xi,n}(t;z) - u_{\xi}(t;z)| \rho_{\xi}(t;z)  \d t.
	\end{align*}
In particular, Fatou's Lemma implies
\[
\liminf_{n\rightarrow \infty} \int_{\R^d} |u_n \circ  T_n(x) - u (x)| \rho (x) \d x>0
\]
which contradicts \eqref{eqn: tec1}.
\end{proof}  	

\begin{remark} \label{rem:drill}
The strategy adopted to prove the $\liminf$ inequality is based on the integral form of the functional $\mathcal{GMS}$ (more precisely on its asymptotic counterpart $\mathcal{GAMS}$). Substantially we exploit such an integral form and the monotonicity of $\eta,\zeta$ to compare our energy with the energy 
	\[
	\int_{\R^d} \eta(|\xi|) \d \xi \int_{\xi^{\perp}} E_{\e_n}(\tilde{u}_{\xi,n}(\cdot;z);\xi;\Omega) \d\H^{d-1}(z)
	\]
in order to apply Lemma \ref{Lemma: tech sectional} and Lemma \ref{lem: hope is true}. To achieve this goal we need to compare $|T_n(x)-T_n(y)|$ and $|x-y|$, as well as $\frac{1}{|T_n(x)-T_n(y)|}$ and $\frac{1}{|x-y|}$. Notice that $|T_n(x)-T_n(y)|\approx |x-y|\pm2\ell_n$ and in particular due to the fact that $\ell_n \ll \e_n$ it is not difficult to compare $|T_n(x)-T_n(y)|$ to $|x-y|$ for couples such that $|x-y|\approx \e_n$. The problems arise for all those points $x,y$ that lie very close to one another. The strategy we adopt to overcome this difficulty is to create a hole of fixed size $r$ around the origin in the kernel $\eta$.  That is to replace $\eta$ by 
 \[ \eta^r(t):=\eta(t) (1-\ca_{[0,r)}(t)) \]
 and thus  to neglect all the, small, contributions we cannot compare and then progressively recover the full energy by shrinking the hole in a limit process at the end of our proof. In the sequel we will often write
 $\mathcal{GMS}_{\e_n,n}(u_n;r),\mathcal{GAMS}_{\e_n,n}(u_n;r)$to denote the corresponding energies having $\eta^r$ in place of $\eta$.
 \end{remark}
\begin{proof}[Proof of Proposition \ref{prop: liminf inequality}]
We can assume, without loss of generality, that
	\begin{equation}\label{eqn: uniform bound on energy}
	\sup_{n\in \N}\{\mathcal{GMS}_{\e_n,n}(u_n)\}<\infty.
	\end{equation}
Notice that it is enough to assume that $\eta$ is compactly supported. Indeed we can always replace $\eta$ with $\eta_R:=\eta(t)\ca_{[0,R)}(t)$ and notice that, by meaning of Lemma \ref{lem I can kill the remainder} and \eqref{eqn:key equality} we have
	\begin{align*}
	\liminf_{n\rightarrow +\infty} \mathcal{GMS}_{\e_n,n}(u_n)\geq \liminf_{n\rightarrow +\infty} \overline{\mathcal{GMS}}_{\e_n,n}(u_n)\geq\liminf_{n\rightarrow +\infty} \mathcal{GAMS}_{\e_n,n}(u_n;\eta_R)
	\end{align*}
where $\mathcal{GAMS}_{\e_n,n}(u_n;\eta_R)$ denotes the usual energy $\mathcal{GAMS}_{\e_n,n}$ with $\eta_R$ in place of $\eta$. In particular if we can prove that, for compactly supported kernel, it holds
	\[
	\liminf_{n\rightarrow +\infty} \mathcal{GAMS}_{\e_n,n}(u_n;\eta_R)\geq MS_{\eta_R,\zeta}(u)
	\] 
then the continuity of the constants in $MS$ allows us to send $R$ to infinity and recover
	\begin{align*}
	\liminf_{n\rightarrow +\infty} \mathcal{GMS}_{\e_n,n}(u_n)\geq MS_{\eta,\zeta}(u).
	\end{align*}
We thus focus on proving the Theorem for a compactly supported kernel $\eta$. \\

With this assumption in mind we invoke again Lemma \ref{lem I can kill the remainder} and \eqref{eqn:key equality} to infer
		\begin{equation}\label{eqn: minoration}
			\liminf_{n\rightarrow\infty}	\mathcal{GMS}_{\e_n,n}(u_n) \geq \liminf_{n\rightarrow \infty} \mathcal{GAMS}_{\e_n,n}(u_n;r) \ \ \ \text{for all $r>0$}.
		\end{equation}
For the reader's convenience in what follows we write $\e$ and $\ell:=\ell:=\|T_n-Id\|_{\infty}$ in place of $\e_n$  and $\ell_n$ by omitting the dependence on $n$. Thanks to \eqref{eqn: decay on epsilon} we have
	\begin{equation}
	\frac{|x-y|}{\e} - \frac{2\ell}{\e}\leq \frac{|T_n(x)-T_n(y)|}{\e}\leq \frac{|x-y|}{\e}+\frac{2\ell}{\e}.
	\end{equation}
On the set  $\{  |T_n(x)-T_n(y)| \geq r\e\}$ we thus have
	\[
    \frac{|x-y|}{1+\frac{2\ell}{\e r}} \leq  |T_n(x)-T_n(y)| \leq     \frac{|x-y|}{1- \frac{2\ell}{\e r}}.
	\]
In particular, since the function $t\rightarrow \zeta(a/t^q)\eta_{\e}^r(t)$ is non-increasing for all $\e,r,a\in \R$, we conclude that
\begin{align*}
\mathcal{GAMS}_{\e,n}&(u;r) \geq \frac{1}{\e} \int_{\Omega \times \Omega} \zeta\left(  \frac{|u_n(T_n(x))-u_n(T_n(y))|^p}{\frac{\e^{p-q-1}}{\left(1-\frac{2\ell}{\e r}\right)^{q}}|x-y|^q}\right) \eta_{\e}^r\left(\frac{|x-y|}{1-\frac{2\ell}{\e r}}\right)\rho(x)^2 \d x \d y\\
&= \frac{1}{\e} \int_{\Omega \times \Omega} \zeta\left(  \frac{|u_n(T_n(x))-u_n(T_n(y))|^p}{\frac{\e^{p-q-1}}{\left(1-\frac{2\ell}{\e r}\right)^{q-p+1} } \left(1-\frac{2\ell}{\e r}\right)^{1-p}|x-y|^q}\right) \eta^r\left(\frac{|x-y|}{\e \left(1-\frac{2\ell}{\e r}\right) }\right) \e^{-d} \rho(x)^2 \d x \d y\\
& \geq \frac{1}{\e} \int_{\Omega \times \Omega} \zeta\left(  \frac{|u_n(T_n(x))-u_n(T_n(y))|^p}{\left[\e \left(1-\frac{2\ell}{\e r}\right) \right]^{p-q-1}  |x-y|^q}\right) \eta^r\left(\frac{|x-y|}{\e \left(1-\frac{2\ell}{\e r}\right) }\right) \e^{-d} \rho(x)^2 \d x \d y.\\
% & \geq \frac{\left(1-\frac{2l}{\e r}\right)^{d+1}}{\e \left(1-\frac{2l}{\e r}\right)} \int_{\Omega \times \Omega} \zeta\left(  \frac{|u(T_n(x))-u(T_n(y))|^p}{\left[\e \left(1-\frac{2l}{\e r}\right) \right]^{p-q-1}  |x-y|^q}\right) \eta^r\left(\frac{|x-y|}{\e \left(1-\frac{2l}{\e r}\right) }\right) \e^{-d}\left(1-\frac{2l}{\e r}\right)^{-d} \rho(x)^2 \d x \d y\\
\end{align*}
By setting
	\[
	\delta=\delta_{\e} :=\e \left(1- \frac{2\ell}{\e r}\right)
	\]
and 
\begin{align*}
 \mathcal{G}_{\delta,n}(u_n,r)&:=\frac{1}{\delta} \int_{\Omega \times \Omega} \zeta\left(  \frac{|u_n(T_n(x))-u_n(T_n(y))|^p}{\delta^{p-q-1}  |x-y|^q}\right) \eta_{\delta}^r\left(|x-y|\right)\rho(x)^2 \d x \d y\\
	=\frac{1}{\delta} \int_{\R^d}& \eta^r\left(|\xi|\right) \d \xi  \int_{ \Omega \cap (\Omega-
\delta \xi)} \zeta\left(  \frac{|u_n(T_n(x+\delta \xi))-u_n(T_n(x))|^p}{\delta^{p-1}  |\xi|^q}\right)  \rho(x)^2 \d x
	\end{align*}
(where we applied the change of variable \eqref{eqn: cov})
we obtain
\begin{equation}\label{eqn: first key estimate liminf}
\left(1-\frac{2\ell}{\e r}\right)^{-d-1}\mathcal{GAMS}_{\e,n}(u;r) \geq  \mathcal{G}_{\delta,n}(u_n,r).
\end{equation}
Notice that, since $\ell\ll\e$, we have $\delta=\delta_{\e} \rightarrow 0$ and $\left(1-\frac{2\ell}{\e r}\right)^{-d-1}\rightarrow 1$. We now focus our attention on the energy $\mathcal{G}_{\delta,n}$ and more precisely on 
	\begin{align*}
	\mathcal{E}_{\delta,n}(u_n;\xi)&:=\frac{1}{\delta} \int_{ \Omega \cap (\Omega-
\delta \xi)} \zeta\left(  \frac{|u_n(T_n(x+\delta \xi))-u_n(T_n(x))|^p}{\delta^{p-1}  |\xi|^q}\right)  \rho(x)^2 \d x.
	\end{align*}
Clearly
	\begin{equation}\label{eqn: all pieces together}
	\mathcal{G}_{\delta,n}(u_n,r)=\int_{\R^d}\eta^r(|\xi|) \mathcal{E}_{\delta,n}(u_n;\xi)\d \xi.
	\end{equation}
Fix $\xi$ and consider an open bounded set with finite perimeter and smooth boundary $A\subset\subset \Omega$ and notice that, for $\delta$ small enough, we have
		\[
		A\subset \Omega\cap(\Omega-\delta\xi).
		\]
In particular
\begin{align*}
\mathcal{E}_{\delta,n}(u_n;\xi)&\geq\frac{1}{\delta} \int_{A} \zeta\left(  \frac{|u_n(T_n(x+\delta \xi))-u_n(T_n(x))|^p}{\delta^{p-1}  |\xi|^q}\right)  \rho(x)^2 \d x \\
&= \frac{1}{\delta} \int_{\xi^{\perp}} \d \H^{n-1}(z) \int_{ [A]_{z}} \zeta\left(  \frac{|\tilde{u}_{\xi,n} (t+\delta|\xi |;z) -\tilde{u}_{\xi,n}(t;z)|^p}{\delta^{p-1}  |\xi|^q}\right)  \rho_{\xi}(t;z)^2 \d t \\
&=\int_{\xi^{\perp}} E_{\delta}(\tilde{u}_{\xi,n}(\cdot;z);\xi,[A]_{z})\d \H^{d-1}(z)
\end{align*}
Thanks to the co-area formula we have that for, $\H^{d-1}$ almost every $z\in \xi^{\perp}$, the set $[A]_{z}$ must be a finite union of open intervals. Moreover, by applying Lemma \ref{Lemma: tech sectional} and the boundedness of $\rho$ we conclude $\tilde{u}_{\xi,n}(\cdot;z)\rightarrow u_{\xi}(\cdot;z)$ in $L^1$ for $\H^{d-1}$ almost every $z\in \xi^{\perp}$. Thus, by applying Lemma \ref{lem: hope is true} and Fatou's Lemma we reach
	\begin{align}
	\liminf_{n\rightarrow\infty} \mathcal{E}_{\delta,n}(u_n;\xi)\geq |\xi|^{p-q} \zeta'(0) & \int_{\xi^\perp}  \int_{[A]_{z}}  |u'_{\xi}(t;z)|^p \rho_{\xi}(t;z)^2\d t \d \H^{d-1}(z)+ \nonumber \\
	& + \Theta|\xi| \int_{\xi^{\perp} }\int_{S_{ u_{\xi}(\cdot;z) }\cap [A]_{z} }\rho_{\xi}(t,z)^2\d\H^{0}(t) \d \H^{d-1}(z).\label{eqn: liminf inside}
	\end{align}
%We re-arrange once again the left-hand side member of \eqref{eqn: liminf inside} as
%	\begin{align*}
%	 |\xi|^{p-q} \zeta'(0)&\int_{\xi^\perp}   \int_{A_{z,\xi}}  |u'_{\xi}(t;z)|^p \rho_{\xi}(t;z)^2 \d t \d \H^{d-1}(z)+   \Theta |\xi|\int_{\xi^{\perp} }\int_{S_{ u_{\xi}(\cdot;z) }\cap A_{z,\xi} }\rho_{\xi}(t,z)^2\d\H^{0}(t) \\
%	  =|\xi|^{p-q}  & \zeta'(0)\int_{\xi^\perp} \d \H^{d-1}(z) \int_{A_{z,\xi}}  \left| \nabla u\left(z+t\frac{\xi}{|\xi|}\right)\cdot \frac{\xi}{|\xi|}\right|^p \rho\left(z+t\frac{\xi}{|\xi|}\right)^2\d t\\  
%	  & \ \ \ \ + \Theta|\xi| \int_{\xi^{\perp} } \d \H^{d-1}(z)  \int_{S_{ u_{\xi}(\cdot;z) }\cap A_{z,\xi} }\rho\left(z+t\frac{\xi}{|\xi|}\right)^2\d\H^{0}(t)\\
%	  =|\xi|^{p-q} & \zeta'(0) \int_{A} \left|\nabla u(x)\cdot \frac{\xi}{|\xi|}\right|^p \rho(x)^2 \d x +   \Theta|\xi| \int_{S_u \cap A}\left|N_u(y)\cdot \frac{\xi}{|\xi|} \right|\rho(y)^2\d\H^{d-1}(y) 	
%	\end{align*}
%where $N_u(y)$ is any vector field normal to $S_u$.
%The above last equality follows thanks to the Coarea formula and to the well known relation between the one-dimensional slices of an SBV function and the total length of its jump (see for instance \cite{ambrosio1989compactness}). 
In particular, from \eqref{eqn: MS almost sliced}, \eqref{eqn: liminf inside}, \eqref{eqn: all pieces together} and Fatou's Lemma we get
	\begin{align*}
	\liminf_{n\rightarrow\infty} \mathcal{G}_{\delta,n}(u_n;r) \geq \zeta'(0) \int_{\R^d}  |\xi|^{p-q}& \eta^r(|\xi|)\d \xi \int_{A}  \left|\nabla u(x)\cdot \frac{\xi}{|\xi|}\right|^p \rho(x)^2 \d x \\
	 + &  \Theta \int_{\R^d}|\xi|\eta^r(|\xi|) \d \xi \int_{S_u \cap A}\left|N_u(y)\cdot \frac{\xi}{|\xi|} \right|\rho(y)^2\d\H^{d-1}(y) 
	\end{align*}
which, being valid for all $A\subset \subset \Omega$ , allows us to conclude
	\begin{align*}
	\liminf_{n\rightarrow\infty} \mathcal{G}_{\delta,n}(u_n;r) \geq \zeta'(0) \int_{\R^d}  |\xi|^{p-q}& \eta^r(|\xi|)\d \xi \int_{\Omega}  \left|\nabla u(x)\cdot \frac{\xi}{|\xi|}\right|^p \rho(x)^2 \d x \\
	 + &  \Theta \int_{\R^d}|\xi|\eta^r(|\xi|) \d \xi \int_{S_u }\left|N_u(y)\cdot \frac{\xi}{|\xi|} \right|\rho(y)^2\d\H^{d-1}(y).
	\end{align*}
Thanks to \eqref{eqn: MS almost sliced}, we have that
 	\begin{align*}
	\liminf_{n\rightarrow \infty} \mathcal{GAMS}_{\e_n,n}(u_n;r) \geq \liminf_{n\rightarrow\infty} \mathcal{G}_{\delta_n,n}(u_n;r) \geq MS_{\eta^r,\zeta}(u;\rho)
	\end{align*}
%and, by noticing that $\vartheta_{\eta^r}(p,q)\rightarrow \vartheta_{\eta}(p,q)$, $\sigma_{\eta^r}\rightarrow \sigma_{\eta}$,  
and by exploiting once again the continuity in $r$ of the constants $\vartheta_{\eta^r}(p,q), \sigma_{\eta^r}$ we can take the limit as $r\rightarrow 0$. This, considered also \eqref{eqn: minoration}, achieves the proof.
\end{proof}

\subsection{The limsup inequality}

We now prove the limsup inequality part of Theorem \ref{thm main: G conv}:
\begin{proposition}\label{prop: limsup}
Let $\{\e_n\}_{n\in \N}$ be any sequence satisfying \eqref{eqn: condition on eps}. Let $u\in  SBV(\Omega)$. Then there exists $u_n\in L^1(\Omega;\mu_n)$  such that $(\mu_n,u_n)\rightarrow (\mu,u)$ in $TL^1$ and
	\[
	\limsup_{n\rightarrow \infty} \mathcal{GMS}_{\e_n,n}(u_n) \leq MS_{\eta,\zeta}(u;\rho).
	\]
\end{proposition}
We prove the proposition by  providing a recovery sequence for regular functions and argue by approximation.
We start by showing how to recover the energy of a function $u\in SBV(\R^d)$ having the following properties:
\begin{itemize}
\item[(H1)] $\overline{S_u}$ is the union of a finite
number of $(d-1)$-dimensional simplexes, $\H^{d-1}(\overline{S_u}\setminus S_u)=0$;
\item[(H2)] $u\in C^{\infty}(\Omega\setminus \overline{S_u})\cap W^{1,\infty}(\Omega\setminus \overline{S_u})$;
\item[(H3)] $MS_{\eta,\zeta}(u)<\infty$
\end{itemize}
We then use of the following density theorem which is a consequence of a well known result of Cortesani and Toader \cite{cortesani1999585}.
\begin{theorem}\label{theorem CT}
Let $\Omega$ be an open bounded set with Lipschitz boundary and $u\in SBV(\Omega)^p$. Then there exists a sequence of function $u_j\in SBV(\Omega)^p$ satisfying (H1),(H2) and (H3) such that:
\begin{itemize}
\item[(i)]
$\displaystyle	\limsup_{j\rightarrow \infty} \int_{S_{u_j}} \rho(y)^2 \d\H^{d-1}(y) \leq   \int_{S_{u}} \rho(y)^2 \d\H^{d-1}(y)$;
\item[(ii)] $\nabla u_j  \stackrel{L^p}{\longrightarrow} \nabla u$ and $ u_j \stackrel{L^1}{\longrightarrow}  u$,
where $\nabla u$ is, as before, the absolutely continuous part of the gradient $Du$. 
\end{itemize}
\end{theorem}

The following Lemma is used to compare the energy of $u(T_n)$ with the energy of $u$.
\begin{lemma}\label{lem: key size lemma}
Let $\{x_i\}_{i=1}^n$ be a sequence of i.i.d. points chosen according to the density $\rho$. Let $\tilde{\mu}_n$ be the measures provided by Lemma \ref{lem: improved convergence} and $T_n:\Omega \rightarrow \{x_1,\ldots,x_n\}$ be the transport maps  between $\tmu_n$ and $\mu_n$. Let $\{\e_n\}_{n\in \N}$ be a sequence satisfying \eqref{eqn: condition on eps}. For any $u\in SBV(\Omega)$ satisfying (H1)-(H2), $\xi \in\R^d$, $\ell>0$ and $\e>0$ define
	\begin{align*}
	(S_u)_{\ell}&:=\{ x\in \Omega \ : \ d(x,S_u) \leq \ell\},\\
	(S_u- \e \xi)_{\ell}&= \{ x\in \Omega \ : \ d(x,S_u-\e \xi) \leq \ell\},\\
	D(\e_,\ell)&:=(S_u)_{\ell} \cup (S_u- \e \xi)_{\ell }.
	\end{align*}
Let $\ell_n=\|T_n-Id\|_{\infty}$. Then for any $x\in \Omega\setminus D(\e_n,\ell_n)$ it holds
\begin{equation}\label{eqn: comparison}	
	|u(T_n(x+\e_n \xi)) - u(T_n(x))|\leq |u(x+\e_n \xi) - u(x)|+2\ell_n \|\nabla u\|_{\infty}.
\end{equation}
Moreover
\begin{equation}\label{eqn: size of the bad points}
\lim_{n\rightarrow \infty} \frac{|D(\e_n,\ell_n)|}{\e_n} =0.
\end{equation}
\end{lemma}
The proof is presented in Appendix \ref{sec:app2}.

\begin{proposition}\label{prop: limsup reg}
Let $\{\e_n\}_{n\in \N}$ be any sequence satisfying \eqref{eqn: condition on eps}. Let $u\in SBV(\Omega)$ satisfying (H1)-(H3). Then there exists a sequence of function $\{u_n\}_{n \in \N}\subset L^1(\Omega;\mu_n)$ such that $(\mu_n,u_n)\rightarrow (\mu,u)$ in $TL^1$ and
	\begin{equation}
	\limsup_{n\rightarrow \infty} \mathcal{GMS}_{\e_n,n}(u_n)\leq  MS_{\eta,\zeta}(u;\rho).
	\end{equation}
\end{proposition}
\begin{proof}
We start by noticing that, since $\Omega$ has Lipschitz boundary, for any polyhedral set $A\supset \supset \Omega$, there exists an extension $\tilde{u}\in SBV(\R^d)$, still satisfying hypothesis (H1)-(H3), such that $|D\tilde{u}|(\partial \Omega)=0$ and $\tilde{u}=0$ outside $A$ (see for example \cite[Proposition 3.21]{ambrosio2000functions}). We thus fix  $A\supset \Omega$ and consider such an extension (still denoted, with a slight abuse of notation, by $u$). We also extend $\rho(x)=0$ on $\R^d\setminus \overline{\Omega}$.  For every $n$, 
consider $\tilde \mu_n$ of Lemma \ref{lem: improved convergence}. Let $T_n$, as before be the $d_\infty$ optimal transport map between $\tmu_n$ and $\mu_n$ and $\ell_n = \|T-I_d\|_{L^\infty(\Omega)}$. 
Notice that, since $MS_{\eta,\zeta}(u;\rho)<\infty$ we can infer that, for $\L^d$-a.e. $\xi\in \R^d$ and $\H^{d-1}$-a.e. $z\in \xi^{\perp}$ it holds
	\begin{align}
		\int_{\R} |u_{\xi}'(t;z)|\rho_{\xi}(t;z)^2\d t + \int_{S_{u_{\xi}(\cdot;z)}}  \rho_{\xi}(t;z)^2   \d \H^0(t)&<\infty \label{eqn: cond1}.
	\end{align}
We define $u_n:\{x_i\}_{i=1}^n\rightarrow \R$ as
\begin{equation}
	u_n(x_i)=\left\{
	\begin{array}{ll}
	u(x_i) & \text{if $x_i\in \Omega\setminus S_u$};\\
	u^+(x_i) & \text{if $x_i\in S_u$}
	\end{array}
	\right.
\end{equation}
where $u^+$ is defined in \eqref{eq:uplus}. We now divide the proof in two steps.\\
\text{}\\
\textbf{Step one:}\textit{ $\limsup$ bound on $\mathcal{GAMS}$}. We first prove that the $\limsup$ bound holds for  the auxiliary energy 
	\[
	\mathcal{GAMS}_{\e_n,n}(u_n)=\frac{1}{\e_n}\int_{\Omega\times \Omega} \zeta\left(\frac{|u_n(T_n(x))-u_n(T_n(y))|}{\e_n^{p-q-1}|T(x)-T_n(y)|^q}\right) \eta_{\e_n}(|T_n(x)-T_n(y)|)\rho(x)^2\d x\d y
	\]
From now on, we will omit, as in other proofs, the dependence on $n$ of $\e_n$ and $\ell_n$. 
Define, for $t \in \R$, the kernel  $\overline{\eta}(t):=\eta(\max\{t-2\ell/\e, 0\})$ (where we are omitting to explicitly denote the dependence on $\e$). Since
	\[
	|x-y|-2\ell\leq |T_n(x)-T_n(y)|\leq |x-y|+2\ell
	\] 
and since $\eta$ is non-increasing, we deduce
	\begin{align*}
	\eta\left(\frac{|T_n(x)-T_n(y)|}{\e}\right)\leq \eta\left(\max\left\{\frac{|x-y|}{\e}-\frac{2\ell}{\e} , 0 \right\}\right) \leq \overline{\eta} \left(\frac{|x-y|}{\e}\right).
	\end{align*}
Since $\zeta$ is non-decreasing we have
\begin{align*}
\mathcal{GAMS}_{\e,n}(u_n)&\leq \frac{1}{\e} \int_{\Omega\times \Omega} \zeta\left(\e^{1-p+q}\frac{|u_n(T_n(x))-u_n(T_n(y))|^p}{(|x-y|-2\ell)_+^q}\right)\eta_{\e}(|x-y|)\rho(x)^2\d x \d y\\
&\leq \frac{1}{\e} \int_{\R^d} \d\xi \int_{\Omega }\zeta\left(\frac{|u_n(T_n(x+\e \xi))-u_n(T_n(x))|^p}{\e^{p-1}(|\xi|-2\ell/\e)_+^q}\right)\overline{\eta
} (|\xi|)\rho(x)^2 \d x
\end{align*}
For any fixed $\xi\in \R^d$ consider
\begin{align}
		\mathcal{F}_{\e,n}(u_n;\xi)&= \frac{1}{\e}  \int_{\Omega}\zeta\left(\frac{|u_n(T_n(x+\e \xi))-u_n(T_n(x))|^p}{\e^{p-1}(|\xi|-2\ell/\e)^q}\right) \rho(x)^2 \d x \nonumber
	\end{align}
so that
	\begin{equation}\label{eqn:limsup all togheter}
	\mathcal{GAMS}_{\e,n}(u_n)\leq \int_{\R^d} \overline{\eta}(|\xi|) \mathcal{F}_{\e}(u_n;\xi) \d \xi.
	\end{equation}
Let $D(\e,\ell)$ be the set defined in Lemma \ref{lem: key size lemma}. Notice that
	\begin{align}
		\mathcal{F}_{\e,n}(u_n;\xi)	=&\frac{1}{\e}  \int_{\Omega  \setminus D(\e,\ell)}\zeta\left(\frac{|u_n(T_n(x+\e \xi))-u_n(T_n(x))|^p}{\e^{p-1}(|\xi|-2\ell/\e)_+^q}\right) \rho(x)^2 \d x \label{eqn: part 0}\\
	&+\frac{1}{\e}  \int_{ D(\e,\ell)}\zeta\left(\frac{|u_n(T_n(x+\e \xi))-u_n(T_n(x))|^p}{\e^{p-1}(|\xi|-2\ell/\e)_+^q}\right) \rho(x)^2 \d x. \label{eqn: part 1}
	\end{align}
The second integral \eqref{eqn: part 1} can be easily estimated as
	\begin{equation}\label{eqn: rbad points decay limsup}
\frac{1}{\e}  \int_{D(\e,\ell) }\zeta\left(\frac{|u_n(T_n(x+\e \xi))-u_n(T_n(x))|^p}{\e^{p-1}(|\xi|- 2\ell/\e)_+^q}\right) \rho(x)^2 \d x\leq C^2\Theta \frac{ |D(\e,\ell)|}{\e} 
	\end{equation}
%\red (Marco, I changed $|\xi| + 2\ell/\e$ to $|\xi|- 2\ell/\e$ in the denominator above) \nc
which, thanks to \eqref{eqn: size of the bad points}, is decaying to $0$ as $n\rightarrow \infty$ (recall that $\e=\e_n\rightarrow 0$).\\

Let us now treat the first term \eqref{eqn: part 0} in the light of Lemma \ref{lem: key size lemma}. Since $\zeta$ is non-decreasing 
	\begin{align}
	 \frac{1}{\e}  \int_{\Omega  \setminus D(\e,\ell)	}&\zeta\left(\frac{|u_n(T_n(x+\e \xi))-u_n(T_n(x))|^p}{\e^{p-1}(|\xi|-2\ell/\e)_+^q}\right) \rho(x)^2 \d x \nonumber\\
	  &\leq \frac{1}{\e}  \int_{\Omega	}\zeta\left(\frac{ |\,|u(x+\e \xi)-u(x))| + 2\ell\|\nabla u\|_{\infty}|^p}{\e^{p-1}(|\xi|-2\ell/\e)_+^q}\right) \rho(x)^2 \d x \label{eqn: from transport to u}.
	   \end{align}
From now on, we use the same arguments of of the proof of Gobbino \cite[Theorem 3.4, Proposition 3.5, Theorem 3.6]{gobbino1998finite}, suitably adapted to our situation (see also \cite{gobbino2001finite}). By slicing along $\xi^{\perp}$ we get
	   \begin{align*}    
	    \frac{1}{\e}  \int_{\Omega  	} &\zeta\left(\frac{|\, |u(x+\e \xi)-u(x))| + 2\ell\|\nabla u\|_{\infty}|^p}{\e^{p-1}(|\xi|-2\ell/\e)_+^q}\right) \rho(x)^2 \d x\\
	    &=\frac{1}{\e}\int_{\xi^{\perp}} \d\H^{d-1}(z) \int_{[\Omega] _z}   \zeta\left(\frac{|\, |u_{\xi}(t+\e|\xi|; z)-u_{\xi}(t;z))| + 2\ell\|\nabla u\|_{\infty}|^p}{\e^{p-1}(|\xi|-2\ell/\e)_+^q}\right) \rho_{\xi}(t;z)^2 \d t.
	   \end{align*}
and, for the sake of clarity we introduce the notation
	\begin{align*}
	F_{\e}(u_{\xi}(\cdot;z);\xi,[\Omega]_z):=\frac{1}{\e} \int_{[\Omega] _z}   \zeta\left(\frac{|\,|u_{\xi}(t+\e|\xi|; z)-u_{\xi}(t;z))| + 2\ell\|\nabla u\|_{\infty}|^p}{\e^{p-1}(|\xi|-2\ell/\e)_+^q}\right) \rho_{\xi}(t;z)^2 \d t.
	\end{align*}
We define
	\begin{equation} \label{eq:Sz}
	[S]_z:=\{t\in [\Omega]_z \ : \ [t,t+\e |\xi|) \cap S_{u_{\xi}(\cdot;z) } \neq \emptyset \}
	\end{equation}
and, for a fixed $\xi \in \R^d\setminus \{0\}$, $z\in \xi^{\perp}$, we split once again $F_{\e}(u_{\xi}(\cdot;z);\xi,[\Omega]_z)$ as (notice that, for any fixed $\xi$ we can find $\e>0$ small enough such that $|\xi|-2\ell/\e>0$)
	\begin{align}	
	\frac{1}{\e}& \int_{[\Omega]_z\setminus [S]_z}   \zeta\left(\frac{|\,|u_{\xi}(t+\e|\xi|; z)-u_{\xi}(t;z))| + 2\ell\|\nabla u\|_{\infty}|^p}{\e^{p-1}(|\xi|-2\ell/\e)_+^q}\right) \rho_{\xi}(t;z)^2 \d t\nonumber\\
	&+ \frac{1}{\e} \int_{[S]_z}   \zeta\left(\frac{|\,|u_{\xi}(t+\e|\xi|; z)-u_{\xi}(t;z))| + 2\ell\|\nabla u\|_{\infty}|^p}{\e^{p-1}(|\xi|-2\ell/\e)_+^q}\right) \rho_{\xi}(t;z)^2 \d t \nonumber\\
	\leq& \ \frac{\zeta'(0)}{(|\xi|-2\ell/\e)^{q}}  \int_{[\Omega]_z\setminus [S]_z}   \frac{|\,|u_{\xi}(t+\e|\xi|; z)-u_{\xi}(t;z))| + 2\ell\|\nabla u\|_{\infty}|^p}{\e^{p}}\rho_{\xi}(t;z)^2 \d t\label{eqn: ac part limsup}\\
	&+ \Theta  \frac{1}{\e} \int_{[S]_z}   \rho_{\xi}(t;z)^2 \d t\label{eqn: jump part limsup}
	\end{align}
%\red Marco, there is a bit of a problem above when $|\xi|- 2\ell/\e < 0$. We discussed this over Skype this summer, and you mentioned that you will take care of it when the file returns to your hands. \nc 
Notice that, if $t\in [\Omega]_z\setminus [S]_z$, we have
	\begin{align*}
 \frac{|\,|u_{\xi}(t+\e|\xi|; z)-u_{\xi}(t;z))| + 2\ell\|\nabla u\|_{\infty}|^p}{\e^{p}}&\leq \e^{-p}  \left[  \int_{0}^{\e|\xi|} \left( |u_{\xi}'(t+\tau;z)|  + \frac{\ell}{\e|\xi|} \|\nabla u\|_{\infty}\right) \d \tau \right]^p\\
 \leq \e^{-1}|\xi|^{p-1}\int_{0}^{\e|\xi|} &\left[ |u_{\xi}'(t+\tau;z)| + \frac{\ell}{\e|\xi|} \|\nabla u\|_{\infty}\right]^p \d \tau\\
 =\e^{-1}|\xi|^{p-1}\int_{t}^{t+\e|\xi|}& \left[ |u_{\xi}'(s;z)| + \frac{\ell}{\e|\xi|} \|\nabla u\|_{\infty}\right]^p \d s
	\end{align*}
hence
\begin{align}
&\frac{\zeta'(0)}{(|\xi|-2\ell/\e)^{q}}  \int_{\Omega_z\setminus [S]_z}   \frac{|\,|u_{\xi}(t+\e|\xi|; z)-u_{\xi}(t;z))| + 2\ell\|\nabla u\|_{\infty}|^p}{\e^{p}}\rho_{\xi}(t;z)^2 \d t\nonumber\  \\
&\leq \frac{\zeta'(0)|\xi|^{p-1} \e^{-1}}{(|\xi|-2\ell/\e)^{q}} \int_{\R} \int_{t}^{t+\e|\xi|} \left[ |u_{\xi}'(s;z)| + \frac{\ell}{\e|\xi|} \|\nabla u\|_{\infty}\right]^p \d s \rho_{\xi}(t;z)^2 \d t\nonumber\\
&= \frac{\zeta'(0)|\xi|^{p-1} \e^{-1}}{(|\xi|-2\ell/\e)^{q}} \int_{\R} \left[ |u_{\xi}'(s;z)|  + \frac{\ell}{\e|\xi|} \|\nabla u\|_{\infty}\right]^p  \int_{s-\e|\xi|}^{s}  \rho_{\xi}(t;z)^2 \d t \d s\label{eqn: en passant} 
\end{align}
where we used the identity
	\[
	\ca_{\R}(t) \ca_{[t,t+\e|\xi|]}(s)=\ca_{\R}(s)\ca_{[s-\e|\xi|,s]}(t).
	\]
Notice now that, for $\L^1-$a.e. $s\in \R$, we have
	\begin{align*}
		\lim_{\e\rightarrow 0} (\e|\xi|)^{-1}\left[ |u_{\xi}'(s;z)| + \frac{\ell}{\e|\xi|} \|\nabla u\|_{\infty}\right]^p & \int_{s-\e|\xi|}^{s}  \rho_{\xi}(t;z)^2 \d t=|u_{\xi}'(s;z)|^p\rho_{\xi}(s;z)^2.
	\end{align*}
In particular by exploiting the dominated convergence Theorem (the sequence is dominated by twice its limit for example, which, for $\L^d$-a.e. $\xi \R^d$ is summable for $\H^{d-1}-$a.e. $z\in \xi^{\perp}$ due to \eqref{eqn: cond1}) we obtain for $\L^d$-a.e. fixed $\xi$ and $\H^{d-1}$-a.e. fixed $z\in \xi^{\perp}$
	\begin{align}
	\lim_{\e\rightarrow 0^+} \frac{\zeta'(0)}{(|\xi|-2\ell/\e)^{q}}  \int_{[\Omega]_z\setminus [S]_z}   &\frac{|\,|u_{\xi}(t+\e|\xi|; z)-u_{\xi}(t;z))| + 2\ell\|\nabla u\|_{\infty}|^p}{\e^{p}}\rho_{\xi}(t;z)^2 \d t\nonumber\\	
	& \ \ \ \leq\zeta'(0)|\xi|^{p-q}\int_{\R} |u_{\xi}'(s;z)|^p\rho_{\xi}(s;z)^2 \d s\nonumber\\
	& \ \ \ =\zeta'(0)|\xi|^{p-q}\int_{[\Omega]_z}  |u_{\xi}'(s;z)|^p\rho_{\xi}(s;z)^2 \d s \label{eqn: limsup end part 1}
	\end{align}
since $\rho$ is defined to be zero outside $\Omega$. Let $\omega$ be the modulus of continuity of $\rho^2$ on $\overline{\Omega}$. That is, for $r>0$
\[ \omega(r) = \sup \{ |\rho^2(x) - \rho^2(y)| \;: \; x,y \in \overline{\Omega}, \;\; |x-y|<r \}. \]
From the definition of $[S]_z$ in \eqref{eq:Sz}  follows that
\begin{align} \label{eqn: limsup end part 2}
\begin{split}
\frac{\Theta}{\e}\int_{[S]_z} \rho_{\xi}(t;z)^2 \d t & \leq \int_{S_{u_{\xi}(\cdot;z)}}  \int_{(y - \e |\xi|,y] \cap [\Omega]_z} \rho_{\xi}(t;z)^2 dt \, \d \H^0(y) \\
& \leq \Theta |\xi| \int_{S_{u_{\xi}(\cdot;z)} \cap [\Omega]_z}  \rho_{\xi}(y;z)^2 + \omega(\e |\xi|) \,  \d \H^0(y) 
\end{split}
\end{align}

By collecting together \eqref{eqn: limsup end part 1} and \eqref{eqn: limsup end part 2} we get, for $\L^d$-a.e. $\xi\in \R^d$ and for $\H^{d-1}$-a.e $z\in \xi^{\perp}$, that
	\begin{align*}
	\limsup_{\e \rightarrow 0} F_{\e}(u_{\xi}(\cdot;z);\xi,[\Omega]_z)
	 \leq \zeta'(0)|\xi|^{p-q}\int_{[\Omega]_z} |u_{\xi}'(s;z)|^p&\rho_{\xi}(s;z)^2 \d s  \\
	 + \Theta|\xi| \int_{S_{u_{\xi}(\cdot;z)}\cap [\Omega]_z}  & \rho_{\xi}(t;z)^2   \d \H^0(t).
	\end{align*}
In particular we can apply the reverse Fatou's Lemma (again for every $\xi$, $F_{\e}$ is dominated by twice its limit which is summable in $z$ due to $MS_{\eta,\zeta}(u)<\infty$) and conclude
	\begin{align*}
	\limsup_{\e\rightarrow 0}  \int_{\xi^{\perp}}  F_{\e}(u_{\xi}(\cdot;z);\xi,[\Omega]_z) \d \H^{d-1}(z) 
	\leq & \zeta'(0)|\xi|^{p-q}\int_{\xi^{\perp}} \d \H^{d-1}(z)\int_{[\Omega]_z}  |u_{\xi}'(s;z)|^p \rho_{\xi}(s;z)^2 \d s  \\
	 &+ \Theta|\xi| \int_{\xi^{\perp}} \d \H^{d-1}(z) \int_{S_{u_{\xi}(\cdot;z)}\cap [\Omega]_z}   \rho_{\xi}(t;z)^2   \d \H^0(t).
	\end{align*}
By collecting \eqref{eqn: part 0}, \eqref{eqn: part 1}, \eqref{eqn: rbad points decay limsup}, \eqref{eqn: from transport to u} the definition of $F_{\e}$ and Lemma \ref{lem: key size lemma} we get, for $\L^d$-a.e. $\xi\in \R^d$, that
\begin{align*}
	\limsup_{n\rightarrow \infty} \mathcal{F}_{\e,n}(u_n;\xi)\leq & \zeta'(0)|\xi|^{p-q}\int_{\xi^{\perp}} \d \H^{d-1}(z)\int_{[\Omega]_z}  |u_{\xi}'(s;z)|^p \rho_{\xi}(s;z)^2 \d s  \\
	& + \Theta|\xi| \int_{\xi^{\perp}} \d \H^{d-1}(z) \int_{S_{u_{\xi}(\cdot;z)}\cap [\Omega]_z}   \rho_{\xi}(t;z)^2   \d \H^0(t).
	\end{align*}
A further application of the reverse Fatou's Lemma on \eqref{eqn:limsup all togheter},   combined with the fact that $\overline{\eta}(t)=\eta(\max\{t-2\ell/\e,0\})\rightarrow \eta(t)$ in $L^1$ as $\e\rightarrow 0$, leads to
	\begin{align*}
	\limsup_{n \rightarrow \infty} \mathcal{GAMS}_{\e_n,n}(u_n) \leq & \zeta'(0)\int_{\R^d} \eta(|\xi|)|\xi|^{p-q} \d \xi \int_{\xi^{\perp}} \d z\int_{[\Omega]_z} |u_{\xi}'(s;z)|^p  \rho_{\xi}(s;z)^2 \d s  \\
	&+ \Theta\int_{\R^d} \eta(|\xi|) |\xi| \d \xi \int_{\xi^{\perp}} \d z \int_{S_{u_{\xi}(\cdot;z)}\cap [\Omega]_z}   \rho_{\xi}(t;z)^2   \d \H^0(t)
\end{align*}	
which, thanks to \eqref{eqn: MS sliced} achieves the proof of Step one.\\
\text{}\\
\textbf{Step two:} \textit{$\limsup$ bound on $\mathcal{GMS}$}. Consider $\eta_M:=\ca_{[0,M)}(t)\eta(t)$ and notice that, by exploiting the notation of the proof of Lemma \ref{lem I can kill the remainder}, we have
	\begin{align*}
	\overline{\mathcal{GMS}}_{\e_n,n}(u_n)&=\mathcal{GAMS}_{\e_n,n}(u_n)+\mathcal{R}_n(u_n)\\
	&=\mathcal{GAMS}_{\e_n,n}(u_n)+\mathcal{R}_n(u_n;\eta_M)+\mathcal{R}_n(u_n;\eta-\eta_M)
\end{align*} 
where, with $\mathcal{R}_n(u_n;\eta_M),\mathcal{R}_n(u_n;\eta-\eta_M)$ we mean the energy $\mathcal{R}_n(u_n)$ with $\eta^M$, $\eta-\eta^M$ in place of $\eta$.  Since $\mathcal{R}_n(u_n;\eta_M)=	\overline{\mathcal{GMS}}_{\e_n,n}(u_n;\eta_M)- 	\mathcal{GAMS}_{\e_n,n}(u_n;\eta_M) $, by virtue of Lemma \ref{lem I can kill the remainder} we have
	\[
	\lim_{n\rightarrow +\infty} \mathcal{R}_n(u_n;\eta_M)=0.
	\]
From the other side, since $\rho$ is bounded from above and below we have that
	\[
	|\mathcal{R}_n(u_n;\eta-\eta_M)|\leq C \mathcal{GAMS}_{\e_n,n}(u_n;\eta-\eta_M)
	\]
for a universal constant $C$. Thanks to the step one and to Proposition \ref{prop: liminf inequality} we thus have
	\begin{align*}
	\lim_{n\rightarrow +\infty} & |\mathcal{R}_n(u_n;\eta-\eta_M)| \leq C\left[ \lim_{n\rightarrow +\infty}\mathcal{GAMS}_{\e_n,n}(u_n)-\lim_{n\rightarrow +\infty}\mathcal{GAMS}_{\e_n,n}(u_n;\eta_M) \right]\\
 & = C(MS_{\eta,\zeta}(u) - MS_{\eta_M,\zeta}(u)).
	\end{align*}
Since $MS_{\eta,\zeta}(u)<+\infty$ by taking the limit as $M\rightarrow +\infty$ and by exploiting the continuity of the constants in $MS$ we get
\begin{align*}
	\lim_{M\rightarrow +\infty} \lim_{n\rightarrow +\infty} & |\mathcal{R}_n(u_n;\eta-\eta_M)| = 0,
	\end{align*}
yielding 
	\[
	\lim_{n\rightarrow + \infty} \mathcal{R}_n(u_n)=0.
	\]
In particular, by invoking \eqref{eqn:key equality}, we reach
	\[
	\limsup_{n\rightarrow +\infty} \mathcal{GMS}_{\e_n,n}(u_n)\leq \limsup_{n\rightarrow +\infty} \overline{\mathcal{GMS}}_{\e_n,n}(u_n)=\limsup_{n\rightarrow +\infty} \mathcal{GAMS}_{\e_n,n}(u_n)\leq MS_{\eta,\zeta}(u).
	\]
\end{proof}

\begin{proof}[Proof of Proposition \ref{prop: limsup}]
Assume now  that $u\in SBV(\Omega)$. Let $u_j$ be the sequence given by Theorem \ref{theorem CT}. Then $u_j\rightarrow u$ in $L^1$ which means that $d_{TL^1}((\mu,u_j), (\mu,u))\rightarrow 0$. Set  (up to a subsequence) 
	\[
	L:=\limsup_{j\rightarrow \infty} MS_{\eta,\zeta}(u_j)
	\]
Notice that thanks to assertions (i) and (ii) of Theorem \ref{theorem CT}
  we have that $L\leq MS_{\eta,\zeta}(u)$.  For all $k\in \N$, consider $j_k$ such that
	\begin{align*}
	d_{TL^1}((\mu,u_{j_k}), (\mu,u))&\leq 1/(2k)\\
	MS_{\eta,\zeta}(u_{j_k}) &\leq L+ 1/(2k).
	\end{align*}
In particular it also holds that
	\[
MS_{\eta,\zeta}(u_{j_k})\leq L+1/(2k)\leq MS_{\eta,\zeta}(u)+1/(2k).
	\]
For every $j_k$ chosen as above let $\{u_{j_k}^n\}_{n\in \N}$ be the sequence given by Proposition \ref{prop: limsup reg} relative to $u_{j_k}$. By exploiting Proposition \ref{prop: liminf inequality} we can infer
	\[
	MS_{\eta,\zeta}(u_{j_k})=\lim_{n\rightarrow \infty}  \mathcal{GMS}_{\e_n,n}(u_{j_k}^n).
	\] 
Choose now $n_k$ such that
	\begin{align*}
	d_{TL^1}((\mu_n,u_{j_k}^{n}),(u_{j_k},\mu))<&1/(2k)\\
	\mathcal{GMS}_{\e_{n},n}(u_{j_k}^{n})<&MS_{\eta,\zeta}(u_{j_k})+1/(2k)  \ \  \ \ \ \ \ \text{for all $n\geq n_k$}\\
	&\left(<MS_{\eta,\zeta}(u) +\frac{1}{k} \ \  \ \ \ \ \ \ \text{for all $n\geq n_k$}\right).
	\end{align*}
Define now the following recovery sequence	
	\begin{equation}\label{eqn: sqc recovery}
	w_n:=	u_{j_k}^{n}, \ \ \text{if $n\in [n_{k},n_{k+1})$}, \ k \in \N 
	\end{equation}	
This means that, for any $n\in [n_k,n_{k+1})$ we have
	\begin{align*}
	d_{TL^1}((\mu_n, u_{j_k}^{n}),(\mu,u))&<1/k,\\
	\mathcal{GMS}_{\e_{n},n}(u_{j_k}^{n})&\leq  MS_{\eta,\zeta}(u) +\frac{1}{k}.
	\end{align*}
Implying
\begin{align*}
	d_{TL^1}((\mu_n, w_n),(\mu,u))&<1/k, \ \ \ \text{for all $n\in [n_k, n_{k+1})$}\\
	\mathcal{GMS}_{\e_{n},n}(w_n)&\leq  MS_{\eta,\zeta}(u) +\frac{1}{k} \ \ \ \text{for all $n\in [n_k, n_{k+1})$}.
	\end{align*}
In particular, $(\mu_n,w_n) \rightarrow (\mu,u)$ in $TL^1$ and
	\[
	\limsup_{k\rightarrow \infty} \mathcal{GMS}_{\e_{n},n}(w_n)\leq MS_{\eta,\zeta}(u).
	\]
\end{proof}
\section{Proof of the compactness result (Theorem \ref{thm : compact})}\label{sec: compactness}
This section is devoted to the proof of Theorem \ref{thm : compact} that establishes a compactness result for sequences of functions with uniformly bounded $\mathcal{GMS}_{\e_n,n}$ where $\e_n$ is  any sequence satisfying \eqref{eqn: condition on eps}. 
	\begin{remark}\label{rmk: counterexample}
%Marco wrote:  We discussed about enlight this remark, nonetheless I think it might be useful to see in practical what the graphs and the functions and so on, with some explicit computation. \nc 
Let us point out that, in contrast to \cite{GTS16} where an $L^1$ bound is assumed, our compactness Theorem \ref{thm : compact} requires an $L^{\infty}$ bound on the sequence $u_n$. Namely due to the fact that in $\mathcal{GMS}_{\e_n,n}(u_n)$ the differences in $u_n$ are inside a bounded concave function $\zeta$  a uniform bound on $\mathcal{GMS}_{\e_n,n}(u_n)$ is, in general, not translatable into a uniform bound on $GTV_{\e,n}$. This is not just a technical issue and in fact an $L^1$-type  bound is not sufficient for compactness. Here we provide a counterexample to compactness if one only assumes an $L^1$ bound on $u_n$. Choose $\rho=1$ and $\Omega=Q$ the unit cube centered at $0$. Let $\{Q_i^k\}_{i=1}^{2^{kd}}$ be a di-adic division  of $Q$ in cubes of edges-size $2^{-k}$ and let $\{x_i^k\}_{i=1}^{2^{kd}}$ be the uniform grid given by the baricenter of each cube $Q_i^k$. Consider the sequence of functions $u_k : \{x_i^k\}_{i=1}^{2^{kd}} \rightarrow \R$ defined as
		\[
		u_k(x_j^k):=\frac{\ca_{B_{r_k}(0)}(x_j^k)}{\omega_d r_k^d}
		\] 
	with $r_k:= 2^{-k/2}$. Notice that $x_i^k\in B_{r_k}(0)$ implies $Q_i^k \subset B_{2r_k}(0)$ and thus
		\begin{align*}
	\#(\{i: x_i^k \in B_{r_k}(0)\})&\leq \#(\{i:  Q_i^k \subset B_{2r_k}(0)\})=2^d \omega_d 2^{dk/2}.
		\end{align*}
On the other hand
\begin{align*}
	\#(\{i : Q_i^k\subset B_{r_k}(0)\}) &=
	 2^{kd}\left| \bigcup_{\substack{ Q_i^k \subset B_{r_k}(0)}} Q_j^k \right|\geq  2^{kd}|B_{r_k/2}(0)|=2^{-d}\omega_d  2^{kd/2}.
	\end{align*}
Since $Q_i^k\subset B_{r_k}(0)$ implies $x_i^k\in B_{r_k}(0)$, we conclude
		\begin{align*}
		2^{-d} \omega_d 2^{dk/2} \leq \#(\{i : x_i^k \subset B_{r_k}(0)\})\leq 2^d \omega_d 2^{dk/2}.
		\end{align*}
In particular, setting $\nu_k:=2^{-kd} \sum_{i=1}^{2^{kd}} \delta_{x_i^k}$, we have
	\begin{align*}
	\int_{Q} u_k \d \nu_k =2^{-kd} \sum_{i=1}^{2^{kd}} u_k(x_i^k)&= 2^{-kd} \frac{\#(\{i: x_i^k \in B_{r_k}(0)\})}{\omega_d r_k^{d}}= \frac{\#(\{i: x_i^k \in B_{r_k}(0)\})}{\omega_d 2^{kd/2}}
	\end{align*}
and so
	\begin{equation}\label{eqn: L1 estimate}
	2^{-d}   \leq \int_{Q} u_k \d \nu_k \leq 2^d  \ \ \ \text{for all $k\in \N$}.
	\end{equation}
This means that $u_k\in L^1\left(Q; \nu_k\right)$ and that
		\begin{equation}\label{eqn: L1 upper bound}
		\sup_{k\in \N}\{\|u_k\|_{L^1}\}\leq 2^d 
		\end{equation}
Consider now $\e_k:=2^{-k \alpha}$ for some $1/2<\alpha<1$ and notice that it satisfies \eqref{eqn: condition on eps}, since
 	\[
 	\lim_{k\rightarrow \infty} \frac{\log(n_k)^{1/d}}{\e_k n_k^{1/d} }=  0
 	\] 
 (here $n_k=2^{kd}$ and we are also considering $d>2$). Now we choose $\zeta$ as
 	\begin{equation}
 	\zeta(x)=
 	\left\{
 	\begin{array}{ll}
 	x & \text{for $x\leq 1$},\\
 	1 & \text{for $x\geq 1$}.
 	\end{array}
 	\right.
 	\end{equation}
With all these choices in mind, for any kernel $\eta$ satisfying the assumptions (B1)-(B2), we can conclude
 		\begin{align*}
 		\mathcal{GMS}_{\e_k,n_k}(u_k)&=\frac{1}{\e_k n_k^2}\sum_{i,j=1}^{n_k} \zeta\left( \e_k^{1-p+q} \frac{|u_k(x_i^k)-u_k(x_j^k)|^p}{|x_i^k-x_j^k|^q}\right) \eta_{\e_k}(|x_i^k - x_j^k|) \\
 		&\leq \frac{2}{\e_k n_k^2}\sum_{x_i^k\in B_{r_k}(0)^c} \sum_{x_j^k\in B_{r_k}(0)}  \eta_{\e_k}(|x_i^k - x_j^k|)=2GTV_{\e_k,n_k}(\ca_{B_{r_k}}(0)).
 		\end{align*}
Notice that $\sup_{k} \{GTV_{\e_k,n_k}(\ca_{B_{r_k}}(0))\}<\infty$ and so
 		\begin{equation}\label{eqn: energy upper bound}
 		\sup_{k\in \N} \{\mathcal{GMS}_{\e_k,n_k}(u_k)\}<\infty.
 	 		\end{equation}
By collecting \eqref{eqn: L1 upper bound} and \eqref{eqn: energy upper bound} we are finally lead to
 	\[
 		\sup_{k\in \N} \{\|u_k\|_{L^1}+\mathcal{GMS}_{\e_k,n_k}(u_k)\}<\infty.
 	\]
Nonetheless we cannot expect any sort of $TL^1$ compactness for the sequence $(\nu_{k},u_k)$. Indeed, the only possible pointwise limit for $u_k\circ T_{n_k}:Q\rightarrow \{0,r_k^{-d}/\omega_d\}$ can be $u=0$ but 
 		\begin{align*}
 		\int_{Q} u_k(T_{n_k}(x)) \d x= \int_{Q} u_k \d \nu_k\geq 2^{-d}>0
 		\end{align*}
 	because of \eqref{eqn: L1 estimate}.
 \end{remark}
 
Our proof is based on the approach to the  compactness theorem for general non-local functional established in \cite[Theorem 5.1]{gobbino2001finite}. With a careful application of such a Theorem we can indeed obtain the following proposition.
 	\begin{proposition}\label{prop: cmp gobbino}
Let $p\geq 1$ and consider an open set $A\subset \R^d$ with Lipscthiz boundary. Let $u_{\e}\in L^1(\R^d)$ be a sequence of function such that
 		\[
 		\inf\left\{\int_{\R^d \times \R^d}  \frac{|u_{\e}(xi)-u(y)|}{|x-y|}J_{\e}(|x-y|) \d x \d y+\|u_{\e}\|_{\infty} \right\}<+\infty
 		\]
 	where $J$ is any kernel such that $\{\xi : J(\xi)>c\}$ has non-empty interior for some $c>0$. Then the sequence $u_{\e}$ is compact in $L^1(\R^d)$.
 	\end{proposition}
 	
Before proceeding to the proof of Theorem \ref{thm : compact} and in order to apply Proposition \ref{prop: cmp gobbino} (which holds for functions defined on the whole $\R^d$) we need the following extension Lemma in the same spirit of \cite[Lemma 4.4]{GTS16}.
 
\begin{lemma}\label{lem:extension}
Suppose that $\Omega$ is abounded open set with $C^2$ boundary. Let $\eta$ be a compactly supported, non-increasing  kernel which is not identically equal to zero. Let $\{u_{\e}\}_{\e>0}\subset L^1(\Omega)$ be a sequence. Then there exists a sequence of function $\{v_\e\}_{\e>0} \subset  L^1(\R^d)$ such that
	\begin{itemize}
	\item[(i)] $v_{\e}=u_{\e}$ $\L^d$-a.e. on $\Omega$;
	\item[(ii)]  There exists a kernel $J^{\eta}$ such that $\{\xi\ : \ J^{\eta}(|\xi|)>c\}$ has not empty interior for some $c>0$ and such that if
	\[
	\sup_{\e>0} \left\{\int_{\Omega\times \Omega}\frac{|u_{\e}(x)-u_{\e}(y)| }{|x-y|} \eta_{\e}(|x-y|)\left(\frac{|x-y|}{\e}\right)^p\d x \d y + \|u_{\e}\|_{\infty}\right\}<\infty,
	\]
then
	\[
	\sup_{\e>0} \left\{\int_{\R^d\times \R^d}\frac{|v_{\e}(x)-v_{\e}(y)| }{|x-y|} J_{\e}^{\eta}(|x-y|)\d x \d y + \|v_{\e}\|_{\infty}\right\}<\infty
	\]
	\end{itemize}
	\end{lemma}
\begin{proof}
Since $\Omega$ has $C^2$ boundary we can find $\delta>0$ for which the projection operator $x\rightarrow Px\in \Omega$ is well defined on $U:=\{x\in \R^d \ | \ d(x,\Omega)\leq \delta\}$ and satisfies
	\[
	|x-Px|=d(x,\Omega).
	\]
We moreover consider a smooth cut off function $\tau(s)\leq 1$, such that $\tau(s)=1$ for $s\leq \delta/8$ and $\tau(s)=0$ for $s\geq \delta/4$ and we consider the reflection $Rx:=2Px-x$. Set also
\begin{align*}
W&:=\{z\in \R^d\setminus \Omega \ | \ d(x,\Omega)<\delta/4 \}\\
V&:=\{z\in \R^d\setminus \Omega \ | \ d(x,\Omega)< \delta/8 \}.
\end{align*}
It has been shown in the proof of \cite[Lemma 4.4]{GTS16} that 
	\begin{align}
	\frac{1}{4}|x-y|\leq |Rx-Ry|\leq 4 |x-y| &\ \ \  \text{for all $x,y\in W$}\label{min 1};\\
	|Rx- y|\leq 2|x-y| & \ \ \ \text{for all $x\in W$, $y\in \Omega$}\label{min 2};
	\end{align} 
Since $\eta$ can be extended continuously at $0$ with $\eta(0)>0$, up to decreasing the value of $\delta$, we can also guarantee that
	\[
	J^{\eta}(t):=t^p\eta\left(4t\right)\left(1-\ca_{[0,\delta]}(t)\right)
	\] 
is such that $\{t \ | \ J^{\eta}(t)>c\}$ has not empty interior for some $c>0$.
Notice also that 
	\begin{equation}\label{min 3}
	|Rx- y|\geq \frac{3}{4}|x-y| \ \ \text{for $x,y\in W$, $|x-y|\geq \delta$}.
	\end{equation}
In the light of this fact, we introduce the functions $\tilde{v}_{\e}:=u_{\e}(Rx)$ and $v_{\e}(x):=\tau(|Px-x|)\tilde{v}_{\e}(x)$. Clearly $(i)$ is satisfied. Notice that for $y\in \Omega$ 
	\[
	\frac{|x-y|}{\e}\geq \frac{d(x,\Omega)}{\e}.
	\]
Thus, if $d(x,\Omega)\geq \delta/8$, for $\e$ small enough and thanks to the fact that $\eta$ is compactly supported, we conclude that
	\[
	\eta_{\e}(4|x-y|)=0.
	\]
Consequently
	\begin{align}
	& \int_{(\R^d\setminus \Omega)  \times  \Omega} \frac{|v_{\e}(x)-v_{\e}(y)| }{|x-y|} J_{\e}^{\eta}(|x-y|)\d x \d y\nonumber\\
	&\ =\int_{(\R^d\setminus \Omega)  \times  \Omega} \frac{|v_{\e}(x)-v_{\e}(y)| }{|x-y|} \eta_{\e}(4|x-y|)\left(\frac{|x-y|}{\e}\right)^{p}\left(1-\ca_{[0,\delta)}\left(\frac{|x-y|}{\e}\right)\right)\d x \d y\nonumber\\
	 &\ = \int_{\substack{V \times  \Omega \\ \{|x-y|\geq \delta\e\}}}  \frac{|\tilde{v}_{\e}(x)-\tilde{v}_{\e}(y)| }{|x-y|} \eta_{\e}(4|x-y|)\left(\frac{|x-y|}{\e}\right)^{p}\d x \d y\nonumber\\
	 &\ \leq C\int_{\substack{V \times  \Omega \\ \{|x-y|\geq \delta\e\}}} \frac{|u_{\e}(Rx)-u_{\e}(y)| }{|Rx-y|} \eta_{\e}(4|x-y|)\left(\frac{|Rx-y|}{\e}\right)^{p}\d x \d y\nonumber\\
	 &\ \leq C\int_{V \times  \Omega} \frac{|u_{\e}(Rx)-u_{\e}(y)| }{|Rx-y|} \eta_{\e}(2|Rx-y|)\left(\frac{|Rx-y|}{\e}\right)^{p}\d x \d y\nonumber\\
	 &\ \leq C\int_{\Omega \times  \Omega} \frac{|u_{\e}(z)-u_{\e}(y)| }{|z-y|} \eta_{\e}(2|z-y|)\left(\frac{|z-y|}{\e}\right)^{p}\d z \d y\nonumber\\
	 &\ \leq C\int_{\Omega \times  \Omega} \frac{|u_{\e}(z)-u_{\e}(y)| }{|z-y|} \eta_{\e}(|z-y|)\left(\frac{|z-y|}{\e}\right)^{p}\d z \d y\label{ext1}
\end{align}	 
Where we have exploited \eqref{min 1},  \eqref{min 2}, \eqref{min 3} the change of variable $Rx=z$ and the fact that $R$ is bi-Lipscthiz on $W$. From the other side, for $(x,y)\in W\times W$ we have
		\begin{align*}
	\frac{|v_{\e}(x)-v_{\e}(y)| }{|x-y|}&\leq \frac{\tau(|Px-x|)(|\tilde{v}_{\e}(x)-\tilde{v}_{\e}(y)|)+\tau(|Px-x|)-\tau(|Py-y|)|\tilde{v}_{\e}(y)| }{|x-y|}\\
		&\leq \frac{|u_{\e}(Rx)-u_{\e}(Ry)|}{|x-y|}+\|v_{\e}\|_{\infty}\text{Lip}(\tau)\leq \frac{1}{4}\frac{|u_{\e}(Rx)-u_{\e}(Ry)|}{|Rx-Ry|}+\|v_{\e}\|_{\infty}\text{Lip}(\tau).
\end{align*}	 
Moreover
\begin{align*}
\eta_{\e}(4|x-y|)&\left(\frac{|x-y|}{\e}\right)^{p}\left(1-\ca_{[0,\delta)}\left(\frac{|x-y|}{\e}\right)\right)\\
&\leq \eta_{\e}(4|x-y|)\left(\frac{|x-y|}{\e}\right)^{p}\\
&\leq C \eta_{\e}(|Rx-Ry|)\left(\frac{|Rx-Ry|}{\e}\right)^{p}
\end{align*}
In particular with the same change of variable as above we achieve
	\begin{align}
	 \int_{(\R^d\setminus \Omega)  \times  (\R^d\setminus \Omega) } & \frac{|v_{\e}(x)-v_{\e}(y)| }{|x-y|}J_{\e}^{\eta}(|x-y|)\d x \d y\nonumber\\
	&\leq C\left(1+ \int_{\Omega \times  \Omega} \frac{|u_{\e}(x)-u_{\e}(y)| }{|x-y|} \eta_{\e}(|x-y|)\left(\frac{|x-y|}{\e}\right)^{p}\d x \d y\right)\label{ext2}
	\end{align}
By collecting \eqref{ext1}, \eqref{ext2} and the definition of $v_{\e}$ we prove (ii).
\end{proof}
\begin{proof}[Proof of Theorem \ref{thm : compact}]
Since $\rho$ is always bounded from above and below, without loss of generality we can assume $\rho=1$. Moreover we can always assume that $\eta$ is compactly supported since, by replacing $\eta$ with $\eta\ca_{[0,M]}$ for suitable $M$, we are decreasing the energy. Without loss of generality we may also assume that $\eta$ is supported on $[0,1]$. Moreover, as usual, 
 we will omit the dependence on $n$ of the sequences $\e_n$ and $\ell_n=\|T_n-\text{Id}\|_{\infty}$.\\
\text{}\\
Due to the properties of $\zeta$ we can always find real constants $\zeta_2$ and $c_2>0$ such that
	\begin{equation}
	\begin{cases}
	\frac{\zeta'(0)}{2}  \leq \frac{\zeta(t)}{t} \; & \text{for $t\leq \zeta_2$} \medskip \\ 
	 	 c_2 \leq \zeta(t) & \text{for $t\geq \zeta_2$}.
	\end{cases}
	\end{equation}
Set $\tilde{u}_n(x):=u_n\circ T_n$, where $T_n: \Omega \rightarrow  \{x_1,\ldots,x_n\}$ is the map that transports $\tmu_n$ to $\mu_n$; the measures given by Lemma \ref{lem: improved convergence}. We define
	\begin{equation}
	A_{\e}:=\left\{(x,y)\in \Omega\times \Omega \ \Big{|}\ \e^{1-p+q} \frac{|\tilde{u}_n(x) - \tilde{u}_n(y)|^p}{|T_n(x)-T_n(y)|^q} \geq \zeta_2 \right\}.
	\end{equation}
We immediately see that
	\begin{align*}
	\mathcal{GMS}_{\e,n}(u_n)&\geq\frac{1}{\e} \int_{A_{\e} } \zeta\left(\e^{1+q-p} \frac{|\tilde{u}_n(x) - \tilde{u}_n(y) |^p}{|T_n(x)-T_n(y)|^q}\right)\eta_{\e}(|T_n(x)-T_n(y)|)\d x \d y\\
	&\geq \frac{c_2}{\e} \int_{A_{\e}} \eta_{\e}(|T_n(x)-T_n(y)|)\d x \d y.
	\end{align*}
Moreover since $\eta$ is non increasing and non identically $0$ we can find a positive $r>0$ such that $	\eta(t+r)\ca_{[0,1)}(t)$ is not identically $0$. Set $\bar{\eta}(t):=\eta(t+r)\ca_{[0,1]}(t)$ and notice that $\bar{\eta}$ is still a non increasing kernel, supported on $[0,1-r]$. Since for $\e$ small enough we can always guarantee that
	\[
	\frac{|T_n(x)-T_n(y)|}{\e}\leq \frac{|x-y|}{\e}+r \ \ \ \Rightarrow \ \ \ \eta\left(\frac{|T_n(x)-T_n(y)|}{\e}\right)\geq \bar{\eta}\left(\frac{|x-y|}{\e}\right).
	\] 
We can also infer, for $|x-y|/\e\leq (1-r)$
	\[
 \left(\frac{|x-y|}{\e}\right)^p  \bar{\eta}_{\e}(|x-y|)\leq (1-r)^{p-1} \frac{|x-y|}{\e}\eta_{\e}(|T_n(x)-T_n(y)|).
	\]
Thus we can conclude
	\begin{equation}\label{eqn: first part compactness}
\left(\int_{A_{\e}} \frac{|\tilde{u}_n(x)-\tilde{u}_n(y)|}{|x-y|} \left(\frac{|x-y|}{\e}\right)^p\bar{\eta}_{\e}(|x-y|)\d x \d y\right)^p \leq C\|u_n\|^p_{\infty} \mathcal{GMS}_{\e,n}(u_n)^p
	\end{equation}
for a universal constant $C>0$.
For the remaining part we notice the following thing. On $(\Omega\times \Omega)\setminus A_{\e}$ it holds
	\begin{align*}
	\zeta\left(\e^{1-p+q} \frac{|\tilde{u}_n(x)-\tilde{u}_n(y)|^p}{|T_n(x)-T_n(y)|^q}\right)\geq\e \frac{\zeta'(0)}{2}\frac{|\tilde{u}_n(x)-\tilde{u}_n(y)|^p}{|x-y|^p}\frac{\e^q}{|T_n(x)-T_n(y)|^q} \left(\frac{|x-y|}{\e}\right)^p
	\end{align*}
and for $\e$ small enough we have
	\[
	\frac{\e^q}{|T_n(x)-T_n(y)|^q}\geq \frac{1}{\left(\frac{|x-y|}{\e}+\frac{2\ell}{\e}\right)^q}\geq \frac{1}{\left(\frac{|x-y|}{\e}+1\right)^q}.
	\]
This yields, by recalling that $\bar{\eta}_{\e}(|x-y|)=0$ for $|x-y|\geq (1-r)\e$, that
\begin{align}
\mathcal{GMS}_{\e,n}&(u_n)\geq \frac{\zeta'(0)}{2}\int_{(\Omega\times \Omega)\setminus A_{\e}}  \frac{|\tilde{u}_n(x)-\tilde{u}_n(y)|^p}{|x-y|^p}\frac{\left(\frac{|x-y|}{\e}\right)^p}{\left(\frac{|x-y|}{\e}+1\right)^q}  \bar{\eta}_{\e}(|x-y|)\d x \d y\nonumber\\
&\geq  \Lambda^{1-p} C\left(\int_{(\Omega\times \Omega)\setminus A_{\e}}  \frac{|\tilde{u}_n(x)-\tilde{u}_n(y)|}{|x-y|} \left(\frac{|x-y|}{\e}\right)^p  \bar{\eta}_{\e}(|x-y|)\d x \d y\right)^p
\label{eqn: second part compactness}
\end{align}
where
	\begin{align*}
	\Lambda:=&\int_{(\Omega\times \Omega)\setminus A_{\e}} \left(\frac{|x-y|}{\e}\right)^p  \bar{\eta}_{\e}(|x-y|)\d x \d y\leq \int_{\Omega\times \Omega} \left(\frac{|x-y|}{\e}\right)^p  \bar{\eta}_{\e}(|x-y|)\d x \d y\\
	&\leq 2|\Omega| \int_{\R^d} |\xi|^p \bar{\eta}(|\xi|)\d \xi\leq 2|\Omega| \int_{B_1(0)} |\xi|^p \bar{\eta}(|\xi|)\d \xi<+\infty.
	\end{align*}
By  collecting \eqref{eqn: first part compactness} and \eqref{eqn: second part compactness} we conclude     that
	\begin{equation}\label{eqn:temporary equation}
	\sup_{n\in \N}\left\{\int_{\Omega\times \Omega} \frac{|\tilde{u}_n(x)-\tilde{u}_n(y)|}{|x-y|} \left(\frac{|x-y|}{\e_n}\right)^p  \bar{\eta}_{\e_n}(|x-y|)\d x \d y+\|\tilde{u}_n\|_{\infty}\right\}<+\infty.
	\end{equation}
We now divide the proof in three steps.
\medskip

\noindent
\textbf{Step one:} \textit{$\Omega$ has $C^2$ boundary}. In this case, by applying Lemma \ref{lem:extension} we can find a sequence of $\{v_n\}_{n\in \N}\subset  L^1(\R^d)$ such that $v_n=\tilde{u}_n$ $\L^d$-a.e. on $\Omega$. Moreover, due to \eqref{eqn:temporary equation} there exists a kernel $J^{\bar{\eta}}$ such that $\{\xi \ : \ J^{\bar{\eta}}(|\xi|)\geq c\}$ has not empty interior for some $c>0$ and for which
	\[
	\sup_{n\in \N}\left\{\int_{\R^d\times \R^d} \frac{|v_n(x)-v_n(y)|}{|x-y|} J_{\e_n}^{\bar{\eta}}(|x-y|)\d x \d y+\|v_n\|_{\infty}\right\}<+\infty.
\]	
Then, by applying Proposition \ref{prop: cmp gobbino} we deduce that $\{v_{n}\}_{n\in \N}$ is compact in $L^1(\R^d)$ and thus $\{\tilde{u}_n\}_{n\in \N}$ is compact in $L^1(\Omega)$.
\medskip

\noindent
\textbf{Step two:}  \textit{$\Omega$ has Lipschitz boundary}. Thanks to \cite[Remark 5.3]{ball2017partial} there exists a bi-Lipscthiz map $\Psi:\hat{\Omega}\rightarrow\Omega$ where $\hat{\Omega}$ is a domain with smooth boundary. Consider $\hat{u}_n:=\tilde{u}_n\circ \Psi:\hat{\Omega} \rightarrow \R$. Clearly $\|\hat{u}_n\|_{\infty}\leq \|\tilde{u}_n\|_{\infty}$. Moreover
	\begin{align*}
	\int_{\hat{\Omega}\times \hat{\Omega}} &\frac{|\hat{u}_n(x)-\hat{u}_n(y)|}{|x-y|}  \left(\frac{|x-y|}{\e_n}\right)^p  \bar{\eta}_{\e}(\text{Lip}(\Psi)^{-1}|x-y|)\d x \d y\\
	&\leq C\int_{\Omega\times \Omega} \frac{|\tilde{u}_n(x)-\tilde{u}_n(y)|}{|\Psi(x)-\Psi(y)|} \left(\frac{|\Psi(x)-\Psi(y)|}{\e_n}\right)^p  \bar{\eta}_{\e_n}(\text{Lip}(\Psi)^{-1}|\Psi(x)-\Psi(y)|)\d x \d y\\
		&\leq C\int_{\Omega\times \Omega} \frac{|\tilde{u}_n(x)-\tilde{u}_n(y)|}{|x-y|} \left(\frac{|x-y|}{\e_n}\right)^p  \bar{\eta}_{\e_n}(|x-y|)\d x \d y.
	\end{align*}
By exploiting \eqref{eqn:temporary equation}, Lemma \eqref{lem:extension} and by arguing as in Step one we conclude that $\{\hat{u}_n\}_{n\in \N}$ is compact in $L^1(\hat{\Omega})$. Since $\Psi$ is bi-Lipschitz, a simple change of variable shows that $\{\tilde{u}_n\}_{n\in \N}$ is compact in $L^1(\Omega)$.
\medskip

\noindent
\textbf{Step three:}  \textit{Compactness of $(\mu_n,u_n)$ in $TL^1$}. Thanks to Steps one and two we obtained that $\{\tilde{u}_n:=u_n\circ T_n\}_{n\in \N}$ is compact in $L^1(\Omega)$ and converges to some $u$ up to a subsequence. Thanks to Lemma \ref{lem:improved convergence TLp} we deduce that $(\tilde{\mu}_n, u_n\circ T_n)\rightarrow (\mu,u)$ in $TL^1$ as well. In particular, since $d_{\infty}(\mu_n,\tilde{\mu}_n)\rightarrow 0$, we have
	\begin{align*}
	d_{TL^1}((\mu_n,u_n),(\mu,u))&\leq d_{TL^1}((\mu_n,u_n),(\tilde{\mu}_n,u_n\circ T_n)+d_{TL^1}( (\tilde{\mu}_n,u_n\circ T_n),(\mu,u))\\
	&=d_{\infty}(\mu_n,\tilde{\mu}_n)+d_{TL^1}( (\tilde{\mu}_n, u_n\circ T_n),(\mu,u))\rightarrow 0.
	\end{align*}
\end{proof}
%%%%%%%%%%%%%%%%%%%%%%%%%%%%%%%%%%%%%%%%%%%%%
\section{Proofs of corollaries}
We now prove the Corollary \ref{cor:nonoise}.
\begin{proof}
Let $u_n : V_n \to \R$ be a sequence of functions such that $(\mu_n, u_n)$ converges in $TL^2$ towards 
$(\mu,u)$. Let $\overline T_n$ be the $\infty$-optimal transport map between $\mu$ and $\mu_n$. 
To show the liminf inequality needed for $\Gamma$-convergence,  it suffices to establish the convergence of the fidelity term, as the other terms are same as in Theorem \ref{thm main: G conv}.
Note that
\[ \frac{1}{n} \sum_{i=1}^n |u_n(x_i) - f(x_i)|^2 = \int_\Omega |u_n \circ \overline T_n - f \circ \overline T_n|^2 \rho(x) dx. \]
We claim that since $f$ is piecewise continuous $f \circ \overline T_n$ converges to $f$ in $L^2(\mu)$. Namely let $J_f$ be the set of discontinuities of $f$ and let $\tilde J_f = J_f \cup \partial \Omega$. Let
$\tilde J_{f, \delta} = \{ x \in \Omega \::\: d(x, \tilde J_f)< \delta\}$. Since $\H^{d-1}(\tilde J_f \cup \partial \Omega) < \infty$, $\mu(\tilde J_{f, \delta}) \to 0$ as $\delta \to 0$. To establish the convergence let $\e>0$. Let $\delta$ be such that $4 \|f \|_{L^\infty} \mu(\tilde J_{f, \delta}) < \e$. Let $n$ be so large that
$\| \overline T_n - I \|_{L^\infty}  < \frac12 \delta$. Since $f $ is uniformly continuous on $\Omega \setminus \tilde J_{f, \delta/2}$  and $\| \overline T_n - I \|_{L^\infty} \to 0$ as $n \to \infty$, $\, f \circ \overline T_n$ converges uniformly to $f$ on $\Omega \setminus \tilde J_{f, \delta}$.  Therefore for all $n$ large enough
\[ \int_\Omega |f - f \circ \overline T_n|^2 \rho(x) dx \leq 2 \mu(\tilde J_{f, \delta}|)  \| f\|_{L^\infty} + 
\int_{\Omega \setminus \tilde J_{f, \delta}} |f - f \circ \overline T_n|^2 \rho(x) dx  < \frac{\e}{2} + \frac{\e}{2}.\]
Since $(\mu_n,u_n) \to (\mu,u)$ in $TL^2$, and $\overline T_n$ is a stagnating sequence of transport maps we conclude that  we conclude that $u_n \circ \overline T_n \to u$ in $L^2(\mu)$. Combining with the convergence for $f \circ \overline T_n$ obtained above we conclude that
 $ \int_\Omega |u_n \circ \overline T_n - f \circ \overline T_n|^2 \rho(x) dx \to  \int_\Omega |u - f|^2 \rho(x) dx$ as $n \to \infty$. 

Establishing the limsup inequality is straightforward by using the same approximation argument and recovery sequence as in the proof of  Theorem \ref{thm main: G conv}. 

To establish the compactness of the sequence of minimizers, let $u_n$ be a minimizer of $\mathcal{GMS}_{f,\e,n}$. By truncation it is immediate that $\|u_n \|_{L^\infty(\mu_n)} \leq \|f\|_{L^\infty(\mu)}$. Therefore  the compactness claim of the  Theorem \ref{thm : compact} implies that $(\mu_n,u_n)$ converges along a subsequence in $TL^1$ to $(\mu,u)$ for some $u \in L^\infty(\mu)$. The boundedness in $L^\infty$ implies, via interpolation, that the convergence is in $TL^2$. 
The fact that $u$ is a minimizer follows from  $\Gamma$-convergence. 
\end{proof}

We now prove the Corollary \ref{cor:noise}.
\begin{proof}
Let us first establish the liminf inequality. 
Let $u_n : V_n \to \R$ be a sequence of functions such that $(\mu_n, u_n)$ converges in $TL^2$ towards 
$(\mu,u)$. 
Given the results of  Theorem \ref{thm main: G conv} and assumptions on $\beta$,
it suffices to show that the fidelity term converges, that is that
\begin{equation} \label{temp12}
 \frac{1}{n} \sum_{i=1}^n |u_n(x_i) - f(x_i) - y_i|^2 \to \iint_{\Omega \times \R} | u(x) - f(x) - y|^2 \rho(x) dx d \beta(y)
\end{equation}
as $n \to \infty$. 
Note that $ \frac{1}{n} \sum_{i=1}^n |u_n(x_i) - f(x_i) |^2 \to \int_\Omega | u(x) - f(x)|^2 \rho(x) dx$ follows from the proof of Corollary \ref{cor:nonoise}. The fact that $\frac{1}{n} \sum_{i=1}^n \delta_{y_i}$ weakly converges to $\beta$ follows from Glivenko--Cantelli lemma. Due to boundedness of moments we conclude that $\frac{1}{n} \sum_{i=1}^n \delta_{y_i}$ converges to $\beta$ is $q-$Wasserstein distance for all $q\geq 1$. Using the boundedness of $\beta$ we conclude that $\frac{1}{n} \sum_{i=1}^n |y_i|^2$ converges to  $\int_\R y^2 d\beta(y) =  \iint_{\Omega \times \R} y^2 \rho(x) dx d\beta(y)$ since $\int_\Omega \rho(x) dx =1$.  Note that $\gamma_n := \frac{1}{n} \sum_{i=1}^n \delta_{(x_i, y_i)}$ converges to $\gamma:= \mu \times \beta$ in Wasserstein distance, again due to Glivenko-Cantelli lemma. Let $\kappa_n$ be the optimal transport plan for the Wasserstain distance between $\gamma$ and $\gamma_n$. Let $\pi_n:= {\Pi_{1,3}}_\sharp \kappa_n$ where $\Pi_{1,3}$ is the projection to the first and the third variable. By definition $\pi_n$ is a stagnating sequence of transport plans. Therefore, by Proposition 3.12 of \cite{GTS16}, since $(\mu_n, u_n - f)\overset{{TL^2}}{\longrightarrow} 
(\mu, u-f)$ by the proof of Corollary \ref{cor:nonoise},
$ \iiiint |x- \tilde x|^2 + |u_n(\tilde x) - f(\tilde x) - u(x) + f(x)|^2 d \kappa_n(x,y, \tilde x, \tilde y) \to 0$ as $n \to \infty$. Thus $(\gamma_n, u_n - f)\overset{{TL^2}}{\longrightarrow} 
(\gamma, u-f)$ as $n \to \infty$. Similarly $(\gamma_n, y)\overset{{TL^2}}{\longrightarrow} 
(\gamma, y)$ as $n \to \infty$. Consequently $(\gamma_n, (u_n - f) y) \overset{{TL^1}}{\longrightarrow} 
(\gamma, (u-f)y)$ as $n \to \infty$. Thus  
\begin{align*}
 \frac{1}{n} \sum_{i=1}^n & (u_n(x_i) - f(x_i))y_i -  \iint_{\Omega \times \R}  (u(x) - f(x)) y\,  \rho(x) dx d\beta(y) \\
 & =  \iiiint (u_n(\tilde x) - f(\tilde x)) \tilde y - (u(x) - f(x)) y \,  d \kappa_n(x,y, \tilde x, \tilde y) \to 0 \quad \te{ as } n \to \infty.
\end{align*}
 Combining with the limits above establishes \eqref{temp12}. 
  
The proofs of limsup inequality, compactness and the converge of minimizers are as before.  
\end{proof}

\section{Numerical algorithm and experiments} \label{sec:numerics}

%
%{\color{red}From the paper:}
%\begin{equation}\label{eq:GMSf}
%        \mathcal{GMS}_f(u):=  \frac{\lambda}{n} \sum_{i=1}^n |u(x_i) - f_i|^2 + \frac{1}{\e n^2} \sum_{i,j=1}^n \zeta\left( \frac{1}{\e} |u(x_i) - u(x_j)|^2 \right) w_{ij}
%\end{equation}
%\dgreen
%\emph{Remove in the preprint: } In the code \texttt{pricesrn.m} the following functional is minimized
%\begin{equation}\label{eq:price}
%        \mathcal{NGMS}_f(u):=  \sum_{i=1}^n |u(x_i) - f_i|^2 + \frac{1}{\lambda \e n } \sum_{i,j=1}^n \zeta\left(\frac{ |u(x_i) - u(x_j)|^2}{\e}  \right)   \frac{1}{\e^d} \exp \left(-\frac{|x_i-x_j|^2}{2 \sigma^2 \e^2 } \right)
%% sigma above is sigmanew in the code
%\end{equation}
%where $\lambda \sim 8$, $\e \sim 0.05$, $\sigma \sim 1$. 
%\nc
%\bigskip

Here we desribe an efficient numerical algorithm for computing the, approximate, minimizers of the graph Mumford-Shah functional and illustrate its behavior on a real world set of data. We note that similar functionals were minimized using the ADMM algorithm by Hallac,  Leskovec, and Boyd by \cite{hallac2015}. Here we 
minimize \eqref{eq:GMSf}, where $\zeta$ is non convex. We use a standard
``Iterated Reweighted Least Square'' approach which in this context dates back at least to~\cite{GemanReynolds1992}
(\textit{cf}~also the implementation in~\cite{chambolle1999finite}). In our case, the idea is to perform several iteration,  linearizing each time the problem with respect to $\zeta$ around
the previous value.

This can be presented as follows: we assume that $\zeta$ is  concave, with $\zeta'(0)=1$ and $\zeta(+\infty)=1$, for
instance $\zeta(t) = t/(1+t)$ for $t>0$. Then, using the Legendre transform, one can write for $t>0$
$\zeta(t) = \min_{\zz\in [0,1]} \zz t + \Psi(\zz)$ for some convex function $\Psi$. (One has $\zeta(t) = -\Psi^*(-t)$ and
$\Psi(\zz) = \max_t \zeta(t)-t\zz$, where $\Psi^*$ denotes the classical convex conjugate of $\Psi$.)
The minimum (if unconstrained) is reached at $\zz$ which solves $t+\Psi'(\zz)=0$, hence $\zz = (\Psi^*)'(-t) = \zeta'(t)$.

We consider the edge weights given by kernel $\eta(s) = e^{- \frac{s^2}{2 \sigma^2}}$ where $\sigma$ is a parameter that can be tuned. 
Minimizing~\eqref{eq:GMSf} is equivalent to solving:
\begin{equation} \label{eq:price_num}
  \min_{u,\zz} \sum_{i=1}^n |u(x_i) - f_i|^2 + \frac{1}{\lambda \e n}\
  \sum_{i,j=1}^n \left( \zz_{i,j} \frac{1}{\e}  |u(x_i) - u(x_j)|^2 + \Psi(\zz_{i,j}) \right)   \frac{1}{\e^d} \exp \left(-\frac{|x_i-x_j|^2}{2 \sigma^2 \e^2} \right)
\end{equation}
where the new variable $(\zz_{i,j})$ is defined on the active edges.
This is computed by alternatively minimizing the problem with respect to $u$ and $\zz$: in $u$, the problem is quadratic
and can be minimized efficiently, depending on the graph, by inverting the graph Laplacian (plus identity) or a
conjugate gradient method. In $\zz$, the solution is explicitly given by
\[
  \zz_{i,j} =   \zeta' \left(\frac{ |u(x_i) - u(x_j)|^2}{\e}  \right).
\]
In practice, we have implemented the following cases:
\begin{align}
\zeta(t) & = \frac{2}{\pi}\arctan\left(\frac{\pi t}{2}\right),\hspace*{-40pt}  & \hspace*{-40pt}  \zeta'(t) & = \frac{1}{1+\frac{\pi^2t^2}{4}} \label{eq:zetaMS} \\
  \zeta(t) & = \sqrt{\delta^2+t} , \hspace*{-40pt}  & \hspace*{-40pt}  \zeta'(t) & = \frac{1}{2\sqrt{\delta^2+t}},\label{eq:zetaTV}  \\
    \zeta(t) & = t , \hspace*{-40pt}  & \hspace*{-40pt}  \zeta'(t) & = 1,\label{eq:zetaLap} 
\end{align}
The choice~\eqref{eq:zetaTV} leads as $\delta\to 0$
to a consistent approximation of the graph Total Variation which was
first proposed in~\cite{VogelOman1994}.
The choice $\zeta(t) = t$ corresponds to regularization by Dirichlet energy, which corresponds to (unnormalized) graph Laplacian.
Our implementation is available on bitbucket: \url{https://bitbucket.org/AntoninCham/ms_on_graphs/ } 

\subsection{A synthetic example} \label{ex:synth}
We consider denoising and detecting edges in the signal given by piecewise linear function $u$ on domain $[0,1]^2$, shown on Figure \ref{fig:edgeA}. The signal is sampled at 10,000 points, $X_{10,000}$, and corrupted by Gaussian noise with variance $0.2$. We build the graph using $\eta$ as in \eqref{eq:price_num} with $\sigma=5$ and $\e=0.0225$  and with the maximum number of neighbors $k=8$. We considered three models for denoising and edge detection given by $\zeta$ in \eqref{eq:zetaMS}--\eqref{eq:zetaLap}. Namely on Figure \ref{fig:edgeC}
we display the computed  minimizer of the  graph Mumford-Shah for $\zeta$ given by \eqref{eq:zetaMS} and $\lambda=162$. 
On Figure \ref{fig:edgeD}
we display the minimizer of the approximation of the graph TV functional  for $\zeta$ given by \eqref{eq:zetaTV}  with $\delta=0.001$ and $\lambda=438$. 
On Figure \ref{fig:edgeE}
we show the minimizer of the  functional with dirichlet regularization, corresponding to  $\zeta$ given by \eqref{eq:zetaLap} and $\lambda=248$. 
For comparison, for each of the models we display the result for parameter $\lambda$ which minimizes the $L^1(\mu_{10,000})$ error between the minimizer and the clean signal $u$ restricted to $X_{10,000}$.  
The errors observed for optimal lambdas were $\| u_{MS}-u\|_{L^1} = 0.0258$, $\| u_{MS}-u\|_{L^1} = 0.0297$, and $\| u_{L}-u\|_{L^1} = 0.0392$. We note that the recovery by Mumford-Shah is somewhat better than for graph TV. We think that the main reason is that the graph TV tends to decrease contrast (as is well known in image processing, see page 30 of \cite{CCCNP10})  while the Mumford-Shah does not have this bias. 

We also observe to what extent the minimizers recover the edges of the domains by labeling the graph edges that have a relatively large difference between values at the nodes. These are shown in red on the plots. The critical jump size was set manually for visually the best results for each model. We note that Mumford-Shah and TV give similar results, while the Laplacian smoothing blurs the edges as expected. Taking the difference between the minimizers $u_{MS}-u_{TV}$ shows, on Figure \ref{fig:edgeF} that indeed jumps across the edges are typically larger for the Mumford-Shah minimizers that for Total-Variation regularization.

\begin{figure}[ht!]
\centering
%\scriptsize
%\setlength\figureheight{0.35\textwidth}
%\setlength\figurewidth{0.47\textwidth}
\begin{subfigure}[t]{0.47\textwidth}
\centering
%\scriptsize
\includegraphics[width=0.95\textwidth]{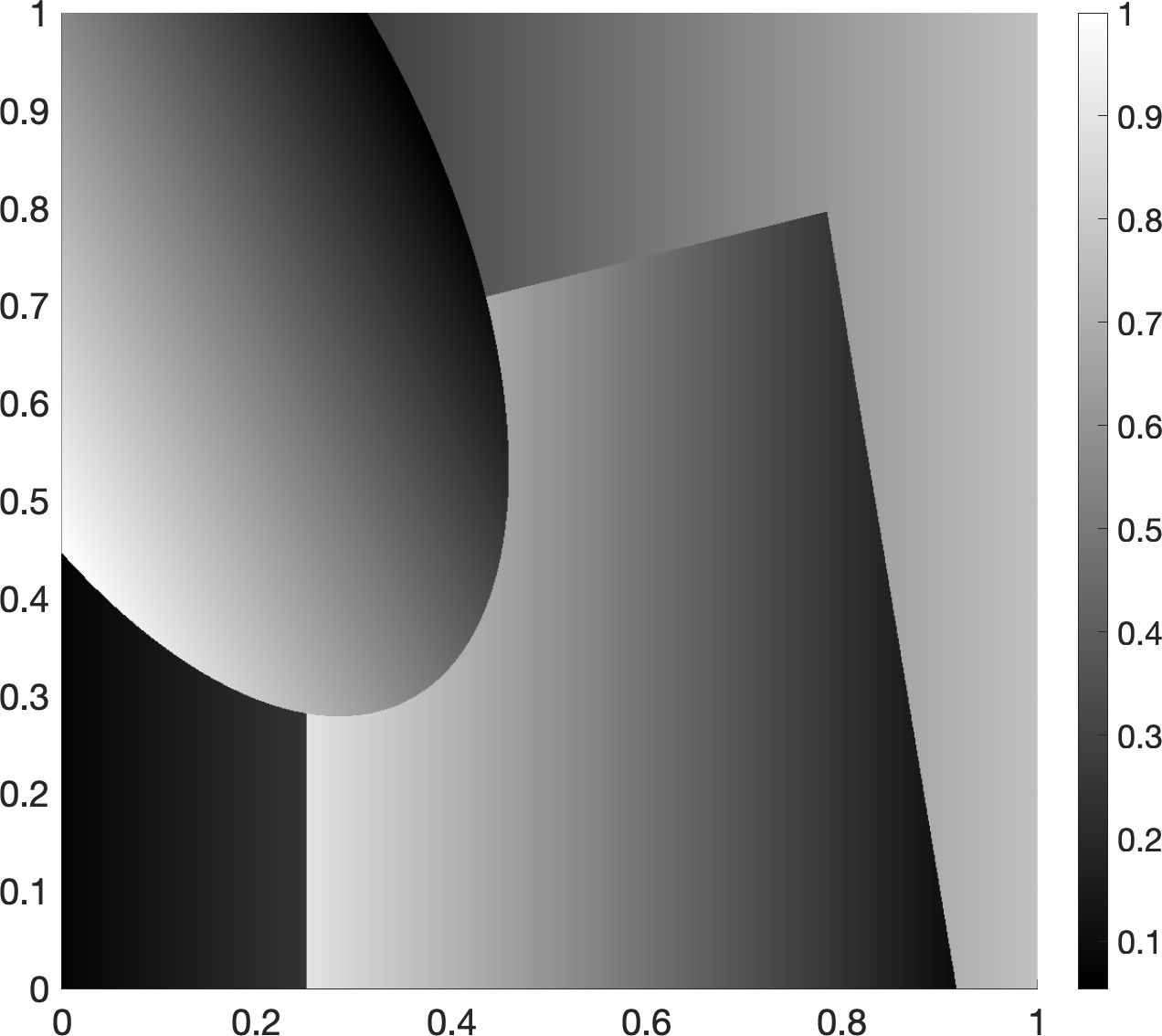}
\caption{Noiseless function.} %; white$=1$, black$=0$.
\label{fig:edgeA}
\end{subfigure}
\hspace*{0.03\textwidth}
\begin{subfigure}[t]{0.47\textwidth}
\centering
%\scriptsize
\includegraphics[width=0.86\textwidth]{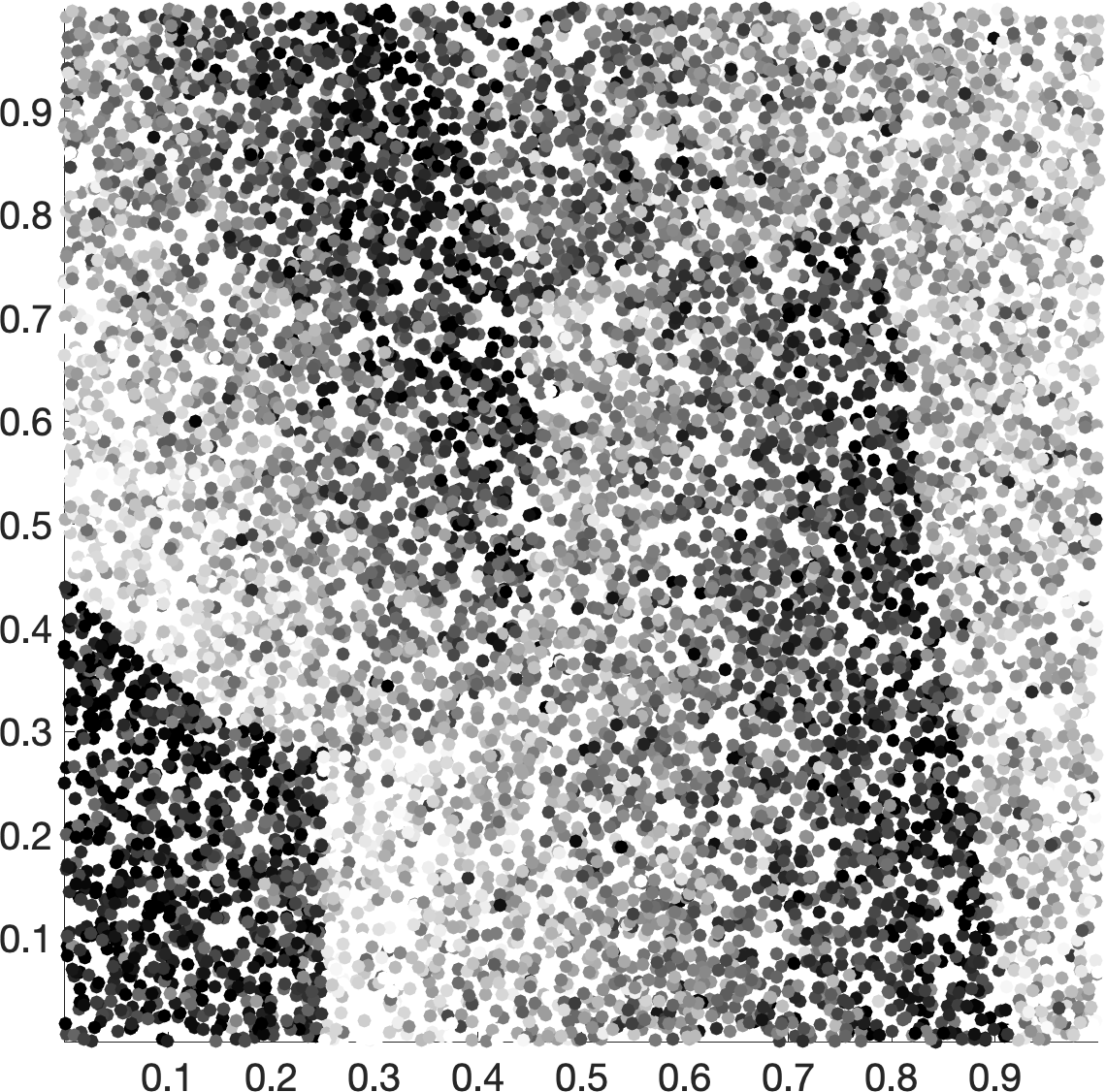}
\caption{
Function sampled at 10,000 random points and corrupted by gaussian noise with $\sigma=0.2$.
%\vspace{2\baselineskip} %\vspace{3pt}
}
\end{subfigure}

%\scriptsize
%\setlength\figureheight{0.35\textwidth}
%\setlength\figurewidth{0.47\textwidth}
\begin{subfigure}[t]{0.47\textwidth}
\centering
%\scriptsize
\includegraphics[width=0.86\textwidth]{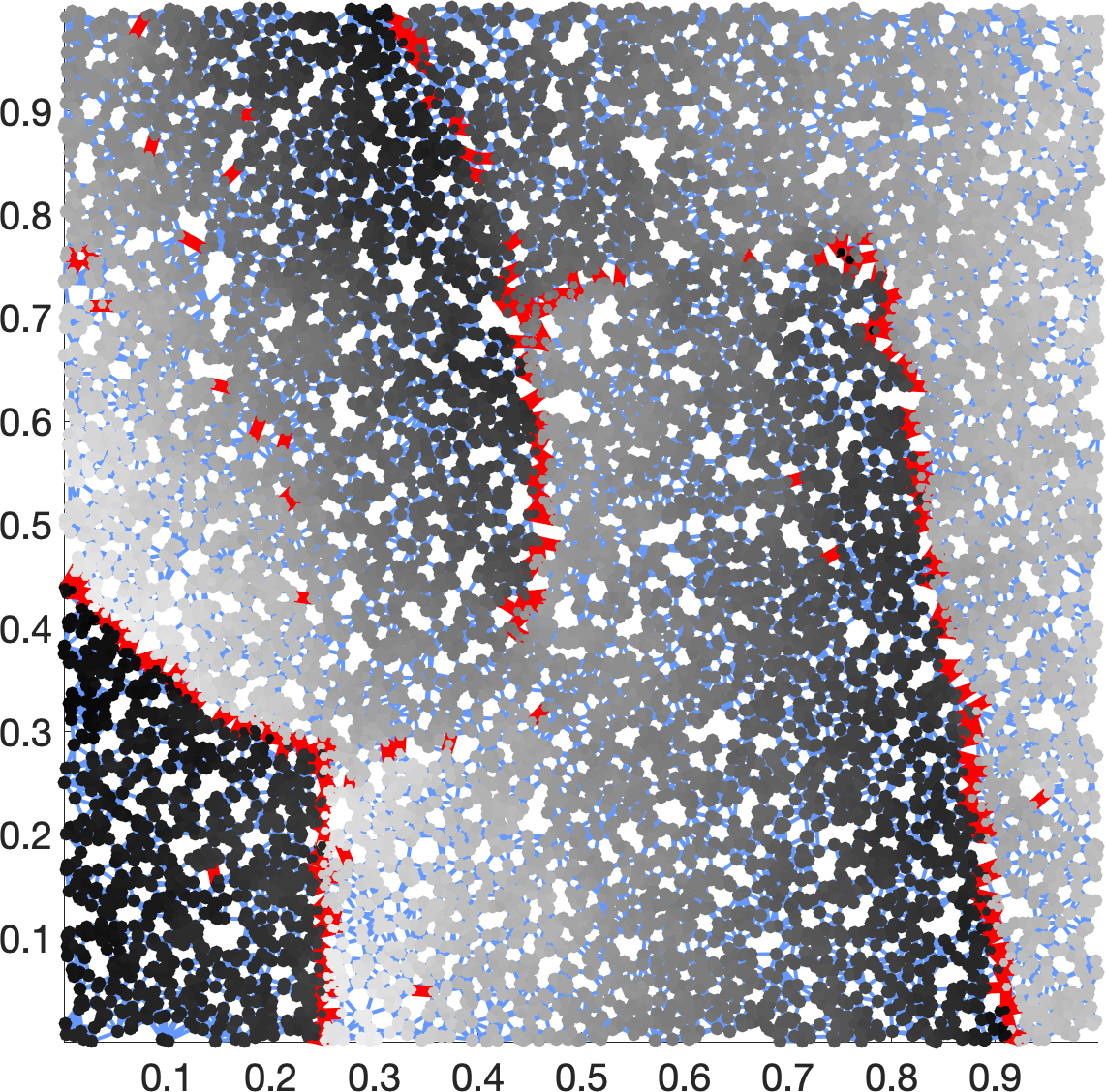}
\caption{Minimizer $u_{MS}$ of \eqref{eq:GMSf}   for $\lambda=162$. Edges with jump over $0.075$ are red. }
\label{fig:edgeC}
\end{subfigure}
\hspace*{0.03\textwidth}
\begin{subfigure}[t]{0.47\textwidth}
\centering
%\scriptsize
\includegraphics[width=0.86\textwidth]{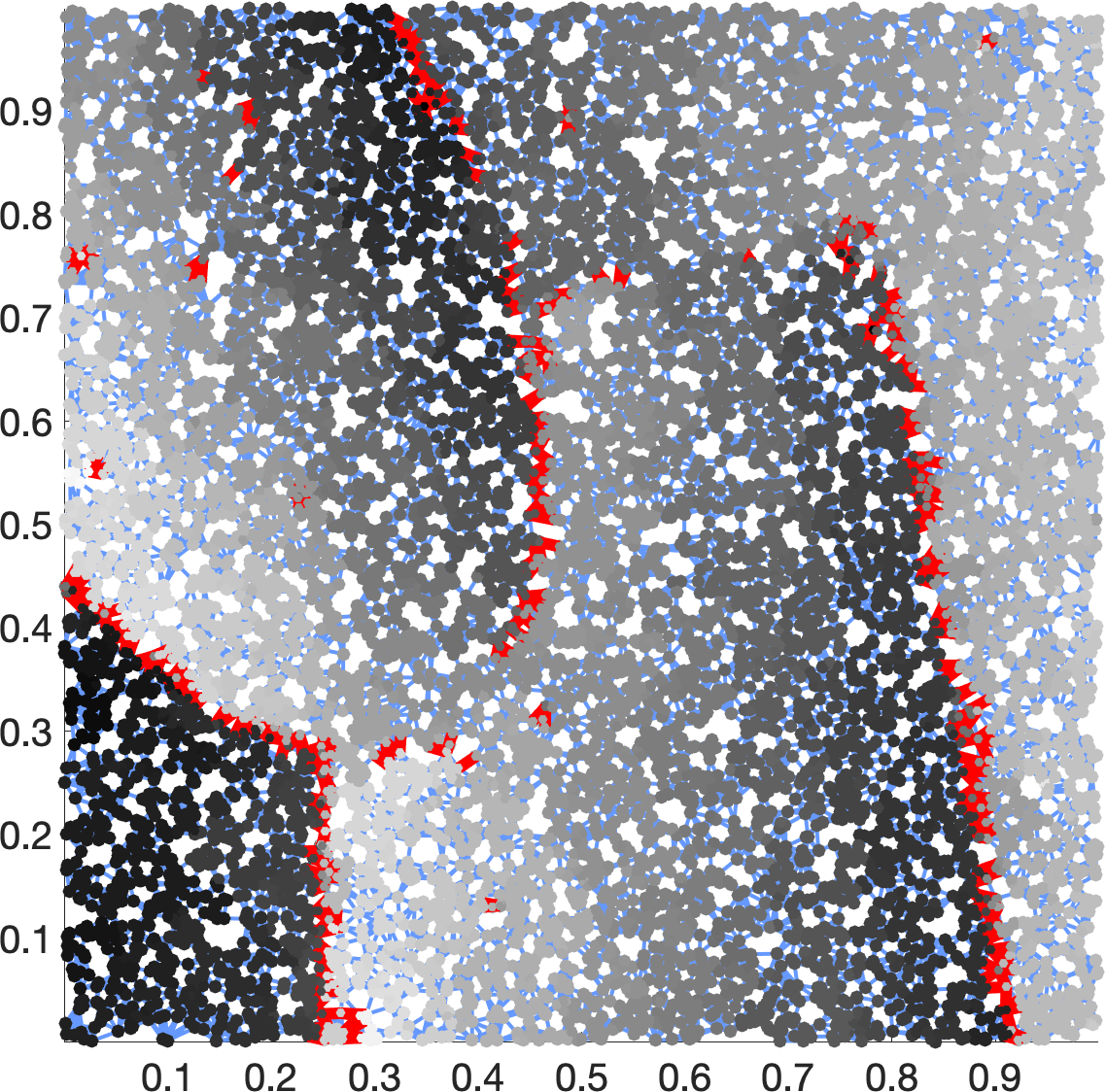}
\caption{Minimizer $u_{TV}$  of the graph TV functional  for $\lambda=438$. Edges with jump over $0.14$ are red. 
%\vspace{2\baselineskip} %\vspace{3pt}
}
\label{fig:edgeD}
\end{subfigure}
%\scriptsize
%\setlength\figureheight{0.35\textwidth}
%\setlength\figurewidth{0.47\textwidth}
\begin{subfigure}[t]{0.47\textwidth}
\centering
%\scriptsize
\includegraphics[width=0.86\textwidth]{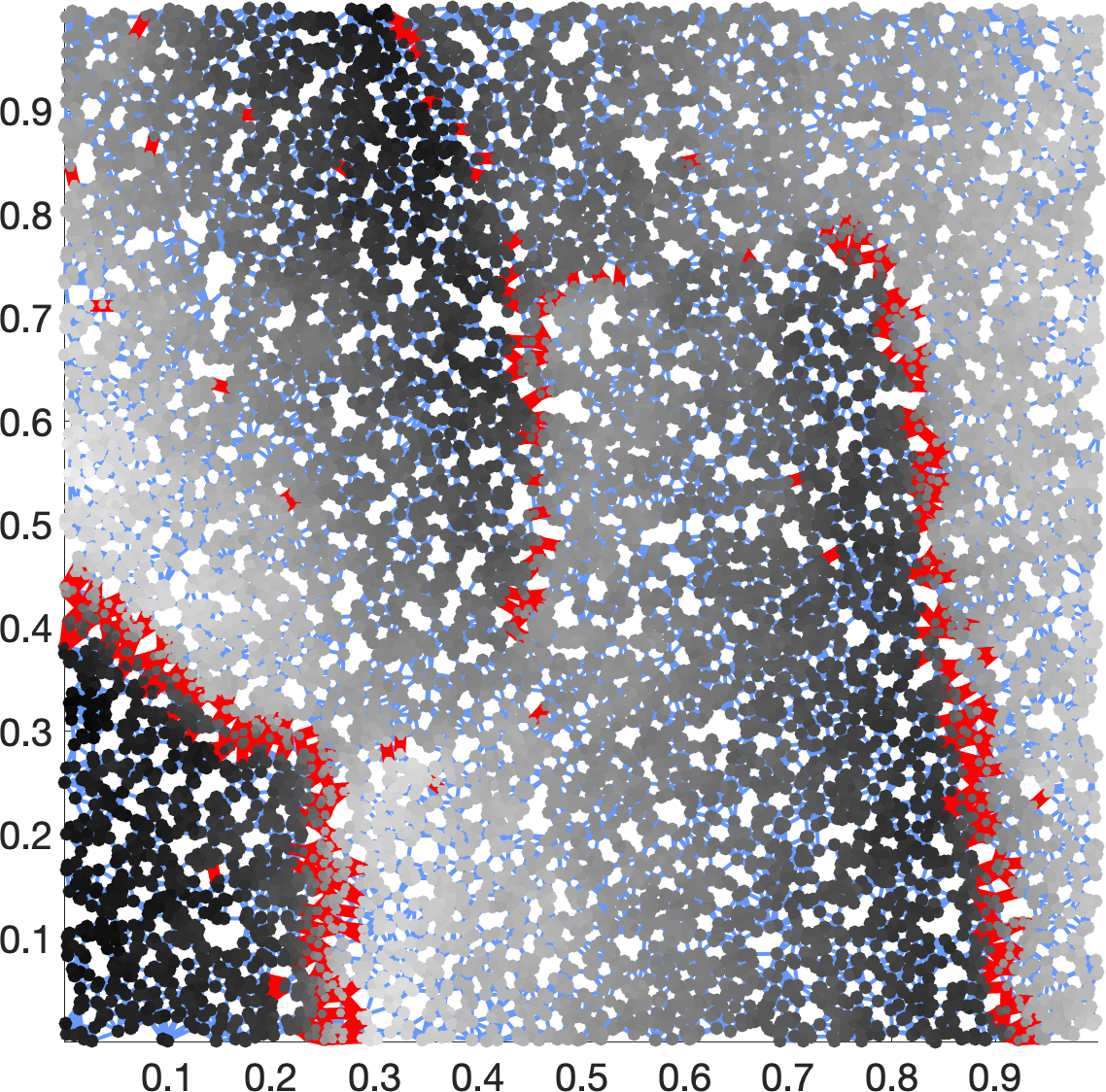}
\caption{Minimizer $u_L$ for $\zeta$ given by \eqref{eq:zetaLap} with $\lambda=248$. Edges with jump over $0.09$ are red. }
\label{fig:edgeE}
\end{subfigure}
\hspace*{0.03\textwidth}
\begin{subfigure}[t]{0.47\textwidth}
\centering
%\scriptsize
\includegraphics[width=0.95\textwidth]{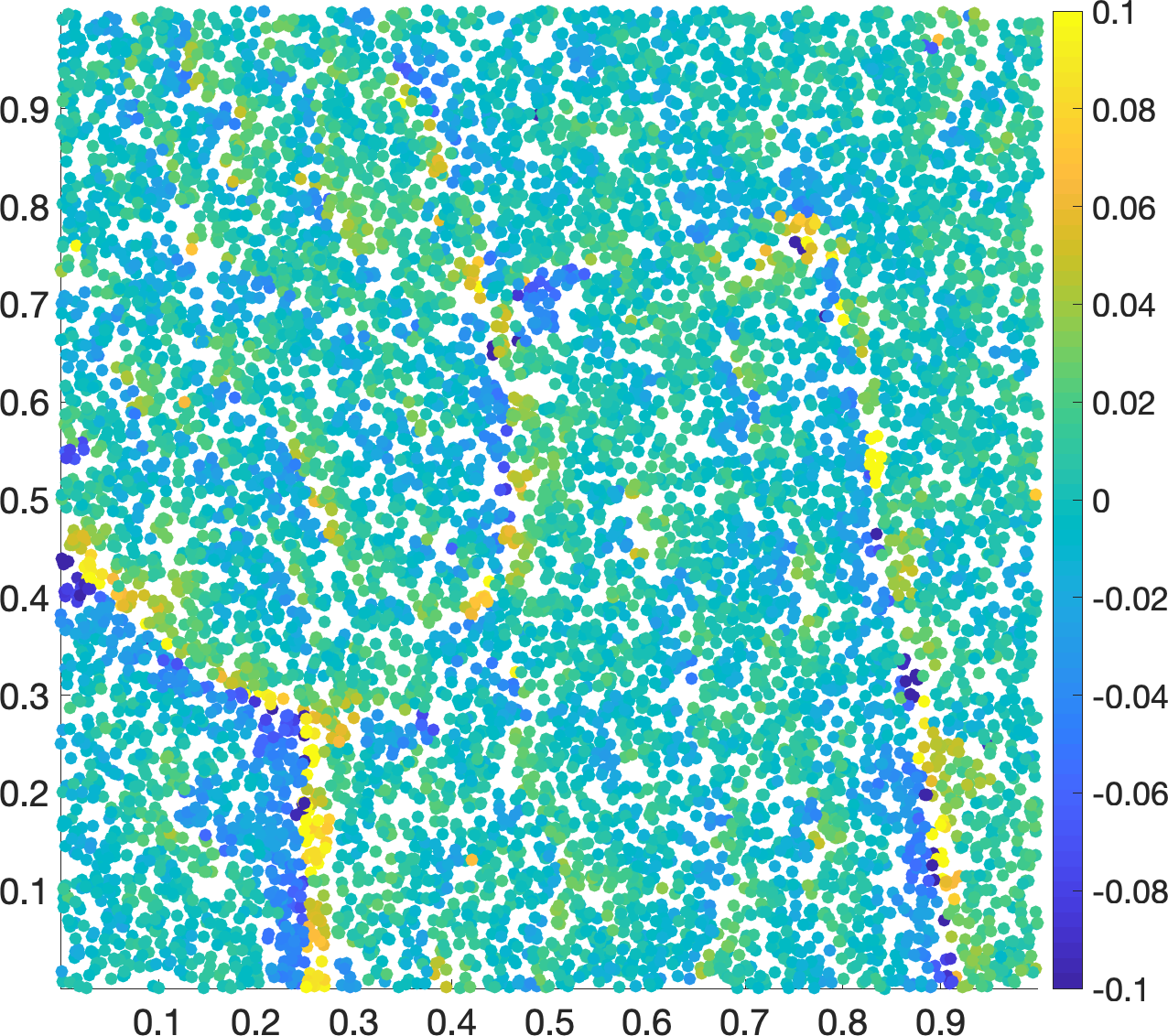}
\caption{
The difference
$u_{MS}-u_{TV}$ tends to be  positive on the upper side of jumps and negative on the lower side.
%\vspace{2\baselineskip} %\vspace{3pt}
}
\label{fig:edgeF}
\end{subfigure}
\caption{
Denoising (regression) and edge detection
% via graph Mumford-Shah functional, graph Total Variation  and for the Laplacian regularizer.
\label{fig:edge}
}
\end{figure}

\subsection{Denoising housing prices}

Here we present an example of minimization of Mumford-Shah functionals on graphs arising from real-world data samples.  This example is given as an illustration of the nature of minimizers.

We consider denoising the real estate prices  in King County, WA. The housing prices in the period May 2014 to May 2015 are obtained  from the Kaggle website:
\url{https://www.kaggle.com/harlfoxem/housesalesprediction}. 

We  removed the geographical outliers (east of longitude $-121.68^o$) and data rows missing square footage. 
The recorded price per square foot is shown on the left. This left 21594 usable records. 
The maximum price per square foot was \$810.14. We normalized the input prices per square foot by dividing by the maximal price. On Figure \ref{fig:housing2} we present the computed minimizers of
the graph Mumford--Shah functional with $\e = 0.04$, $\lambda=14$, and $\sigma=1$. We also allow one to limit the maximal degree of a vertex considered, which we set to $k=15$. 
On a 2018 Macbook Pro, the computation takes 31s, including the construction of the graph. 
The denoised data allow one to visualize by how much the typical price per square foot depends on the location. 
\nc

\begin{figure}[ht!]
\centering
%\scriptsize
%\setlength\figureheight{0.35\textwidth}
%\setlength\figurewidth{0.47\textwidth}
\begin{subfigure}[t]{0.47\textwidth}
\centering
%\scriptsize
\includegraphics[width=0.99\textwidth]{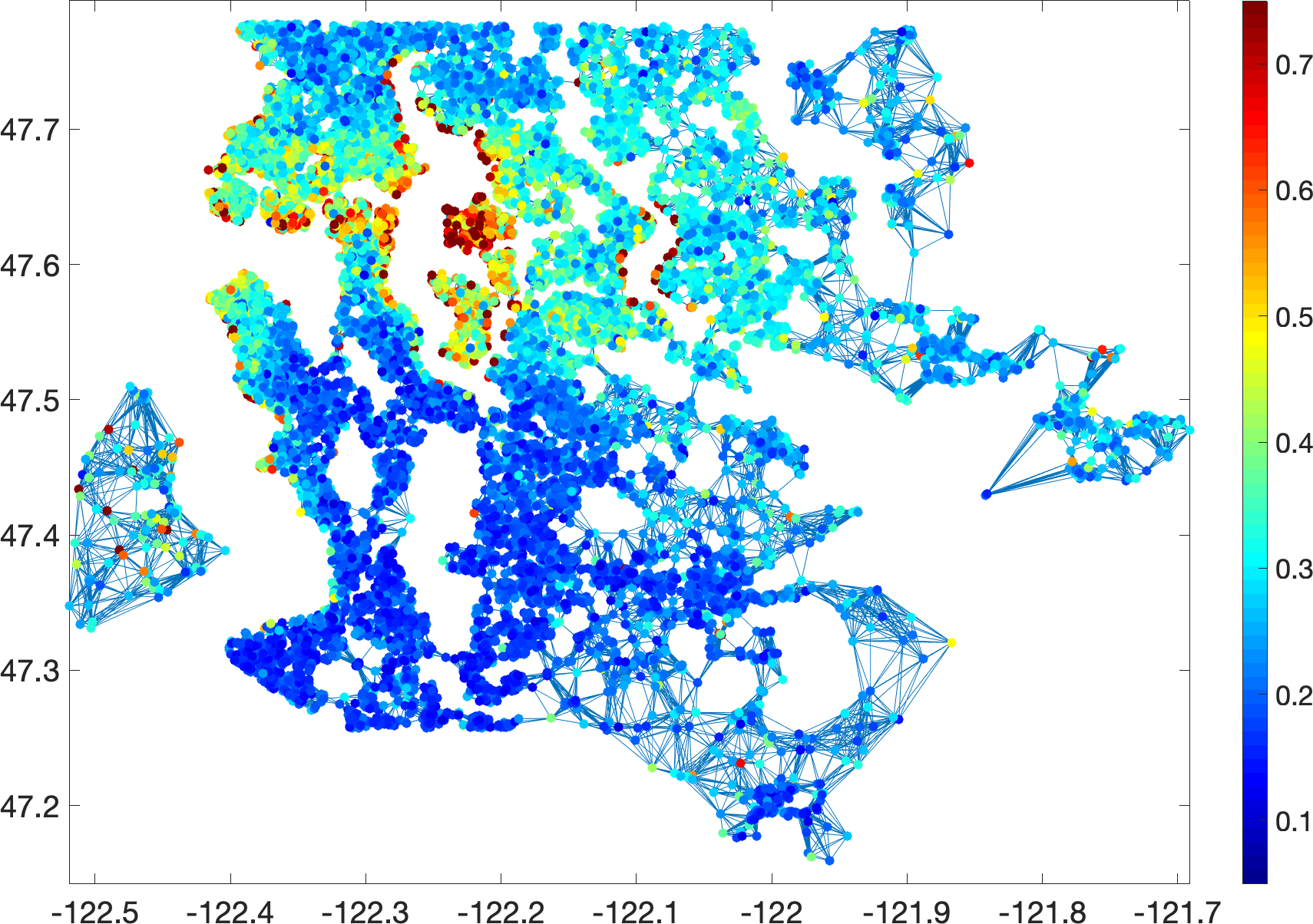}
\caption{The recorded price per square foot.}
\end{subfigure}
\hspace*{0.03\textwidth}
\begin{subfigure}[t]{0.47\textwidth}
\centering
%\scriptsize
\includegraphics[width=0.99\textwidth]{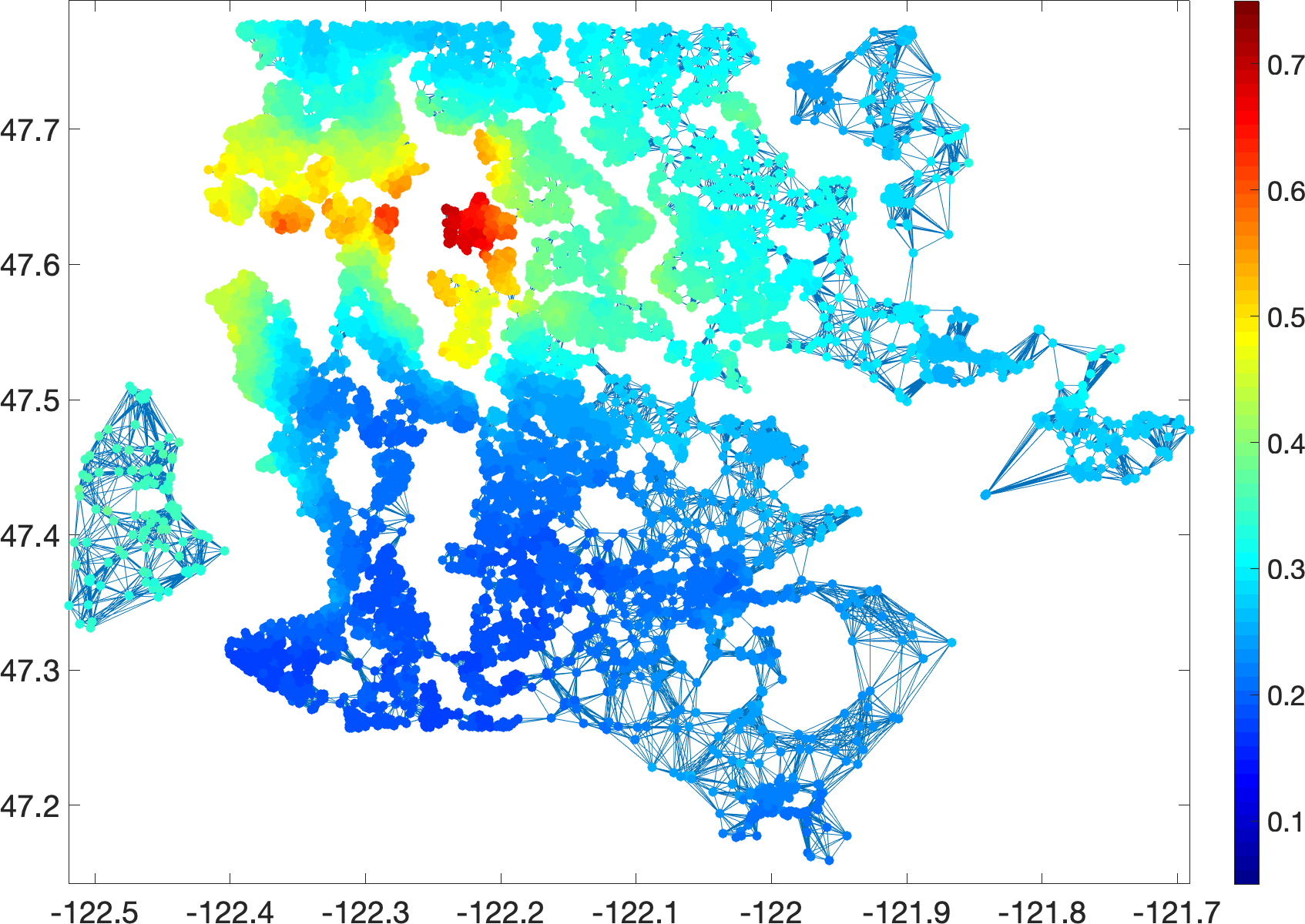}
\caption{
The minimizer of the graph Mumford--Shah functional computed.
%\vspace{2\baselineskip} %\vspace{3pt}
}
\end{subfigure}

%\scriptsize
%\setlength\figureheight{0.35\textwidth}
%\setlength\figurewidth{0.47\textwidth}
\begin{subfigure}[t]{0.47\textwidth}
\centering
%\scriptsize
\includegraphics[width=0.99\textwidth]{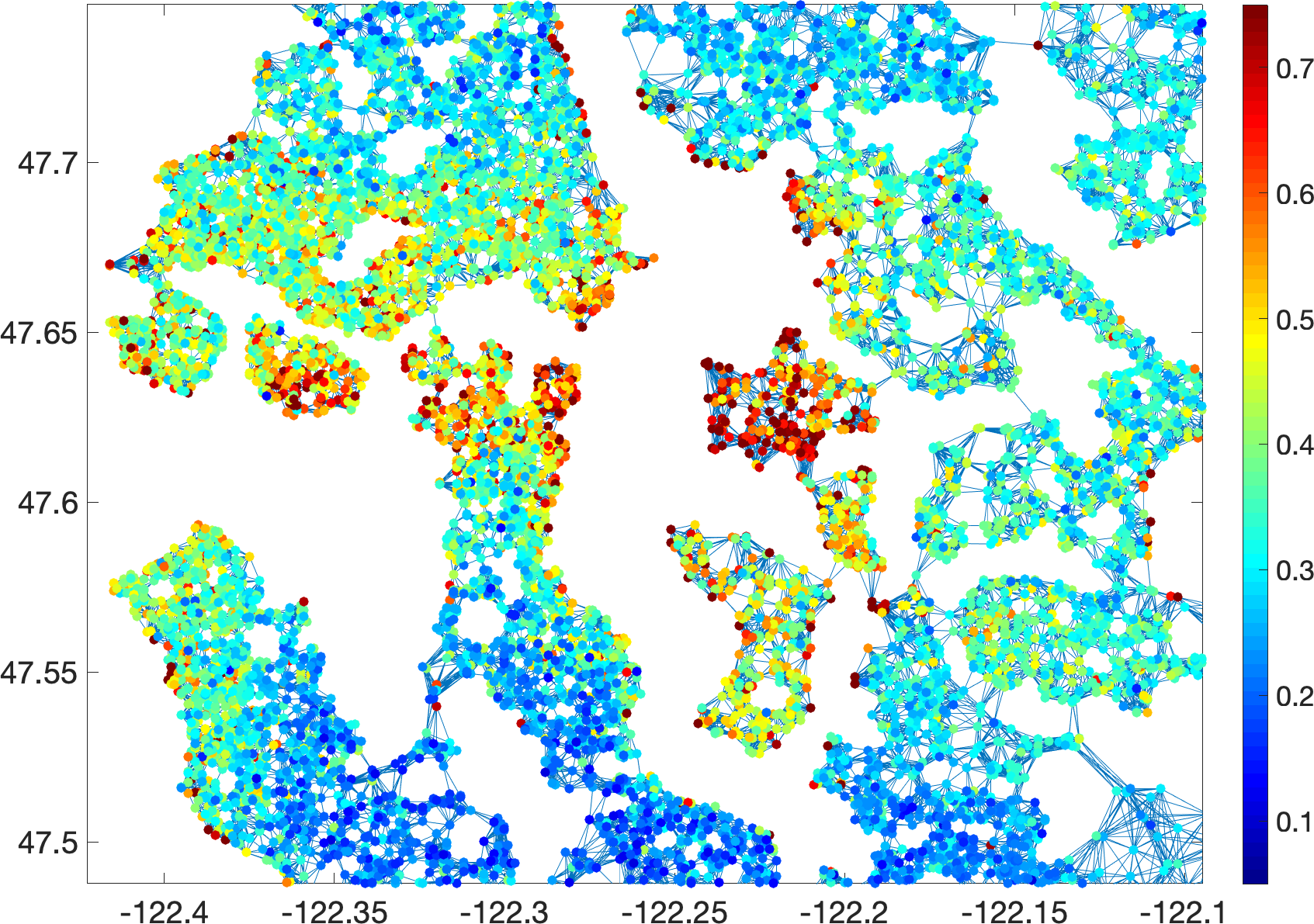}
\caption{Detail of the image above}
\end{subfigure}
\hspace*{0.03\textwidth}
\begin{subfigure}[t]{0.47\textwidth}
\centering
%\scriptsize
\includegraphics[width=0.99\textwidth]{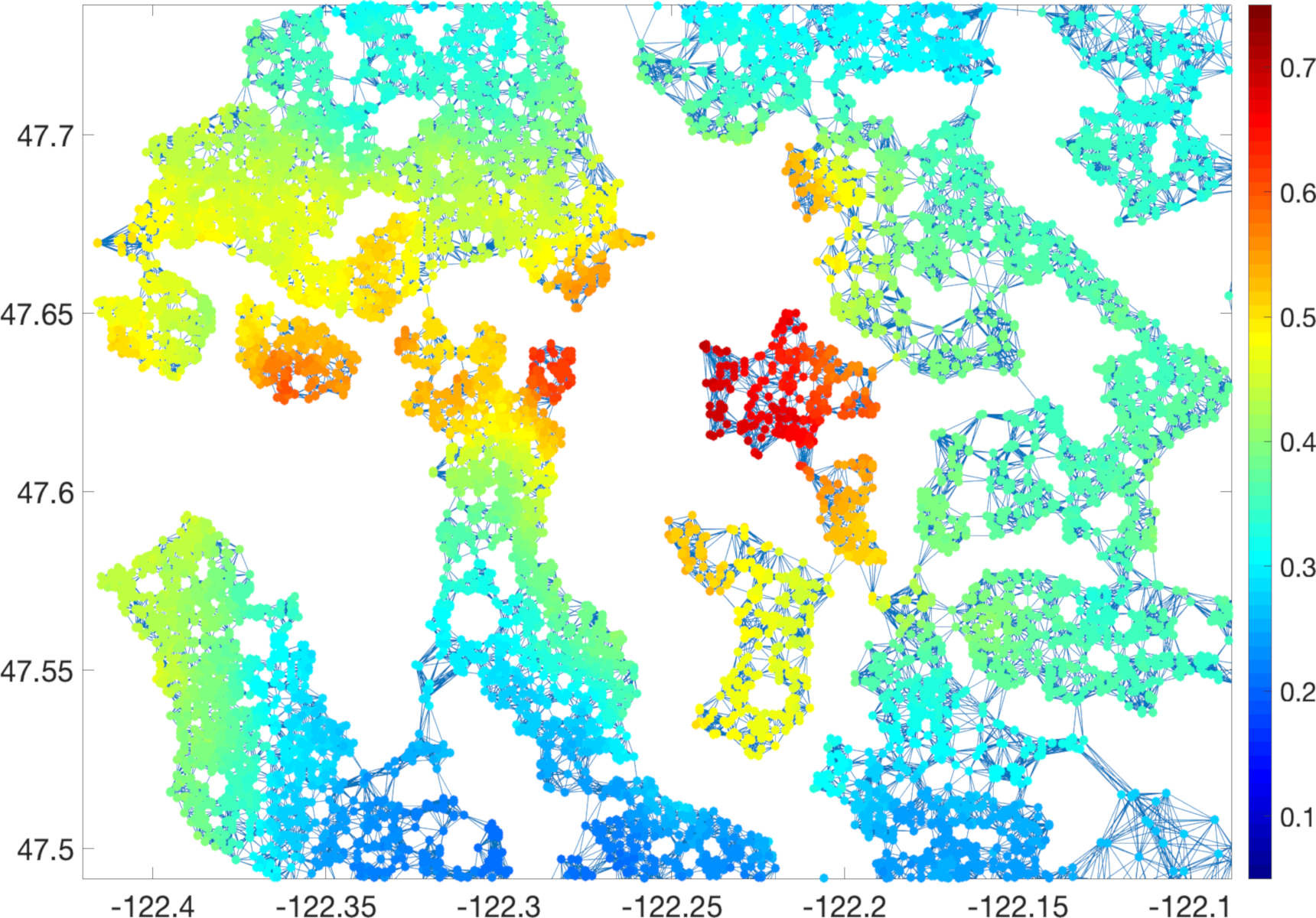}
\caption{
Detail of the image above.
%\vspace{2\baselineskip} %\vspace{3pt}
}
\end{subfigure}

\caption{
Denoising of housing prices. The maximum price per square foot 
\label{fig:housing2}
}
\end{figure}

\subsection*{Acknowledgments.}
The collaboration was supported by European Commission H2020 grant number 777826 (NoMADS).
The work of Caroccia has been partially supported by Fundação para a Ciência e Tecnologia through the Carnegie Mellon--Portugal Program under Grant 18316.1.5004440. Part of the work has been completed under the grant "Calcolo delle variazioni, Equazione alle derivate parziali, Teoria geometrica della misura, Trasporto ottimo" co-founded by Scuola Normale Superiore and the University of Florence. 
Chambolle is grateful to CNRS for support.
Slep\v{c}ev is thankful to the NSF for its support (grants DMS-1516677 and DMS-1814991). He is also grateful to the Center for Nonlinear Analysis of CMU.

%%%%%%%%%%%%%%%%%%%%%%%%%%%%%%%%%%%%%%%%%%%
\appendix 
\section{Proof of Lemma \ref{lem: hope is true}} \label{sec:app}

%\subsection{Proof of Proposition \ref{prop: hope is true}}
The proof of Lemma \ref{lem: hope is true} is obtained as a slight modification of the proof contained in \cite{gobbino2001finite} for $f=1$. In particular we apply \cite[Theorem 3.1, Theorem 3.2]{gobbino2001finite} on the family of functions	
	\[
	\varphi_{\e}(r):=\frac{|\xi|}{\e}\zeta\left(\frac{\e r^p}{|\xi|^{q-p+1}}\right).
	\]
Indeed we note the following facts
	\begin{itemize}
	\item[(i)] $\displaystyle \varphi_{\delta |\xi|}\left(\frac{|u(x+\delta |\xi|)-u(x)|}{\delta |\xi|}\right)=\frac{1}{\delta} \zeta\left(\frac{|u(x+\delta|\xi|)-u(x)|^p}{\delta^{p-1}|\xi|^q}\right)$;
	\item[(ii)] $\displaystyle \lim_{\e\rightarrow 0^+} \varphi_{\e}(r)= \zeta'(0) |\xi|^{p-q} r^p$;
	\item[(iii)]$\displaystyle \lim_{\e\rightarrow 0^+} \e \varphi_{\e}(r/\e)= \Theta|\xi|$;
%	\item[d)] $\displaystyle \varphi_{\e} (x+y) \leq \zeta'(0) \frac{x^p}{|\xi|^{q-p}} + \frac{\Theta|\xi|}{\varepsilon}$ for all $x,y\in \R$. 
	\end{itemize}
In particular, by combining (i)-(iii) and with a slight variation of the proof of  \cite[Theorem 3.1]{gobbino2001finite} (as in \cite[Section 3.2]{chambolle1999finite} ) we conclude that
	\begin{equation}\label{eqn: key limit}
	\liminf_{\delta\rightarrow 0^+} \frac{1}{\delta} \int_{A} \zeta\left(\frac{|u(x+\delta |\xi|)-u(x)|^p}{\delta^{p-1} |\xi|^q}\right) \geq \zeta'(0) |\xi|^{p-q} \int_{A} |u'(x)|^p \d x +\Theta|\xi| \H^{0} (S_u\cap A).
	\end{equation}
We now proceed to the proof of Lemma \ref{lem: hope is true}.
\begin{proof}[Proof of Lemma \ref{lem: hope is true}]
For every $k\in \N$ consider a partition of $A$ in small intervals $\{I_j^k\}_{j=1}^{|A|/k}$. Then
	\begin{align*}
	\frac{1}{\delta} \int_{A} \zeta\left(\frac{|u(x+\delta |\xi|)-u(x)|^p}{\delta^{p-1} |\xi|^q}\right) f(x) \d x &\geq \sum_{j=1}^{|A|/k} \min_{I_j^k}\{f\} \frac{1}{\delta} \int_{A} \zeta\left(\frac{|u(x+\delta |\xi|)-u(x)|^p}{\delta^{p-1} |\xi|^q}\right)  \d x. 
	\end{align*}
In particular, by applying \eqref{eqn: key limit} on each intervals $I_j^k$ we reach
	\begin{align*}
	\liminf_{\delta\rightarrow 0} \frac{1}{\delta} \int_{A} \zeta\left(\frac{|u(x+\delta |\xi|)-u(x)|^p}{\delta^{p-1} |\xi|^q}\right) f(x) \d x &\geq \zeta'(0) |\xi|^{p-q} \sum_{j=1}^{|A|/k}  \int_{I_j^k} |u'(x)|^p\min_{I_j^k}\{f\} \d x  \\
	&  \ \ \ \  +\sum_{j=1}^{|A|/k}\Theta|\xi|\int_{S_u \cap I_j^k} \min_{I_j^k}\{f\}\d \H^{0}(y).
	\end{align*}
Since $f$ is a Lipschitz function we now notice that, given $\e>0$, we can find $\rho$ such that
	\[
	|x-y|<\rho \ \ \Rightarrow  \ \ |f(x)-f(y)|<\e.
	\]
In particular, for any fixed $\e>0$, we can find a $k\in \N$ big enough such that
	\[
	\min_{I_j^k}\{f\}\geq \max_{I_j^k}\{f\}-\e\geq f(x)-\e \ \ \ \ \text{for all $x\in I_j^k$}
	\]
Thus, 
	\begin{align*}
	\liminf_{\delta\rightarrow 0} \frac{1}{\delta} \int_{A} \zeta\left(\frac{|u(x+\delta |\xi|)-u(x)|^p}{\delta^{p-1} |\xi|^q}\right) f(x) \d x &\geq \zeta'(0) |\xi|^{p-q} \sum_{j=1}^{|A|/k}  \int_{I_j^k} |u'(x)|^p\min_{I_j^k}{f(x)} \d x  \\
	& \ \ +\sum_{j=1}^{|A|/k}\Theta|\xi|\int_{S_u \cap I_j^k} \min_{I_j^k}{f(x)} \d \H^{0}(y)\\
	&\geq \zeta'(0) |\xi|^{p-q}   \int_{A} |u'(x)|^p (f(x) - \e) \d x  \\
	&  \ \ +\Theta|\xi|\int_{S_u \cap A} (f(y)-\e)\d \H^{0}(y).
	\end{align*}
Since the above holds for arbitrarily small positive $\e$, we conclude that \eqref{eqn: liminf one d} holds.
\end{proof}

%%%%%%%%%%%%%%%%%%%%
\section{Proof of Lemma \ref{lem: key size lemma}} \label{sec:app2}
\begin{proof}
Note that for all $x\in \Omega\setminus  D(\e_n,\ell_n) $ it holds that
\begin{itemize}
	\item[(i)]	$\{x +s (T_n(x)-x) : s\in [0,1]  \}\cap S_u =\emptyset$.
	\item[(ii)] $\{ x+ \e_n\xi + s(T_n(x+\e_n \xi) - x+\e\xi) : s\in [0,1] \}\cap S_u=\emptyset$;
\end{itemize}
Indeed, assume by contradiction that $x+t_0(T_n(x)-x) \in S_u $. Then $d(x,S_u)\leq t_0\|T_n(x)-x \|\leq t_0 \ell_n$ which would imply $x\in (S_u)_{ \ell_n}$. Thus a) holds. Analogously assume that for some $t_0\in [0,1]$ we have $x+ \e_n\xi + t_0 (T_n(x+\e_n \xi) - x+\e_n\xi)\in S_u$. Then
	\begin{align*}
	d(x,S_u-\e_n\xi)&\leq \| x - (x+ \e_n\xi + t_0 (T_n(x+\e_n \xi) - x+\e_n\xi)-\e_n\xi)\|\\
	&=  t_0  \| (T_n(x+\e_n \xi) - x+\e_n\xi)\|\leq \ell_n
	\end{align*}
again contradicting $x\in \Omega\setminus D(\e_n,\ell_n)$. In particular 
	\begin{align*}
	|u(T_n(x+\e \xi))  - u(T_n(x))|&\leq |u(x+\e \xi) -u(x)|+|u(T_n(x+\e \xi) ) -u(x+\e \xi)| \\
	& +|u(T_n(x))-u(x)|
	\end{align*}
and, since $u$ is regular outside $S_u$,
	\begin{align*}
	|u(T_n(x+\e \xi) ) -u(x+\e \xi)|& \leq \ell_n \int_{0}^{1} |\nabla u ((x+\e \xi)s + (1-s) T_n(x+\e \xi))| \d s\\
	|u(T_n(x))-u(x)|&\leq\ell_n \int_{0}^{1} |\nabla u (xs + (1-s) T_n(x)))| \d s,
	\end{align*}
which is proving \eqref{eqn: comparison}. In order to prove \eqref{eqn: size of the bad points} we just notice that 
	\[
	|(Su-\e_n\xi)_{\ell_n}|=|(Su)_{\ell_n}|
	\]
and that, since $S_u$ is a ployhedral set, for big $n$
	\[
|(Su)_{\ell_n}|=2\ell_n\H^{d-1}(S_u)+o(\ell_n).
	\]
This, combined with \eqref{eqn: decay on epsilon}, implies \eqref{eqn: size of the bad points}.
\end{proof}

%	-------------------	Proof of d):
%\begin{align*}
%	\zeta\left(\frac{\e (x+y)^p}{|\xi|^{q-p+1}}\right)&\leq \zeta'(0) \frac{\e x^p}{|\xi|^{q-p+1}}+\Theta
%		\end{align*}
% since $\zeta<\Theta$ we conclude.
%-------------------

\bibliography{references}
\bibliographystyle{alphaabbr}

\end{document}